\documentclass[reqno,11pt]{amsart}

\usepackage{color}
\usepackage{graphicx, amsmath, amssymb, amscd, amsthm, euscript, psfrag, amsfonts,bm}

\usepackage{hyperref}

\usepackage[latin1]{inputenc}
\DeclareMathAlphabet{\mathpzc}{OT1}{pzc}{m}{it}
\numberwithin{equation}{section}
\theoremstyle{plain}
\newtheorem*{maintheorem*}{Main Theorem}
\newtheorem*{thm*}{Theorem}
\newtheorem*{thma*}{Theorem A}
\newtheorem*{thmaa*}{Theorem A'}
\newtheorem*{thmb*}{Theorem B}
\newtheorem*{thmo*}{Theorem 1.1}
\newtheorem*{thmc*}{Theorem C}
\newtheorem*{thmd*}{Theorem D}
\newtheorem*{thmf*}{Theorem 4.1}
\newtheorem*{remark*}{Remark}

\newtheorem*{conjecture*}{Conjecture}
\newtheorem*{prop*}{Proposition}
\newtheorem*{lem*}{Basic Lemma}
\newtheorem{thm}{Theorem}[section]
\newtheorem{cor}[thm]{Corollary}
\newtheorem{lem}[thm]{Lemma}

\newtheorem{Lem}[thm]{Lemma}
\newtheorem{prop}[thm]{Proposition}

\newtheorem{dfn}[thm]{Definition}

\theoremstyle{definition}

\newtheorem*{proofc*}{Proof of Theorem C}

\newtheorem{definition}[thm]{Definition}

\newtheorem{remark}[thm]{Remark}

\def\bbr{\mathbb{R}}

\def\bbc{\mathbb{C}}

\def\qcal{{\mathcal Q}}
\def\Q{\qcal}\def\D{\mathcal D}
\def\Om{{\Omega}}

\def\ecal{\mathcal{E}}

\def\ccal{\mathcal{C}}

\def\kcal{\mathcal{K}}

\def\lcal{\mathcal{L}}
\def\pcal{\mathcal{P}}

\def\tbf{{\mathbf{t}}}

\def\qpz{\mathpzc{q}}

\def\PSL{{\rm PSL}}
\def\SO{{\rm SO}}

\def\supp{{\rm supp}}
\def\dist{{d}}

\newcommand{\PS}{\operatorname{PS}}
\newcommand{\BR}{\operatorname{BR}}
\newcommand{\Leb}{\operatorname{Leb}}
\newcommand{\BMS}{\operatorname{BMS}}

\newcommand{\LG}{\Lambda(\Gamma)}

\newcommand{\be}{\begin{equation}}
\newcommand{\ee}{\end{equation}}

\def\G{\Gamma}

\def\e{\epsilon}

\def\ba{\backslash}
\def\dU{{\Delta( U)}}
\def\XX{X_1\times X_2}
\def\Xc{\mathcal C_0}
\def\e{\epsilon}
\def\ep{\epsilon}
\def\mups{\mu_x^{\PS}}

\def\fibmeas{\ell}

\def\sfiber{\mu^{\pi_1}}
\def\msec{{\upsilon}}
\def\mfib{{\Upsilon}}
\def\umt{{{\check u}_{\bf r}}}

\def\bigger{{\Delta(U)}}
\def\biggere{{\Delta(u_\tbf)}}
\def\window{{I_{r,T}}}

\def\field{\mathbb{F}}

\newcommand{\T}{\operatorname{T}}
\newcommand{\bH}{\mathbb H}

\renewcommand{\c}{\mathbb{C}}
\newcommand{\br}{\mathbb{R}}\newcommand{\op}{\operatorname}

\newcommand{\pt}{{w}}
\newcommand{{\cont}}{{\check U}}

\begin{document}

\title[Joining]
{Classification of joinings for Kleinian groups}

\author{Amir Mohammadi}
\address{Department of Mathematics, The University of Texas at Austin, Austin, TX 78750}
\email{amir@math.utexas.edu}
\thanks{Mohammadi was supported in part by NSF Grants \#1200388, \#1500677 and \#1128155, and Alfred P.~Sloan Research Fellowship.}

\author{Hee Oh}
\address{Mathematics department, Yale university, New Haven, CT 06511
and Korea Institute for Advanced Study, Seoul, Korea}
\email{hee.oh@yale.edu}

\thanks{Oh was supported by in parts by NSF Grant \#1361673.}

\subjclass[2010] {Primary 11N45, 37F35, 22E40; Secondary 37A17, 20F67}

\keywords{Geometrically finite groups, Joining, Burger-Roblin measure, Bowen-Margulis-Sullivan measure}

\begin{abstract}
  We classify all {\it locally finite}  joinings of a horospherical subgroup action on $\G \ba G$ 
  when $\G$ is a Zariski dense geometrically finite
  subgroup of $G=\PSL_2(\br)$ or $\PSL_2(\c)$. This generalizes Ratner's 1983 joining theorem for the case when $\G$ is a lattice in $G$.
 
  One of the main ingredients is equidistribution of non-closed horospherical orbits with respect to the Burger-Roblin measure
  which we prove in a greater generality where $\G$ is any Zariski dense geometrically finite subgroup of $G=\SO(n,1)^\circ$, $n\ge 2$.
  \end{abstract}

\maketitle
\tableofcontents

\section{Introduction}\label{sec;intro}
M. Ratner obtained in 1983 the classification of joinings
for horocycle flows on a finite volume quotient of $\PSL_2(\br)$ \cite{Rat-Joining}; this precedes her general measure 
classification theorem for unipotent flows on any finite volume homogeneous space of a connected Lie group \cite{Rat-Ann}.
The problem of classifying all
 locally finite invariant measures for unipotent flows on an infinite volume homogeneous space $\G\ba G$ is quite mysterious and even a conjectural picture is not clear at present. However if $G$ is a simple group of rank one, there are classification results for locally finite measures on $\Gamma\ba G$ invariant under a {\it horospherical subgroup} of $G$, when $\G$ is either geometrically finite (\cite{Bu,Roblin2003,Wi}) or a special kind of geometrically infinite subgroups  (\cite{Bab,Led,LedSa,Sar,Sar-ICM}). In this article, our goal is to extend Ratner's joining theorem to an infinite volume homogeneous space $\G\ba G$ where $G=\PSL_2(\field)$ for $\field=\bbr,\bbc$ and $\G$ is a geometrically finite and Zariski dense subgroup of $G.$  This seems to be
  the first measure classification result in homogeneous spaces of infinite volume
for unipotent subgroups which are not horospherical.

Letting $n=2, 3,$ respectively, for $\field=\br,\c$,
 the group $G$ consists of all orientation preserving isometries of the real hyperbolic space $\bH^n$.
Let $U$ be a horospherical subgroup of $G$, i.e.,
for some one-parameter diagonalizable subgroup $A=\{a_s\}$ of $G$,
$$U=\{g\in G: a_{s} g a_{-s} \to e\text{ as $s\to +\infty$}\}$$
 and set $\Delta(U):=\{(u,u): u\in U\}$.
  Let $\G_1$ and $\G_2$ be discrete
subgroups of $G$, and set
 $$Z:=\G_1\ba G\times \G_2\ba G.$$

\begin{definition} Let  $\mu_i$ be a locally finite $U$-invariant Borel measure on $\G_i \ba G$ for $i=1,2$.
A locally finite  $\Delta(U)$-invariant measure $\mu$ on $Z$ is called a {\it $U$-joining}
  with respect to the pair $(\mu_1, \mu_2)$ if  
 the push-forward  $(\pi_i )_*\mu$ is proportional to $\mu_i$ for each $i=1,2$; here $\pi_i$ denotes
  the canonical projection of 
$Z$ to $\G_i\ba G$. If $\mu$ is $\Delta (U)$-ergodic, then we call $\mu$ an ergodic $U$-joining.
\end{definition}

We will classify $U$-joinings with respect to the pair of the {\it Burger-Roblin measures}.
 The reason for this is that for $\G$ geometrically finite  and Zariski dense, 
the Burger-Roblin measure $m^{\BR}_\G$  is the unique locally finite $U$-invariant ergodic measure in $\G\ba G$ 
which is not supported on a closed $U$-orbit (\cite{Bu}, \cite{Roblin2003}, \cite{Wi}).
 Therefore the Burger-Roblin measure for $\G\ba G$, which we will call the BR-measure for simplicity,
plays the role of the Haar measure in Ratner's joining theorem for $\G$ a lattice.


In what follows, we assume that at least one of
 $\G_1$ and $\G_2$ has infinite co-volume in $G$; otherwise, the joining classification was already proved by Ratner. 
Under this assumption, the product measure
$m_{\G_1}^{\BR}\times  m_{\G_2}^{\BR}$ is never a $U$-joining (with respect to the pair
$(m_{\G_1}^{\BR}, m_{\G_2}^{\BR})$), since its projection to  $\G_i\ba G$ is an {\it infinite} multiple of
 $m_{\G_i}^{\BR}$ for at least one of $i=1,2$.
On the other hand, a finite cover self-joining provides an example of $U$-joining. 
Recall that two subgroups of $G$ are said to be commensurable with each other if their intersection has finite index in each of them.

\begin{definition}[Finite cover self-joining]\label{sj} Suppose that for some $g_0\in G$,
$\G_1$ and $g_0^{-1}\G_2 g_0$ are commensurable with each other; this in particular implies that the orbit
$[(e, g_0)] \Delta(G)$ {\footnote {For $S\subset G$,  $\Delta(S)$ denotes the diagonal embedding of $S$ into $G\times G$.}}
 is closed in $Z$. Now using the isomorphism $$(\G_1\cap g_0^{-1}\G_2 g_0)\ba G \to  [(e,g_0)]\Delta(G)$$
given by the map $[g]\mapsto  [(g,g_0g)]$, the push-forward of the BR-measure
 $ m^{\BR}_{\G_1\cap g_0^{-1}\G_2 g_0 }$  to $Z$ gives a $U$-joining, which we call a {\it finite cover self-joining}.
If $\mu$ is a $U$-joining, then any translation of $\mu$ by $(e, u_0)$  is 
also a $U$-joining, whenever $u_0$ belongs to the centralizer of $U$, which is $U$ itself.
 Such a translation of a finite cover self-joining will also be called a finite cover self-joining.
\end{definition}

The following is our main result:
\begin{thm}[Joining Classification] \label{thm;joining-class} Let $\G_1$ and $\G_2$ be geometrically finite
and Zariski dense subgroups of $G$.
Suppose that either $\G_1$ or $\G_2$ is of infinite co-volume in $G$. Then
any locally finite ergodic $U$-joining
on $Z$ is a finite cover self-joining.
\end{thm}

If $\mu$ is a $U$-joining and $\mu=\int \mu_x$ is the $U$-ergodic decomposition, then 
it follows from the $U$-ergodicity of the $\BR$-measure that almost every $\mu_x$ is a $U$-joining.
Therefore the classification of $U$-ergodic joinings gives a complete description for all possible $U$-joinings.


Theorem \ref{thm;joining-class} yields the following by the definition of a finite cover self-joining:
\begin{cor} \label{exj} Let $\G_1$ and $ \G_2$ be as in Theorem \ref{thm;joining-class}.
Then $Z$ admits a $U$-joining measure if and only if $\G_1$ and $\G_2$  are commensurable with each other,
up to a conjugation. \end{cor}
See Remark \ref{exj2} for a slightly more general statement where $\Gamma_2$ is not assumed to be geometrically finite.

 In the course of our classification theorem, we 
obtain equidistribution of a non-closed $U$-orbit $xU$ in $\G\ba G$. 
When $\G$  is a lattice, it is well-known that such $xU$ is equidistributed with respect to the Haar measure
(\cite{DS}, \cite{Rat-Ann}).
For $\G$ geometrically finite, the equidistribution is described by the BR-measure, with the normalization given by the Patterson-Sullivan measure
$\mu_x^{\PS}$ on 
$xU$ (see section \ref{ss;PS-meas} for a precise definition). We call a norm $\|\cdot \|$ on $\field$ {\it algebraic} if  
the $1$-sphere $\{\tbf\in \field: \|\tbf \|=1\}$ is contained in a union of finitely many algebraic varieties.

\begin{thm}\label{it;horo-equid-BR} Let $\G$ be a geometrically finite and Zariski dense subgroup of $G$. Fix $x\in \G\ba G$ 
such that $xU$  is not closed.
Then for any continuous function $\psi$ on $\G\ba G$ with compact support, we have
\[
\lim_{T\to\infty}\frac{1}{\mu_x^{\PS}(B_U(T))}\int_{B_U(T)}\psi(xu)du= m^{\BR}_\G(\psi)
\] where $B_U(T)=\{u \in U: \|u \|<T\}$  is the norm ball in $U\simeq \field$ with respect to an algebraic norm.
\end{thm}
Indeed, we prove this theorem in a greater generality where
$\Gamma$ is a geometrically finite and Zariski dense subgroup of $G=\SO(n,1)^\circ$ for any $n\ge 2$
(see Theorem \ref{t;horo-equid-BR}).

  
Theorem \ref{it;horo-equid-BR}  was proved for $G=\PSL_2(\br)$ in (\cite{Sch-Eq},  \cite{SM}).
One of the difficulties in extending Theorem \ref{it;horo-equid-BR} to a higher dimensional case (even to the case
 $n=3$) is
the lack of a uniform
control of the $\PS$-measure of a small neighborhood of the boundary of $B_U(T)$.
For $n=2$, such a neighborhood has a fixed size independent of $T$, but grows 
with $T\to \infty $ for $n\ge 3$. 

We mention that Theorem \ref{it;horo-equid-BR} applied to the Apollonian group
can be used to describe the distribution of the accumulation of large circles in an unbounded Apollonian circle packing, whereas the papers
\cite{KO} and \cite{OS1} considered the distribution of small circles.
 We refer to Theorem \ref{eqwi} for a window version of Theorem \ref{it;horo-equid-BR} 
which is one of the main ingredients in our proof of Theorem \ref{thm;joining-class}.

\medskip

Similarly to the finite joining case, we can deduce
the classification of $U$-equivariant factor maps
from the classification of joinings (see sub-section \ref{factors}):


\begin{thm} [$U$-factor classification]\label{thm:factor} Let $\G$ be a geometrically finite and Zariski dense subgroup of $G$.
Let $(Y, \nu)$ be a measure space 
with a locally finite $U$-invariant measure $\nu$, and $p: (\G\ba G,m^{\BR}_{\G})\to (Y,\nu)$
be a measurable $U$-equivariant factor map, that is,
$p_*m^{\BR}_\G=\nu$. Then $(Y,\nu)$ is isomorphic to $(\G_0\ba G,m^{\BR}_{\G_0})$ 
for some subgroup $\G_0$ of $G$ which contains $\G$ as a subgroup of finite index.
Moreover, under this isomorphism, the map $p$
can be conjugated to the natural projection $\G\ba G\to\G_0\ba G.$
\end{thm}

We now discuss the proof of Theorem~\ref{thm;joining-class}. Our proof is modeled on Ratner's proof~\cite{Rat-Joining}. 
However, the fact that we are dealing with an infinite measure introduces several 
serious technical difficulties which are dealt with in this paper. Here we briefly describe some of the main steps and difficulties. 

The main ingredient in the proof is the {\it polynomial behavior} of
unipotent flows, which guarantees {\it slow} divergence of orbits of two
nearby points under unipotent subgroups.
This means  that for two nearby generic points $x$ and $y$ in $Z$,
the set of $u\in B_U(T)$  such that the divergence of the two orbits
$x(u,u)$ and $y(u,u)$ is in the intermediate range, that is, roughly of distance
 one apart, has Lebesgue measure comparable to that of  $B_{U}( T ).$
In order to utilize this property in acquiring an additional invariance of a $U$-joining in concern,
we also need to know that this set is
{\it dynamically non-trivial}, i.e., $x(u,u)$ and $y(u,u)$ for these times of $u$ 
stay in a fixed compact subset.
In the case of a finite measure, this can be ensured using the Birkhoff ergodic theorem. 

When $Z$ is of infinite volume, there is no a priori reason for this to hold, for instance,
 the set of $u\in B_U(T)$
where $x(u,u)$ returns to a compact subset can be ``concentrated" near the center, along some sequence of $T$'s.
 This is one of the main issues 
in the way of obtaining  rigidity type results for general locally finite measures (cf.~\cite{MO}).

For a $U$-joining measure $\mu$ on $Z$, we first prove a window-type equidistribution result
for the $U$-action on each space $\G_i\ba G$ with respect to the $\BR$-measure. This enables us
to establish the required non-concentration property for $\mu$-generic points,
 using the fact that $\mu$ projects to the $\BR$-measure

Based on this non-concentration property, we can use
the construction of a polynomial like map in section \ref{ss;quasi-reg} and 
the fact that an infinite joining measure cannot be invariant
under a non-trivial subgroup of $\{e\}\times U$ (Lemma \ref{nf}) to draw two important corollaries: 
\begin{enumerate}
\item  almost all fibers of each projection of $\mu$ are finite;
\item  $\mu$ is quasi-invariant under the diagonal embedding of a connected
subgroup of the normalizer $N_G(U)$. \end{enumerate}
 


In the finite measure case, it is possible to finish using an entropy argument, 
based on the invariance by the action of a diagonalizable subgroup. 
Such an argument  using entropy is not understood in the infinite measure case. We adapt an original argument of
Ratner in~\cite{Rat-Joining} (also used by Flaminio and Spatzier in \cite{FS-Factor} for convex cocompact groups), which avoids a ``direct" use of entropy.
 This involves a step 
 of showing that  a measurable set-valued $N_G(U)$-equivariant factor map $\Upsilon:\G_1\ba G\to \G_2\ba G$
is also equivariant under the action of the opposite horospherical subgroup $\check U$.
We use a close relationship between the $\BR$ and $\BMS$ measures and show 
that essentially the same proof as in~\cite{FS-Factor} works here, again, modulo the extra technical difficulties caused by the presence of cusps. Roughly speaking, one constructs two nearby points $x$ and 
$y=x\umt$ so that 
the $U$-orbits of $\Upsilon(x )\umt$ and $\Upsilon(y)$ do not diverge on average. The fact that the divergence of these two orbits is governed by a polynomial map then implies that the two orbits ``do not diverge". One then concludes that the map commutes with the action of $\check U$ and completes the proof.


\medskip
\noindent{\bf Notation}
In the whole paper, we use the following standard notation. We write $f(t)\sim g(t)$ as $t\to \infty$
to mean $\lim_{t\to \infty} f(t)/g(t)= 1$. We use the Landau  notations $f(t) = O(g(t))$ and $f\ll g$
 synonymously to mean that there exists an implied constant $C>1$ such that $f(t) \le C\cdot g(t)$ for all $t>1$.
 We write $f(t)\asymp g(t)$ if $f(t)=O(g(t))$ and $g(t)=O(f(t))$.
For a space $X$, $C(X)$ (resp.\ $C^\infty(X)$) denotes the set of all continuous (resp.\ smooth) functions on $X.$
 We also let $C_c(X)$ and $C_c^\infty(X)$ denote the set of functions in $C(X)$ 
and $C_c^\infty(X)$ respectively which are compactly supported. 
 For a compact subset $\Om$ of $X$,
 we denote by $C(\Om)$ and $C^\infty(\Om)$ the set of functions in $C_c(X)$ and $C_c^\infty(X)$
respectively which vanish outside of $\Om$. For a subset $B$ in $X$,
$\partial (B)$ denotes the topological boundary of $B$ with the exception that $\partial(\bH^n)$ means
the geometric boundary of the hyperbolic $n$-space $\bH^n$. 
For a function $f$ on $\G\ba G$ and $g\in G$,
 the notation $g.f$ means the function on $\G\ba G$ defined by $g.f(x)=f(xg)$.

Given a subset $S\subset G$,
we let $\Delta(S)=\{(s,s): s\in S\}$. Also given a group $H$ and a subset $S\subset H$ we set
$N_H(S):=\{h\in H: hSh^{-1}=S\}$, i.e., the normalizer of $S$ in $H$

\medskip

\noindent{\bf Acknowledgment}
We would like to thank M.~Einsiedler, Y.~Minsky and A.~Reid for useful discussions.

\section{Basic properties of PS-measure}\label{s;def-pre}
For sections 2--5, let $n\ge 2$ and $G$ be the identity component of the special orthogonal group $\SO(n,1)$. Then $G$ 
 can be considered as the group of orientation preserving isometries
of the hyperbolic space $\bH^n$.
 Let
$\G< G$ be a  discrete subgroup.  Let $\Lambda(\G)\subset \partial(\bH^n)$ denote the limit set of $\G$, and $\delta$ the critical exponent of $\G$.
The convex core of $\G$ is the quotient by $\G$ of the smallest convex subset in $\bH^n$ which contains all geodesics
connecting points of $\LG$. The group $\G$ is called {\it geometrically finite} if a unit neighborhood of the convex core
of $\G$ has finite volume. Throughout the paper, we assume that 
$$\text{$\G$ is geometrically finite and Zariski dense.}$$

\subsection{Conformal densities}  A family of finite measures
$\{\mu_x:x\in \bH^n\}$ on $\partial(\bH^n)$ is called  a {\em $\G$-invariant conformal
density\/} of dimension $\delta_\mu > 0$  if  for any $x,y\in \bH^n$, $\xi\in \partial(\bH^n)$ and
$\gamma\in \G$,
$$\gamma_*\mu_x=\mu_{\gamma x}\quad\text{and}\quad
 \frac{d\mu_y}{d\mu_x}(\xi)=e^{-\delta_\mu \beta_{\xi} (y,x)}, $$
where $\gamma_*\mu_x(F)=\mu_x(\gamma^{-1}(F))$ for any Borel
subset $F$ of $\partial(\bH^n)$. Here
$\beta_{\xi} (y,x)$ denotes the Busemann function:
$\beta_{\xi}(y,x)=\lim_{t\to\infty} d(\xi_t, y) -d(\xi_t, x)$ where
$\xi_t$ is a geodesic ray tending to $\xi$ as $t\to \infty$.


 We denote by $\{\nu_x\}$ the
{\em Patterson-Sullivan density}, i.e.,
a  $\G$-invariant conformal
density
 of dimension $\delta$, which is
 unique up to a scalar multiple.
 
For each $x\in \bH^n$,  we set $m_x$ to be the unique probability measure on $\partial(\bH^n)$ which is invariant under
the compact subgroup $\op{Stab}_G(x)$. Then $\{m_x:x\in \bH^n\}$ forms a $G$-invariant conformal density of dimension $(n-1)$, which will be referred
to as  the {\em Lebesgue density}.

Denote by $\{\mathcal G^s:s\in \br\}$ the geodesic flow on the unit tangent bundle $\T^1(\bH^n)$ of $\bH^n$.
For $\pt\in \T^1(\bH^n)$, we denote by $\pt^{\pm}\in \partial(\bH^n)$ the forward and the backward endpoints of the geodesic determined by
$\pt$, i.e., $\pt^{\pm}=\lim_{s\to \pm \infty}\mathcal G^s(\pt)$.

Fix $o\in \bH^n$ and $w_o\in \T^\chi_o(\bH^n)$. Let $K:=\op{Stab}_G(o)$ and $M:=\op{Stab}_G(w_o)$ so that
$\bH^n$ and $\T^1(\bH^n)$ can be identified with $G/K$ and $G/M$ respectively.
Let $A=\{a_s: s\in \br\}$ be the one-parameter subgroup of diagonalizable elements in the centralizer
of $M$ in $G$ such that  $\mathcal G^s(w_o)=[M]a_s=[a_sM]$.
For $g\in G$, define $$g^{\pm}:=gw_o^{\pm} \in \partial(\bH^n).$$
The map $\op{Viz}:G\to \partial (\bH^n)$ given by $g\mapsto g^+$ (resp. $g\mapsto g^-$) will be called the visual map (resp. the backward
visual map).

\subsection{PS measure on $gU\subset G$}\label{ss;PS-meas}
Let $U$ denote the expanding horospherical subgroup, i.e.,
$$U =\{g\in G: a_s ga_{-s}\to e\text{ as $s\to + \infty$}\}.$$ 
The group $U$ is a connected abelian group, isomorphic to $\br^{n-1}$; we use the parametrization
 $\tbf\mapsto u_\tbf$  so that for any $s\in \br$,
$$a_{-s}u_\tbf a_s=u_{e^s\tbf}.$$

For any $g\in G$, the restriction $\op{Viz}|_{gU}$ is 
 a diffeomorphism between $gU$ and $\partial(\bH^n)-\{g^-\}$.
Using this diffeomorphism, we can define measures on $gU$ which are equivalent to conformal densities on the boundary:
$$d \mu_{gU}^{\Leb} (gu_\tbf)= e^{(n-1) \beta_{(gu_\tbf)^+}(o,gu_\tbf(o))} dm_o(gu_\tbf)^+;$$
$$d \mu_{gU}^{\PS} (gu_\tbf )= e^{\delta \beta_{(gu_{\tbf})^+}(o,gu_{\tbf}(o))} d{\nu_o}(gu_{\tbf})^+.$$
The conformal properties of $\{m_x\}$ and $\{\nu_x\}$ imply that these definitions are independent of the choice of $o\in \bH^n$.
The measure $d \mu_{gU}^{\Leb} (gu_\tbf)$ is independent of the orbit $gU$:
$$d \mu_{gU}^{\Leb} (gu_\tbf)= d \mu_{U}^{\Leb} (u_\tbf)=d\tbf$$
 and is simply the Lebesgue measure on $U=\br^{n-1}$, up to a scalar multiple. 
We call the  measure $d \mu_{gU}^{\PS}$ the Patterson-Sullivan measure (or simply PS-measure) on $gU$. For simplicity,
we write
$$d\mu^{\Leb}_g(\tbf)=d \mu_{gU}^{\Leb} (gu_\tbf )\quad\text{ and }\quad d\mu^{\PS}_g(\tbf)=d \mu_{gU}^{\PS} (gu_\tbf ) $$
although these measures depend on the orbits but not on the individual points.
 
These expressions are also useful as we will sometimes consider $\mu_g^{\PS}$ as  a measure on $U$ in an obvious way, for instance, in the following lemma.
Let $\mathcal M(U)$ be the set of all regular Borel measures 
on $U$ endowed with the following topology:
$\mu_n\to \mu$ if and only if $\int f d\mu_n \to \int f d\mu$ for all $f\in C_c(U)$.

The following is proved  in \cite{FS-Factor} for $\G$ convex co-compact but the proof works for $\G$ geometrically finite.
\begin{lem} \cite[Lem.~4.1, Cor.~4.2]{FS-Factor}.\label{pscon}
\label{cont}
\begin{itemize}
\item For $g\in G$, the measure $\mu_g^{\PS}$ assigns $0$ measure to any proper sub-variety of $U$.
\item The map  $g \mapsto \mu_g^{\PS}$
is a continuous map from $\{g\in G: g^+\in \Lambda(\G)\}$ to $ \mathcal M(U)$. \end{itemize}
\end{lem}

The following is an immediate consequence:
\begin{cor}\label{fxp} For a compact subset $\Omega\subset G$ and any $s>0$,
\begin{equation*}
0<\inf_{g\in \Omega, g^+\in \Lambda(\G)}\mu_g^{\PS}( gB_U(s)) 
\le \sup_{g\in \Omega, g^+\in \Lambda(\G)}\mu_g^{\PS}( gB_U(s)) <\infty .\end{equation*}
\end{cor}

For simplicity, we set $|\tbf|$ to be the maximum norm of $\tbf$. For $T>0$, we set
$$B_U(T):=\{u_\tbf\in U: |\tbf|\le T\} .$$
For any $s\in \br$, we have:
 \begin{align*} \label{e;leafwise-meas}
\mu^{\PS}_g(B_U(e^s))&=e^{\delta s}\mu_{ga_{-s}}^{\PS} (B_U(1));
\\
 \mu^{\Leb}_g(B_U(e^s))&=e^{(n-1)s}\mu_{ga_{-s}}^{\Leb}(B_U(1)).
\end{align*}

\subsection{PS measure on $xU\subset \Gamma\ba G$}\label{ss;PS-meas2}
Set $$X=\G\ba G .$$For $x\in X$, we define the measure $\mu_x^{\PS}$ on the orbit $xU$ as follows.
Choose $g\in G$ such that $x=[g]$. 
If the map $u\mapsto xu$ is injective, $\mu_x^{\PS}$ will be simply 
the push-forward of the measure $\mu_g^{\PS}$ on $gU$.
In general, we first define a measure $\bar \mu_x^{\PS}$ on $(U\cap g^{-1}\G g)\ba U$ as follows:
for $f\in C_c(U)$, let $\bar f \in C_c ((U\cap g^{-1}\G g)\ba U)$ be given by
$\bar f([u])=\sum_{\gamma_0 \in U\cap g^{-1}\G g } f (\gamma_0 u)$. Then the map $f\mapsto \bar f$ is a surjective map
from $C_c(U)$ to $C_c((U\cap g^{-1}\G g)\ba U)$, and
$$\int_{[u]\in U\cap g^{-1}\G g\ba U} \bar f ([u]) \bar \mu_x^{\PS}[u] :=\int_U f(u_\tbf) d\mu_g^{\PS}(\tbf)$$
is well-defined,  by the $\G$-invariance of the PS density.
This defines a locally finite measure on $(U\cap g^{-1}\G g)\ba U$ by \cite[Ch.1]{Ragh}.
Noting that the map $(U\cap g^{-1}\G g)\ba U\to xU$ given by $[u]\to xu$ is injective, 
we denote by $\mu_x^{\PS}$ the push-forward of the measure $\bar\mu_x^{\PS}$ from
$(U\cap g^{-1}\G g)\ba U$ (cf. \cite{OS}).

The map $\mu_x^{\Leb}$ is defined similarly. We caution that $\mu_x^{\PS}$ is not a locally finite measure on $\G\ba G$ unless $xU$ is a closed subset
of $X$.


 A limit point $\xi\in \Lambda(\G)$ is called {\it radial} if
some (and hence every) geodesic ray toward $\xi$ has accumulation points in a compact subset of $\G\ba G$, and
{\it parabolic} if it is fixed by a parabolic element of $\G$
 (recall that an element $g\in G$ is parabolic if the set of fixed points of $g$ in
$\partial(\bH^n)$ is a singleton.) A parabolic limit point $\xi\in \Lambda(\G)$ is called {\it bounded} if
the stabilizer $\G_{\xi}$ acts co-compactly on $\Lambda(\G)-\{\xi\}$.
 We denote by $\Lambda_{\rm r}(\G)$ and $\Lambda_{\rm bp}(\G)$ 
the set of all radial limit points and the set of all bounded parabolic limit points respectively.
As $\G$ is geometrically finite, we have (see \cite{Bo})
 $$\Lambda(\G)=\Lambda_{\rm r}(\G)\cup \Lambda_{\rm bp}(\G).$$

\begin{dfn}\rm {For $x=[ g]\in X$, we write
$x\in \LG$ and $x^-\in \Lambda_{\rm r}(\G)$ if $g^-\in \LG$ and
$g^-\in \Lambda_{\rm r}(\G)$ respectively; this is well-defined independent of the choice of $g$.} \end{dfn}
If $x^-\in \Lambda_{\rm r}(\G)$, then the map $u\mapsto xu$ is injective on $U$. 

\begin{Lem}\label{rinf}
For $x\in X$,
we have $x^-\in \Lambda_{\rm r}(\G)$ if and only if $|\mu_x^{\PS}|= \infty$.
   \end{Lem}
\begin{proof} Choose $g\in G$ so that $[g]=x$. 
  If $x^-\notin \Lambda(\G)$, then  $\Lambda(\G)$ is a compact subset of $\partial(\bH^n) \setminus \{g^-\}$, 
  and hence  $ |\mu_g^{\PS}|<\infty$.  
  If $g\in \Lambda_{\rm bp}(\Gamma)$ and  $xU$ is a closed subset of $X$ and
  $\mu_x^{\PS}$ is supported on the quotient of $\{gu\in gU: (gu)^+\in \Lambda(\G)\}$ by the stabilizer
  $\G_{x^-}$, which is compact by the definition of a bounded
  parabolic fixed point. Hence $|\mu_x^{\PS}|<\infty$. 
Suppose that $g^-\in \Lambda_{\rm r}(\G)$. Since $|\mu_{g}^{\PS}|=|\mu_x^{\PS}|$ and $(gU)^+\cap \Lambda(\G) \ne \emptyset$,
we may assume without loss of generality that $g^+\in \Lambda(\G)$.

 Let $\Omega$ be a compact subset of $G$ such that
  $\gamma_i ga_{-s_i}\in \Omega $   for sequences $s_i \to +\infty$ and $\gamma_i\in \G$.
 Then 
\begin{align*} \mu_x^{\PS}(B_U(e^{s_i} ))&= e^{s_i \delta} \mu_{x a_{- s_i}}^{\PS} (B_U(1))\\
&= e^{s_i \delta} \mu_{\gamma_i g a_{- s_i}}^{\PS} (B_U(1))
\\ & \ge e^{s_i \delta} \cdot \inf_{h\in \Omega, h^+\in \Lambda(\G)} \mu_{h}^{\PS} (B_U(1)) .
\end{align*}
Since $\inf_{h\in \Omega, h^+\in \Lambda(\G)} \mu_{h}^{\PS}(B_U(1)) >0$ by Corollary \ref{fxp},
we have $|\mu_x^{\PS}|=\infty$.
\end{proof}

\subsection{BMS and BR measures} 
Fixing $o\in \bH^n$, the map
\[
\pt\mapsto (\pt^+, \pt^-, s=\beta_{\pt^-}(o, \pi(\pt)) 
\]
is a homeomorphism between $\op{T}^1(\bH^n)$ with
 \[
 (\partial(\bH^n)\times \partial(\bH^n) - \{(\xi,\xi):\xi\in \partial(\bH^n)\})  \times \br.
 \]

Using this homeomorphism,
we define measures
$\tilde m^{\BMS}=\tilde m^{\BMS}_\G$ and $ \tilde m^{\BR} =\tilde m^{\BR}_\G$ on $\T^1(\bH^n)$ as follows:

\[
d \tilde m^{\BMS}(\pt)= e^{\delta \beta_{\pt^+}(o,\pi(\pt))}\;
e^{\delta \beta_{\pt^-}(o, \pi(\pt)) }\; d\nu_o(\pt^+) d\nu_o(\pt^-) ds; \text{ and} 
\]
\[
 d \tilde m^{\BR}(\pt)= e^{(n-1) \beta_{\pt^+}(o,\pi(\pt))}\;
 e^{\delta \beta_{\pt^-}(o, \pi(\pt)) }\; dm_o(\pt^+) d\nu_o(\pt^-) ds.
 \]

The conformal properties of $\{\nu_x\}$ and $\{m_x\}$ imply that these definitions are independent of the choice of $o\in \bH^n$.
Using the identification $\T^1(\bH^n)$ with $G/M$, we lift the above measures to $G$ so that they are all invariant under $M$ from the right.
By abuse of notation, we use the same notation $\tilde m^{\BMS}$ and $ \tilde m^{\BR}$.
 These measures are left $\G$-invariant, and hence
induce locally finite Borel measures on $X$, which are
the  Bowen-Margulis-Sullivan measure $m^{\BMS}$ 
and the Burger-Roblin measure $m^{\BR}$ respectively.

Note that 
\begin{itemize}
\item $\supp (m^{\BMS})=\{x\in X: x^{\pm}\in \LG\}$;
\item $\supp(m^{\BR})=\{x\in X: x^-\in \LG\}$.
\end{itemize}

Sullivan showed that $m^{\BMS}$ is a finite measure \cite{Sullivan1984}.
The BR-measure $m^{\BR}$ is an infinite measure unless $\G$ is a lattice \cite{OS}.

We also consider the contracting horospherical subgroup
\[
{\cont}=\{g\in G: a_{-s} g a_s \to e\text{ as $s\to \infty$}\}
\]
and the parabolic subgroup $$P=MA{\cont} .$$
We note that $P$ is precisely the stabilizer of $w_o^+$ in $G$.
Given $g_0\in G$ define the measure $\nu$ on $g_0P$
as follows 
\[
d\nu(g_0p)=e^{-\delta \beta_{(g_0p)^-}(o, g_0p(o))} d\nu_o(g_0pw_o^-) dm ds,
\] 
for $s=\beta_{(g_0p)^-}(o, g_0p(o))$, $p=ma\check u$, and $dm $ is the probability Haar measure of $M$.

This way we get a product structure of the BMS measure which is an important ingredient in our approach:
 for any $g_0\in G$,
\begin{equation} \label{bmsp} \tilde m^{\BMS}(\psi )=\int_{g_0P}\int_{g_0pU} \psi (g_0pu_\tbf) d\mu_{g_0p}^{\PS}(\tbf) d\nu(g_0p) \end{equation}

 \section{Uniformity in the equidistribution of PS-measure}\label{notsec}
We keep notations from the last section; so  $X=\G\ba G$.
 For simplicity, we will assume that $|m^{\BMS}|=1$;
this can be achieved by replacing $\{\nu_x\}$ by a suitable
scalar multiple if necessary.
\begin{thm}[Mixing of the $A$-action]\label{fm}  \cite{Wi}
For any $\psi_1, \psi_2\in L^2(m^{\BMS})$, we have 
\[
\lim_{s\to \infty} \int_{\G\ba G} \psi_1(xa_s) \psi_2(x) dm^{\BMS}(x)=m^{\BMS}(\psi_1) m^{\BMS}(\psi_2).
\]
\end{thm}
 
 
The following is immediate from Theorem \ref{fm}:
\begin{thm}\label{fmu}
 Fix $\psi \in L^2(m^{\BMS})$ and a compact subset $\Omega$ of $X$.
  Let $\mathcal F\subset L^2 (m^{\BMS})$ be a a relatively compact subset.
     Then for any $\e>0$, there exists $S>1$ such that for all $s\ge S$ and
  any $\varphi\in \mathcal F$,
$$\left| \int_{X} \psi (ga_s) \varphi (g) dm^{\BMS}(g) - m^{\BMS}(\psi) m^{\BMS}(\varphi ) \right| \le \e .$$
 \end{thm}

Theorem  \ref{fm} can be used to prove the following:
for any $\psi\in C_c(X)$, $x\in X$, and  any bounded Borel subset $B\subset xU$ with $\mups(\partial(B))=0$,
\begin{equation}\label{bmsu}
 \lim_{s\to \infty} \int_{B}\psi(xu_\tbf a_s)d\mups(\tbf)  =\mups(B)m^{\BMS}(\psi); \text{ and } \end{equation}
 \begin{equation}\label{bur}   \lim_{s\to \infty} e^{(n-1-\delta)s}\int_{B}\psi(xu_\tbf a_s)d\tbf= \mups(B)m^{\BR}(\psi) \end{equation}
  (see \cite{Roblin2003} and  \cite{OS} for 
   $M$-invariant $\psi$'s and \cite{MO-Matrix} for general $\psi$'s). 

 In this paper, we will need uniform versions of these two equidistribution statements; more precisely,
 the convergence in both statements are uniform on compact subsets.
 Since the uniformity will be crucial for our
purpose, and it is not as straightforward as in the case when $\Gamma$ is a lattice, we will provide a proof. We will be using the following definitions.

Let  $d$
denote the left $G$-invariant and  bi-$K$-invariant metric on $G$ which induces  the hyperbolic metric on $G/K=\bH^n$.
For a subset $S\subset G$, $S_\e$ denotes the $\e$-neighborhood of $e$ in the metric $\dist$, that is,
\be\label{eq:S-ep}
S_\e=\{s\in S: \dist(s, e)<\e\}.
\ee






\begin{lem}\label{omr}
For any compact subset $\Omega\subset G$,
there exists $R>0$ such that 
$$(gB_{{\cont}}(R))^-\cap \Lambda(\G) \ne \emptyset\quad\text{ for any $g\in \Omega$}.$$
\end{lem}
\begin{proof} The claim follows since
the map $g\mapsto d_{\partial (\bH^n)}
(g^-, \Lambda(\G))$ is continuous where $d_{\partial (\bH^n)}$ is the spherical distance of $\partial(\bH^n)$.
\end{proof}

For $\Omega \subset X,$
the injectivity radius of $\Omega$ is defined to be the
supremum of $\e>0$ such that the map $g \mapsto xg$ is injective on $G_\ep$ for all $x\in \Omega$.

\begin{thm}\label{thm;exp-peice}\label{mixing}
 Fix $\psi\in C_c(X)$, a compact subset $\Omega\subset X$ and a number $r>0$ smaller than
 the injectivity radius of $\Omega$. Fix $c>0$ and
let $\mathcal F_r (c)$ be an equicontinuous family of functions $0\le f\le 1$ on $B_U(r)$ such that
$f|_{B_U(r/4)}\ge c$.
For any sufficiently small $\e>0$, there exists $S=S(\psi,\mathcal F_r,\ep,\Omega)>1$ 
 such that for $x\in  \Omega$ with $x^+\in\Lambda(\G)$,
 for any $f \in \mathcal F_r$, and for any $s>S$,
 we have \begin{enumerate}
\item $$\left|  \int_{B_U(r)}\psi(xu_\tbf a_s)f(\tbf) d\mups(\tbf)  -\mups(f)m^{\BMS}(\psi)\right| < \e \cdot \mups(f); $$
\item 
$$ \left| e^{(n-1-\delta)s}\int_{B_U(r)}\psi(xu_\tbf a_s)f(\tbf)d\tbf -\mups(f)m^{\BR}(\psi)\right| <  \e \cdot \mups(f)$$
where $\mups(f)=\int_{xB_U(r)}f(\tbf) d\mups(\tbf)$.
\end{enumerate}
\end{thm}
\begin{proof}
Let $\tilde\Om$ be a compact subset  of $G$ which projects onto $\Om B_U(1)$.
Fix a non-negative function $\psi\in C_c(G)$, whose support injects to $\G\ba G$. 
For each  small $\eta>0$, we define functions $\psi_{\eta,\pm}$ on $G$ by
\[
\psi_{\eta,+}(h):=\sup_{w\in G_\eta} \psi (hw)\text{ and } \psi_{\eta,-}(h):=\inf_{w\in G_\eta}\psi (hw).
\]
Fix $\e>0$. By the continuity of $\psi$ and equicontinuity of $\mathcal F$, there exists $0<\eta <\e $ such that 
\begin{align*}
&\sup_{h\in G}|\psi_{\eta,+}(h)-\psi_{\eta,-}(h)| < \e/2,\text{ and }\\
& \sup_{f\in\mathcal F_r, |\tbf-\tbf'|<\eta}|f(\tbf)-f({\tbf'})|<\e/2 .
\end{align*}

Recall $P=MA{\cont}$ and $P_\eta$ denotes the $\eta$-neighborhood of $e$ in $P$. 
The basic idea is to thicken $xB_U(r)$ in the transversal direction of $P$, 
and then to apply the mixing Theorem~\ref{fm}.
However the transverse measure of $xP_\eta$ may be trivial for all small $\eta$
 in which case the thickening approach does not work. So we flow
$x$ until we reach $xa_{s}$ for which the transverse measureon $xa_s P_\eta$ is non-trivial uniformly over all
$x\in \Om$. This however will force us to further subdivide $B_U(r)$ into ``smaller" boxes. 
Let us now begin to explain this process carefully.

For any $p\in P_\eta$ and $\tbf\in B_U(1)$, 
we have $$p^{-1}u_{\tbf} \in  u_{\rho_p(\tbf) }P_{D}$$ where $\rho_p:B_U(1) \to B_U(1+O(\eta))$ 
is a diffeomorphism onto its image
and $D>0$ is a constant depending only on $\eta$.


Let $R>1$ be as in Lemma \ref{omr} with respect to $\tilde\Omega$.
Set 
\be\label{eq:r-1}
\mbox{$r_1:=\eta/R$ and $s_0:=\log (r_1^{-1})$.}
\ee
Without loss of generality we may assume $r_1< r/10$, by taking $\eta$ smaller if necessary.

For any $g\in \tilde\Omega$ we put $$g_0:=ga_{s_0}\in \tilde\Om a_{s_0}.$$
Then for all $g\in \tilde\Omega,$ we have
\be\label{pop}
\nu_{g_0P}(g_0P_\eta)>0;
\ee
this follows from he choice of $R$ in view of the fact that 
$(gB_{{\cont}}(R))^-$ is contained in $(g_0P_\eta)^-= (gP_{\eta e^{s_0}})^-$.


For any $c_1>0,$ let $\mathcal F_{r_1}(c_1)\subset C(B_U(r_1))$ 
be an equicontinuous family of functions $f$ such that $0\le f \le 1$ and $f|_{B_U(r_1/4)}\geq c_1$. 
Fix $g\in  \tilde\Omega$ with $g^+\in \Lambda(\G)$. 
We first claim that there exists some $S>1$ (depending on $r_1,$ $c_1,$ equicontinuity of $\mathcal F_{r_1}(c_1)$
and $\tilde\Omega$)
such that for all $s\geq S,$ and all $f\in \mathcal F_{r_1}(c_1)$, 
\begin{equation}\label{c2}
\left| \sum_{\gamma \in \G} \int\psi(\gamma g u_\tbf a_s) f ({\tbf}) 
d\mu_{g}^{\PS}(\tbf) - \mu_g^{\PS}(f) \tilde m^{\BMS}(\psi) \right|<\e \cdot \mu_g^{\PS}(f) 
\end{equation}
where $\mu_g^{\PS}(f)=\int_{gB_U(r)} f(\tbf) d\mu_g^{\PS}(\tbf)$.
Associated to $f\in \mathcal F_{r_1}(c_1)$, define the function $f_0\in C_c(g_0B_U(1))$ by
\[
f_0(g_0u_\tbf):=f( {e^{-s_0}} \tbf).
\]
Let $\varphi_{0}:=\varphi_{f,g}$ be a function defined on $g_0P_\eta B_U(1) \subset G$ given by
\[ \varphi_{0}(g_0pu_{\rho_p(\tbf)})=
\frac{ f( {e^{-s_0}} \tbf) \chi_{g_0P_\eta}  (g_0 p)}{\nu(g_0P_\eta) e^{\delta \beta_{ u_\tbf^+}( u_\tbf (o),  p u_{\rho_p(\tbf)} (o))}},
\] which is well-defined by \eqref{pop}.
Then $\varphi_{0}$ is supported in the set $g_0B_{U}(1+O(\eta)) P_\eta.$

We observe
that for each $p\in P_\eta$, $(g_0u_\tbf)^+= (g_0 p u_{\rho_p(\tbf)} )^+$, the measures $d\mu_{g_0}^{\PS}(\tbf)$
and $d((\rho_p)_*\mu_{g_0p}^{\PS} )(\tbf)=d\mu_{g_0p}^{\PS} ({\rho_p(\tbf)} )$ are absolutely continuous with each other, and the Radon-Nikodym derivative at $\tbf$
is given by
 $e^{-\delta \beta_{g_0u_\tbf^+}( u_\tbf (o),  p u_{\rho_p(\tbf)} (o))}.$

If $s>2s_0+\log (D/ \eta)$, then
\begin{align*}  &e^{\delta s_0}  \int_{B_U(r_1)} \psi(\gamma g u_\tbf a_s) f(\tbf) d\mu_{g}^{\PS}(\tbf)
\\ &= \int_{B_U(1)}\psi(\gamma g_0 u_\tbf a_{s-s_0}) f_0(g_0 u_\tbf) d\mu_{g_0}^{\PS}(\tbf) \le 
\\ &  \tfrac{1}{{\nu(g_0P_\eta)}}\!\!\int_{P_\eta}\!\! \int_{B_U(1)}\!\!\!\!\psi_{\eta,+} (\gamma g_0 p u_{\rho_p(\tbf)} a_{s-s_0}\!){f_0(g_0 u_\tbf)}
 \tfrac{d\mu_{g_0}^{\PS}}{d((\rho_p)_*\mu_{g_0p}^{\PS} )} (\tbf) d\mu_{g_0p}^{\PS}(\rho_p(\tbf)\!)    d\nu(g_0p)
\\ & =\int_{G} \psi_{\eta,+}(\gamma h  a_{s-s_0}) \varphi_{0}( h) d\tilde m^{\BMS}(h)
\end{align*}
by \eqref{bmsp}.

Hence if we set $$\Psi_{\eta,+}(\G h):=\sum_{\gamma\in \G}
\psi_{\eta,+}(\gamma h)\quad\text{
and}  \quad \Phi_{f, g}(\G h):=\sum_{\gamma\in \G} e^{-\delta s_0} \varphi_{f,g}(\gamma h),$$
then 
\be\label{fff}
\sum_{\gamma \in \G} \int\psi(\gamma g u_\tbf a_s)f(\tbf) d\mu_g^{\PS}(\tbf) 
\le  \langle a_{s-s_0} \Psi_{\eta,+},  \Phi_{f,g} \rangle_{m^{\BMS}}.
\ee


Since $\mathcal F_{r_1}(c_1)$ is an equicontinuous and uniformly bounded 
family of functions, the collection $\{\Phi_{f, g}:f\in \mathcal F_{r_1}(c_1), g\in \tilde\Om \}$ 
is a relatively compact subset of $L^2(m^{\BMS}).$

In view of the assumption that $\inf_{f\in \mathcal F_{r_1}(c_1)} f|_{B_U(r_1/4)}\geq c_1$,
 Corollary~\ref{fxp} implies that 
\[
 c_0:=\inf_{g\in \tilde\Omega, g^+\in \Lambda(\G)} \{\mu_g^{\PS}(f): f\in \mathcal F_{r_1}(c_1)\} >0 .
\]

Now by Theorem \ref{fmu}, there exists  $S>2s_0+\log (D_\eta \eta^{-1})$ 
such that for all $s>S$ and for all $g\in  \tilde \Omega$,
\[
|\langle a_{s-s_0} \Psi_{\eta,+},  \Phi_{f,g}\rangle_{m^{\BMS}} - m^{\BMS}(\Psi_{\eta,+}) m^{\BMS}(\Phi_{f,g})|\le  c_0 \e/2.
\]
Since 
$m^{\BMS}(\Phi_{f,g} )=e^{-\delta s_0} \mu_{g_0}^{\PS}(f_0) =\mu_g^{\PS}(f),$ we deduce
\begin{align*}
\langle a_{s-s_0} \Psi_{\eta,+},  \Phi_{f,g}\rangle_{m^{\BMS}} &\le
\tilde m^{\BMS}(\psi) \mu_g^{\PS}(f) +\e /2\mu_g^{\PS}(f) +c_0 \e /2 \\& \le
\tilde m^{\BMS}(\psi) \mu_g^{\PS}(f) +\e \mu_g^{\PS}(f).
\end{align*}


Combined with \eqref{fff}, for all $s\ge S$,
$$\sum_{\gamma \in \G} \int\psi(\gamma g u_\tbf a_s)f(\tbf) d\mu_g^{\PS}(\tbf) 
\le  \tilde m^{\BMS}(\psi) \mu_g^{\PS}(f) + \e  \mu_g^{\PS}(f). $$
 The lower bound 
 $$\sum_{\gamma \in \G} \int\psi(\gamma g u_\tbf a_s)f(\tbf) d\mu_g^{\PS}(\tbf) 
\ge  \tilde m^{\BMS}(\psi) \mu_g^{\PS}(f) - \e \mu_g^{\PS}(f). $$
can be obtained similarly. 
This finishes the proof of~\eqref{c2}.

We now deduce (1) using~\eqref{c2} and a partition of unity argument. 
Consider a sub-covering $\mathcal B$ of $\{u_{\tbf}B_U(r_1): u_{\tbf}\in B_U( r)\}$ for $B_U(r)$
of multiplicity $\kappa$ depending only on the dimension of $U$.

Let $\{\sigma_\tbf: \tbf\in I\}\subset \mathcal F_{r_1}(1)$ be an equicontinuous partition of unity 
subordinate to $\mathcal B$.
By the choice of $\eta$ and $r_1<\eta$, we have
\[
\sup_{f\in\mathcal F_r, |\tbf'|<r_1}|f(\tbf+\tbf')-f({\tbf})|<\ep .
\]

For any $f\in\mathcal F_{r}(c)$, define 
\[
I_1(f):=\{{\tbf}\in I:   \sup_{\tbf'\in B_U(r_1)} |f({\tbf +\tbf'})\sigma_{\tbf}({\tbf'})|<2\ep\}
\]
and let $I_{2}(f):=I-I_1(f).$

Then the family $\{ f\sigma_{\tbf}: {\tbf}  \in I_2(f), {f\in \mathcal F_r(c)}\}$ 
satisfies the conditions of the family in~\eqref{c2} with $r_1$ and $c_1=\ep.$
Therefore, there exists
some $S>1$, independent of $f\in \mathcal F_r(c)$,
 so that  for any $x\in  \Omega$ with $x^+\in\Lambda(\G)$, for all $s>S$ and all $\tbf\in I_2(f)$, we have  

\begin{equation}\label{p-u-1}
\left|  \int\psi(xu_\tbf a_s) f (\tbf+{\tbf'})\sigma_{\tbf}({\tbf'}) d\mu_{xu_{\tbf}}^{\PS}(\tbf') - 
\mu_{xu_{\tbf}}^{\PS}(f\sigma_{\tbf}) \tilde m^{\BMS}(\psi) \right|<\e \cdot \mu_{xu_{\tbf}}^{\PS}(f\sigma_{\tbf}). 
\end{equation}
Hence 
$$\sum_{\tbf\in I_2(f)}\left|  \int\psi(xu_\tbf a_s) f (\tbf +{\tbf'})\sigma_{\tbf}({\tbf'}) d\mu_x^{\PS}(\tbf') - \mu_{xu_{\tbf}}^{\PS}(f\sigma_{\tbf}) \tilde m^{\BMS}(\psi) \right|\le O(\ep)\mu_x^{\PS}(f) .$$
On the other hand, we have
\begin{multline}
\label{p-u-2}
\sum_{\tbf\in I_1(f)}\left|  \int\psi(xu_\tbf a_s) f (\tbf +{\tbf'})\sigma_{\tbf}({\tbf'}) d\mu_x^{\PS}(\tbf') - \mu_{xu_{\tbf}}^{\PS}(f\sigma_{\tbf}) \tilde m^{\BMS}(\psi) \right|\\
< \sum_{I_1(f)}O_{\psi}(\e) \mu_x^{\PS}(u_{\tbf}B_U(r_1)) \leq O(\ep)\mu_{xu_{\tbf}}^{\PS}(B_U(r))\leq O(\ep)\mu_x^{\PS}(f) 
\end{multline}
where the implied constants depend only on $\psi$, $\kappa$ and the positive lower bound for
$\mu_x^{\PS}(f)$'s.

Since $f=\sum_{\tbf\in I_1(f)\cup I_2(f)}f\sigma_{\tbf},$
~\eqref{p-u-1} and~\eqref{p-u-2} imply (1) of the theorem.


Now the uniformity statement regarding (2) follows from the uniformity of  (1); 
this follows directly from the argument in \cite{OS} using the comparison
of the transversal intersections.
\end{proof}

\section{Equidistribution of non-closed $U$-orbits}

We call a point $x\in X$ with $x^+\in \Lambda(\G)$ a PS-point. 

Recall:
$$B_U(T):=\{u_\tbf\in U: |\tbf|\le T\} $$
where $|\tbf|$ is the maximum norm of $\tbf\in U=\br^{n-1}$.

One way of characterizing the set $\Lambda_{\rm r}(\G)$ of radial limit points in terms of $U$-orbits
is that if $x^-\in \Lambda_{\rm r}(\G)$, then
$u_\tbf\to xu_\tbf$ is an injective map from $U$ to $X$, $xUM$ is not closed in $X$ and $\mu_x^{\PS}$ is an infinite measure
(Lemma \ref{rinf}).

The following theorem of Schapira shows that for $x^-\in \Lambda_{\rm r}(\G)$,
most $\PS$-points of $xB_U(T)$ come back to a compact subset in a quantitative way.

\begin{thm} \cite{Sch-nondiv} \label{p;horo-nondiv}\label{sc}
For any $\e>0$ and any compact subset $\Om \subset X$,  there exists a compact subset $Q=Q(\e, \Om)\subset X$
 such that
for any $x\in \Om$ with $x^-\in \Lambda_{\rm r}(\Gamma)$, there exists some $T_x>0$ such that
for all $T\geq T_x$,
\[
\mu_x^{\PS}\{u_\tbf\in B_U(T): xu_\tbf \in Q\}\geq 
(1-\e)\mu_x^{\PS}(B_U(T)).
\]
\end{thm}

We will also make a repeated use of the following basic fact:
\begin{lem} \cite[Lem 4.5]{Sch-nondiv} \label{uc} For a fixed $\kappa >1$, there exists $\beta >1$ such that for any compact 
$\Om\subset X$, there exists $T_0(\Om)>1$ such that
for any $x\in \Om$ with $x^-\in \Lambda_{\rm r}(\G)$ and for any $T>T_0(\Om )$, we have 
$$\mu_x^{\PS}(B_U( \kappa \, T) )\le \beta \cdot 
\mu_x^{\PS}(B_U(  T)) .$$
\end{lem}

\subsection{Relative PS-size of a neighborhood of $\partial(B_U(T))$} In order to study the distribution of $xB_U(T)$ either in the PS measure or in the Lebesgue measure,
it is crucial to understand the size of a neighborhood of $\partial (xB_U(T))$ compared
to the size of $xB_U(T)$ in the PS-measure. 

 We do not know in general whether the following is true:
  for $x\in X$ with $x^-\in \Lambda_{\rm r}(\G)$, there exists $\rho>0$ such that
\begin{equation} \label{ccc} \lim_{T\to \infty} \tfrac{\mu_x^{\PS} (B_U(T+\rho)\setminus B_U(T-\rho))}{\mu_x^{\PS}(B_U(T))}=0 .\end{equation}

Noting that $$\tfrac{\mu_x^{\PS} (B_U(T+\rho)\setminus B_U(T-\rho))}{\mu_x^{\PS}(B_U(T))}
=\tfrac{\mu_{xa_{-\log T}}^{\PS} (B_U(1+\rho T^{-1})\setminus B_U(1-\rho T^{-1}))}{\mu_{xa_{-\log T}}^{\PS}(B_U(1))},$$
this question is directly related to the uniformity of the size of a neighborhood of a $1$-sphere based at 
$xa_{\log_{-T}}$ where $x$ is in a fixed compact subset.

When $\G$ is convex-cocompact, and  $x^+\in \Lambda(\G)$ so that $x\in \text{supp}(m^{\BMS})$, then
 \eqref{ccc} now follows since $\text{supp}(m^{\BMS})$ is compact and any $1$-sphere
has zero PS-measure, and hence 
$\mu_{xa_{-\log T}}^{\PS}(B_U(1+\rho T^{-1}) \setminus B_U(1-\rho T^{-1}))$ is uniformly small for all large $T$. 
When there is a cusp in $X$, the PS-measure of a ball around $xa_{-\log T}$ depends on the rank of a cusp where it stays
(see Theorem \ref{sha} for a precise statement), and it is not clear
whether \ref{ccc} holds or not.

\medskip

We will prove a weaker result which will be sufficient for our purpose.
Let us begin with the following lemma which a consequence of the fact that
$\G$ is Zariski dense, together with a compactness argument.

For a coordinate hyperplane $L$ of $\br^{n-1}=U$, set $L^{(1)}:=L\cap B_U(1).$
Define  $\mathcal O_\rho (L^{(1)})$ to be the $\rho$-thickening of $L^{(1)}$ in the orthogonal direction.

\begin{lem}\label{lem;coord-thick}
Let $\Om \subset X$ be a compact subset.
For any $\eta>0$ there exists some $\rho_0>0$ so that
\[
 \mu_{y}^{\PS}({\mathcal O}_{\rho}(L^{(1)}))<\eta \cdot \mu_y^{\PS}(yB_U(1)).
\]
for all $y\in \Om \cap\supp(m^{\BMS})$ and all $\rho\leq \rho_0.$
\end{lem}

\begin{proof}
We prove this by contradiction. Suppose the above fails;
then there exist $\eta>0$, and  sequences
$y_k\in \Om  \cap \supp(m^{\BMS})$ and coordinate hyperplanes $L_k$ so that
\[
\mu_{y_k}^{\PS}({\mathcal O}_{1/k}({L_k}^{(1)})\geq\eta \cdot\mu_{y_k}^{\PS}(y_kB_U(1)).
\]

Passing to a subsequence we may assume $y_k\to y\in \Om  \cap\supp(m^{\BMS})$
and $L_k= L$ for all $k.$
Now for every $\alpha>0$ we have ${\mathcal O}_{1/k}(L^{(1)})\subset{\mathcal O}_\alpha (L^{(1)})$ 
for all large enough $k.$ 
Since  $\mu_{y_k}^{\PS}\to\mu_y^{\PS} $ by Lemma \ref{cont}, it follows that
 $\mu_y^{\PS}({\mathcal O}_\alpha (L^{(1)}))\geq\eta/2$ for all $\alpha>0.$
This implies $\mu_y^{\PS}(L)>0$ which contradicts the Zariski density of $\G.$
This proves the lemma.
\end{proof}

\begin{lem}\label{l;avoid-var}
Let $\Om \subset X$ be a compact subset. 
For any $\e>0$ there exist some $\rho_0=\rho_0(\e,\Om)>0$ 
with the following property: for all $x\in\Om$ with $x^-\in\Lambda_{\rm r}(\G)$
there exists some $R_x>0$ so that 
\[
 {\mu_x^{\PS} (B_U(T +\rho) \setminus B_U(T -\rho))}<\e \cdot {\mu_x^{\PS}(B_U(T))}.
\]
for all $T\geq R_x$ and all $0<\rho\leq \rho_0$. 
\end{lem}

\begin{proof}
Let $c_1>1$ and $T_1>1$ 
be so that  for all $x\in \Om$ with $x^-\in\Lambda_{\rm r}(\G)$ and $T\ge T_1$,
$$\mu_x^{\PS}(B_U(T+2))\leq \mu_x^{\PS}(B_U(2T)) \le c_1\mu_x^{\PS}(B_U(T))$$ 
as given by Lemma~\ref{uc}.

Apply Theorem~\ref{p;horo-nondiv} with $\Om$ and $\e/(2c_1)$
and let $Q_0=Q(\e/(2c_1), \Om)$ and $T_x\ge 1$ be given by that theorem.

Let $T>T_x+T_1$ in the following.
By our choice of $Q_0$ and $T_x$, we have 
\begin{align*}
\mu_x^{\PS}\{\tbf\in B_U(T+2): xu_\tbf \notin Q_0\}&\leq 
\tfrac{\e}{2c_1}\mu_x^{\PS}(B_U(T+2))\\
&\le \tfrac{\e}{2}\mu_x^{\PS}(B_U(T)).
\end{align*}
For each $y\in \partial (xB_U(T))\cap Q_0\cap\supp(\mu_x^{\PS})$
and any $\rho>0$, 
put 
\[
 {W}_{\rho,T}(y):=x(B_U(T +\rho)\setminus B_U(T -\rho)))\cap yB_U(1).    
\]  

Note that for each $y\in \partial (xB_U(T))$ there exist some 
coordinate hyperplane $L_1(y),\ldots,L_\ell(y),$ for some $1\leq \ell\leq n-1,$ 
so that 
\be\label{e;variety-coor}
{W}_{\rho,T}(y)\subset \cup_j y{\mathcal O}_\rho (L_j(y)^{(1)}).
\ee 
Since
\[
\{yB_U(1): y\in\partial (xB_U(T))\cap Q_0\cap \supp(\mu_x^{\PS}) \}
\] 
is a covering for 
\[
x(B_U(T +\rho)\setminus B_U(T -\rho))\cap Q_0 \cap \supp(\mu_x^{\PS} ),
\]
we can find a finite sub-cover
$\cup_{y\in J_T} yB_U(1)$ with multiplicity at most $\kappa,$ 
where $\kappa$ depends only on the dimension of $U$ by the Bescovitch covering theorem.

Applying Lemma~\ref{lem;coord-thick} with 
$\eta:=\e/(2 c_1\kappa(n-1))$ and $Q_0$ we get some $\rho_0=\rho_0(\eta,Q_0)$
so that the conclusion of that lemma holds true.
Thus using that lemma and~\eqref{e;variety-coor}, for every $0<\rho\leq \rho_0$ and $T>T_x+T_1$, we have
\begin{align*}&
 \mu_x^{\PS} (B_U(T +\rho)\setminus B_U(T -\rho))\\ &\leq
 \mu_x^{\PS}(\{u\in B_U(T+2):xu\notin Q_0\})+\mu_x^{\PS}(\cup_{y\in J_T}{W}_{\rho,T}(y))\\
 &\leq \tfrac{\e}2\mu_x^{\PS}(B_U(T))+\textstyle\sum_{y\in J_T} \mu_{y }^{\PS}({W}_{\rho,T}(y))\\
 &\leq \tfrac{\e}2\mu_x^{\PS}(B_U(T))+\textstyle \sum_{i=1}^{n-1}\sum_{y\in J_T } \mu_{y }^{\PS}({\mathcal O}_{\rho}(L_i(y)^{(1)}))\\
 &\leq \tfrac{\e}2 \mu_x^{\PS}(B_U(T))+{\eta\kappa (n-1)} \mu_x^{\PS}(\cup_{y\in J_T} yB_U(1)) \\ 
 &\leq \tfrac{\e}2 \mu_x^{\PS}(B_U(T))+\tfrac{\e}{2c_1}\mu_x^{\PS}(B_U(T+2)) \\& \leq
  \e \mu_x^{\PS}(B_U(T)).
\end{align*}
This implies the lemma. 
\end{proof}





\subsection{Equidistribution of a $U$-orbit in PS-measure}
We now prove an equidistribution result of $xU$ in the $\PS$-measure: 
\begin{thm}\label{t;horo-equid-BMS}\label{eqbms}
Let $x^-\in \Lambda_{\rm r}(\Gamma).$
Then for any $\psi\in C_c(X)$ we have
\[
\lim_{T\to\infty}\frac{1}{\mu_x^{\PS}(B_U(T))}\int_{B_U(T)}\psi(xu_\tbf)d\mups(\tbf)= m^{\BMS}(\psi).
\] 
\end{thm} 

The main ingredients of the following proof are
Theorem \ref{mixing}, Theorem \ref{sc} and Lemma \ref{l;avoid-var}. 
In view of Theorem \ref{sc},
we only need to focus on the part of  $xB_U(T)$ which comes back to a fixed compact subset $Q$, as this part occupies
most of the PS-measure. We will use a partition of unity argument for 
a cover of $xB_U(T)\cap Q$ by small balls centered at PS-points. 
Each function in the partition of unity will be controlled by Theorem \ref{mixing}, here the uniformity in loc. cit is of crucial importance.
In this process, we have an error occurring in a small neighborhood of the boundary of $xB_U(T)$ and
Lemma \ref{l;avoid-var} says this error can be controlled. 

More precisely, we proceed as follows.
\proof By the assumption that $x^-\in \Lambda_{\rm r}(\G)$,
there exists a compact subset $\Om \subset X $ and a sequence $s_i\to +\infty$
such that $xa_{-s_i} \in \Om$. 

Let $Q=Q(\e, \Om)\subset X$ be a compact subset given by Theorem \ref{sc}.
Let $\rho_0=\rho_0(\e,\Omega)>0$ be given by Lemma \ref{l;avoid-var} applied with $\Om$ and $\e>0$.
Let $S>1$ be the constant provided by Theorem \ref{mixing} applied with $\Om$, $r=\rho_0$ and $\e>0$.
Now choose $s_0>S$ so that $x_0:=xa_{-s_0}\in \Omega$; note that $x_0^-\in\Lambda_{\rm r}(\G).$
Apply Theorem~\ref{p;horo-nondiv} and Lemma~\ref{l;avoid-var} with $x_0,$
and let $T_0:=T_{x_0}+R_{x_0}.$ Then for all $T\geq T_0$
we have 
\be\label{e;non-div-use}
\mu_{x_0}^{\PS}\{u_\tbf\in B_U(T): x_0u_{\tbf}\notin Q\}\leq 
\e\cdot \mu_{x_0}^{\PS}(B_U(T)).
\ee

For each $T\ge T_0$, we will consider a covering of the set
\[
D_T:=x_0B_U(T) \cap \supp(\mu_x^{\PS})\cap Q
\] 
and an equicontinuous partition of unity $\{f_y: y\in J_T\}$ subordinate to this covering 
as follows. Let $J_T:=\{y\in D_T\}$ denote a maximal
collection of points where $\{yB_U(\rho_0/8)\}$ are disjoint. Then 
$\{yB_U(\rho_0/2):y\in J_T\}$ covers $D_T,$ moreover, the covering 
$\{yB_U(\rho_0):y\in J_T\}$
has multiplicity $\kappa$ which depends only on the dimension of $U$.

Let $f\in C^\infty (B_U(\rho_0))$ be such that
\[
\mbox{$0\leq f\leq 1,$ $f|_{B_U(\rho_0/8)}=1$ and $1/2\leq f|_{B_U(\rho_0/2)}\leq 1.$}
\]  
For every $y\in D_T$ put $f'_y(yu_{\tbf}):=f(u_{\tbf})$
and let $F_T=\sum_{y\in J_T} f'_y.$ Then 
\[
1/2\leq F_T(x_0u_\tbf)\leq \kappa,\quad\text{for $x_0u_\tbf\in D_T$}.
\]
For each $y\in J_T$, put $f_y:={f'_y}/{F_T}.$ Then  
$\{f_y: y\in J_T\}$ is an equicontinuous partition of unity
subordinate to the covering $\{yB_U(\rho_0):y\in J_T\}.$ In particular, we have
$0\le f_y\le 1$, 
$\sum_{y\in J_T}  f_y =1$
on $x_0B_U(T-\rho)$, and $\sum_{y\in J_T}  f_y =0$ outside $x_0B_U(T+\rho)$.
Moreover, we have 
\be\label{eq:part-unity}
{\mbox{$f_y|_{yB_U(\rho_0/8)}=1$ and $\tfrac{1}{2\kappa}\leq f_y|_{yB_U(\rho_0/2)}\leq 1\;$ for all $y\in J_T$}.}
\ee

Fix $\psi\in C_c(X)$. Without loss of generality, we may assume that $\psi$ is non-negative.
For $T\ge T_0$, by applying Theorem \ref{mixing}  and Lemma \ref{l;avoid-var}(1), we have
\begin{align*}&
\int_{x_0u_\tbf\in Q\cap  x_0B_U(T)} \psi(x_0u_\tbf a_{s_0}) d\mu_{x_0}^{\PS}(\tbf)  =
\\ & \sum_{y\in J_T} \int_{Q\cap x_0B_U(T)} \psi(x_0u_\tbf a_{s_0}) f_y (x_0u_\tbf)
d\mu_{x_0}^{\PS}(\tbf)  
\\ &+ O(\mu_{x_0}^{\PS}(B_U(T+\rho_0)\setminus B_U(T-\rho_0)))
\\ &=\sum_{y\in J_T}  {\mu_{x_0}^{\PS}  (f_y)}\, m^{\BMS}(\psi)(1\pm  O({\e})) + 
O(\e\cdot \mu_{x_0}^{\PS}(B_U(T)))
\\ &=
 \mu_{x_0}^{\PS}(B_U(T))\cdot m^{\BMS}(\psi)+ O(\e \cdot \mu_{x_0}^{\PS}(B_U(T)))
 \end{align*}
where the implied constants depend only on $\psi$.

Hence if $T\ge e^{s_0} T_0$, then
\begin{align*}
&\int_{B_U(T)}\psi(xu_\tbf)d\mups(\tbf)\\ 
& = e^{\delta s_0}  \int_{B_U(Te^{-s_0})} \psi(x_0 u_\tbf a_{s_0})d\mu_{x_0}^{\PS}(\tbf)\\  
&=e^{\delta s_0}  \left(\int_{x_0B_U(Te^{-s_0})\cap Q } \psi(x_0 u_\tbf a_{s_0})d\mu_{x_0}^{\PS}(\tbf)  
+ O\left(\e\cdot \mu_{x_0}^{\PS}(B_U(T))\right)\right) \\
& =e^{\delta s_0} \left( \mu_{x_0}^{\PS}\left(B_U(Te^{-s_0})\right)\cdot m^{\BMS}(\psi)+ 
O\left( \e \cdot \mu_{x_0}^{\PS}(B_U(Te^{-s_0}))\right)\right)\\ 
&= \mu_{x}^{\PS}(B_U(T))\cdot m^{\BMS}(\psi)+ O\bigl(\e \cdot \mu_{x}^{\PS}(B_U(T))\bigr)
\end{align*}
where the implied constant depends only on $\psi$.

This  finishes the proof as $\e>0$ is arbitrary.
 \qed

\subsection{Equidistribution of a $U$-orbit in Lebesgue measure}
We will use a similar idea as in the proof of Theorem~\ref{t;horo-equid-BMS} 
to show the equidistribution of $xB_U(T)$ in the Lebesgue measure. The main difference is that we now need to control the escape
of the orbit (measured in Lebesgue measure as opposed to
the PS measure as in Theorem \ref{eqbms})
to flares as well as to cusps. For this, we utilize  the idea of comparing 
the``transversal intersections'' of the PS-measure of $xU$ and the Lebesgue measure
of $xU$.

\begin{thm}\label{t;horo-equid-BR}\label{eqbr}
Let  $x^-\in \Lambda_{\rm r}(\Gamma).$
Then for any $\psi\in C_c(X)$, we have
\[
\lim_{T\to\infty}\frac{1}{\mu_x^{\PS}(B_U(T))}\int_{B_U(T)}\psi(xu_\tbf)d\tbf= m^{\BR}(\psi).
\] 
\end{thm}

\proof
Fix $x$ and $\psi$ as above.
Without loss of generality, we may assume that $\psi$ is non-negative.
Let $P$ denote the parabolic subgroup $MA{\cont}$.
We call $W=zP_{\e_1}U_{\e_0}$ an admissible box if $W$
is the injective image of $P_{\e_1}U_{\e_0}$ in $\G\ba G$
 under the map $g\mapsto zg$ and $\mu^{\PS}_{zp}(zpU_{\e_0})\ne 0$ for all $p\in P_{\e_1}$.

Since $\psi$ is compactly supported, using a partition of unity argument, we may assume without loss
of generality that $\psi$ is supported in an admissible box $z P_{\e_1}U_{\e_0} $
for some $z\in \G\ba G$ and $0< \e_1\le \e_0$.

Since $x^-\in \Lambda_{\rm r}(\G)$,
there exists a compact subset $\Om\subset \G\ba G$ and a sequence $s_i\to +\infty$
such that $xa_{-s_i} \in \Om$. Now fix $\e>0$ smaller than $\e_1$ and $\e_0$. 
Let $Q=Q(\e, \Om)\subset X$ be a compact subset given by Theorem \ref{sc}.
Apply Lemma \ref{l;avoid-var} with $\Om$ and $\e>0$
and let $\rho_0>0$ be so that 
\[
\mu_x^{\PS}\left(B_U(T+4\rho_0)\setminus B_U(T-4\rho_0)\right)< \e
\] 
for all $x\in\Omega$ whenever $T>R_x$.


Let $S>1$ be the constant provided by Theorem \ref{mixing}(2)
 applied with $\Om$, $\rho_0=r$, and $\e>0$.
Now choose $s_0>S$ so that $x_0:=xa_{-s_0}\in \Om;$ note that $x_0^-\in\Lambda_{\rm r}(\G).$ 
Apply Theorem~\ref{p;horo-nondiv} and Lemma~\ref{l;avoid-var} with $x_0$ and put $T_0:=T_{x_0}+R_{x_0}.$ 
By Theorem~\ref{p;horo-nondiv}, for all $T\geq T_0$, we have 
\be\label{e;non-div-use}
\mu_{x_0}^{\PS}\{u_\tbf\in B_U(T): x_0u_{\tbf}\notin Q\}\leq 
\e\cdot \mu_{x_0}^{\PS}(B_U(T)).
\ee

Given $\rho>0,$ for each $y\in x_0U$,
let $f_y\in C(yB_U(\rho))$ be a function such that $0\le f_y\le 1$ and $f_y=1$ on $yB_U(\rho/8)$.

\medskip

{\bf Claim A}: There exists $c>1$ such that for any $y\in x_0U$,
$$e^{(n-1-\delta)s_0}\int_{x_0U}\psi(y u_{\tbf}a_{s_0}) f_y(y u_{\tbf}) d\tbf  \ll_\psi \mu_y^{\PS}(f_{y,ce^{-s_0}\e_1, +})$$
where $ f_{y,\eta, +} (yu):=\sup_{u'\in U_{\eta}} f_y(yuu')  $ for $\eta>0$.
To prove this,
define
\[
 P_z(s_0):=\{p\in P_{\e_1}: 
 y B_U(\rho) a_{s_0} \cap zpU_{\e_0} \ne \emptyset\}.
\]

For small $\eta>0$, we define: for $w\in \text{supp}(\psi) G_{\e_0}$,
$$\psi_{\eta, +}(w):=\sup_{g\in G_{\eta}} \psi (wg), \quad \Psi_{\eta,+}(wp):=\int_{wpU} {\psi}_{\eta, +}(wpu_\tbf) 
d\tbf;$$ 
and for $zpu\in zP_{\e_1}U_{\e_0}$,
$${\tilde{\Psi}}_{\eta,+}(zpu)=\frac{1}{\mu^{\PS}_{zp}(zpU_{\e_0})} \Psi_{\eta,+}(zp).$$

 Then we have, for some fixed constant $c>0$,
 \begin{align*} &e^{(n-1-\delta)s_0}\int_{y B_U(\rho_0)}\psi(y u_{\tbf}a_{s_0}) f_y(y u_{\tbf}) d\tbf 
 \\ &\ll e^{-\delta s_0}   \sum_{p\in P_z(s_0)} {f_{y,ce^{-s_0}\e_1,+} }(zpa_{-s_0})\Psi_{c\e_1,+}(zp)
&&\text{ by  \cite[Lemma 6.2]{MO-Matrix} } 
\\
&\ll \int_{U} {\tilde{\Psi}}_{c\e_1 ,+} (y u_\tbf a_{s_0}) f_{y,ce^{-s_0} \e_1, +}
(y u_\tbf)d\mu_{y}^{\PS}(\tbf)
&&\text{ by  \cite[Lemma 6.5]{MO-Matrix} }  \\
&\ll\mu_y^{\PS}(f_{y,ce^{-s_0}\e_1,+}) \end{align*}
where the implied constant depends only on $\Psi$.
This implies the claim.

Without loss of generality, from now we assume that $s_0$
is big enough so that $ce^{-s_0}\e_1\leq\min\{\rho_0/10,\e\}.$

Using the claim (A) applied to a partition of unity subordinate to the covering 
by $\rho_0$-balls of the set 
\[
x_0\bigl(B_U(T'+\rho_0)\setminus B_U(T'-\rho_0)\bigr)
\]
and the fact that $ce^{-s_0}\e_1<\rho_0$, we get 
\begin{multline} 
\label{eq:avoid-var-use}e^{(n-1-\delta)s_0} \int_{B_U(T'+\rho_0)\setminus B_U(T'-\rho_0)} \psi(x_0u_\tbf a_{s_0}) d\tbf \\
\ll_{\psi} \mu_{x_0}^{\PS}(B_U(T'+2\rho_0)\setminus B_U(T'-2\rho_0)) ) \ll \e \cdot \mu_{x_0}^{\PS}(B_U(T')). 
\end{multline}
for all $T'>T_0$ where in the last in equality we used Lemma~\ref{l;avoid-var}.

As it was done in the proof of Theorem~\ref{eqbms},
for every $T'\ge T_0$, we fix a covering $\{yB_U(\rho_0)\}: y\in I_{T'}\}$ of
\[
\mbox{$ x_0B_U(T') \cap\supp ( \mu_{x_0}^{\PS})\cap Q$}
\] 
and let $\{f_y: y\in {I_{T'}}\}$ 
be an equicontinuous partition of unity subordinate to this covering as in~\eqref{eq:part-unity}. 
We recall in particular that 
\be\label{eq:rho-2-cover}
\mbox{$\{yB_U(\rho_0/2):y\in I_{T'}\}$ covers $x_0B_U(T') \cap\supp ( \mu_{x_0}^{\PS})\cap Q.$}
\ee

Set
\[
Q_0(T'):=\cup_{y\in {I_{T'}}} yB_U(\rho_0).
\]
Put $T_1:=Te^{-s_0}>T_0.$ Then using Theorem \ref{mixing}(2) and~\eqref{eq:avoid-var-use} we have
\begin{align*}&
e^{(n-1-\delta)s_0} \int_{x_0u_\tbf\in Q_0(T_1)} \psi(x_0u_\tbf a_{s_0}) d\tbf  
\\ & =e^{(n-1-\delta)s_0} \sum_{y\in I_{ T_1}} \int_{yB_U(\rho_0)} \psi(x_0u_\tbf a_{s_0}) f_y (x_0u_\tbf)
d\tbf    
\\ &+ O \left(e^{(n-1-\delta)s_0}  \int_{B_U(T_1+\rho_0)\setminus B_U(T_1-\rho_0)} \psi(x_0u_\tbf a_{s_0}) d\tbf\right)
\\ &=\sum_{y\in I_{T_1}}  {\mu_{x_0}^{\PS}  (f_y)}\, m^{\BR}(\psi)(1\pm  O({\e})) + 
O(\e\cdot \mu_{x_0}^{\PS}(B_U(T_1)))
\\ &=
 \mu_{x_0}^{\PS}(B_U(T_1))\cdot m^{\BR}(\psi)+ O(\e \cdot \mu_{x_0}^{\PS}(B_U(T_1)))
 \end{align*}
where the implied constants depend only on $\psi$.
 
Let $J_{T_1}$ be a maximal set
of points in $x_0B_U(T_1)\setminus Q_0(T_1)$ so that $\{yB_U(\rho_0/16):y\in J_{T_1}\}$ is a disjoint family.
Then $\{yB_U(\rho_0/4):y\in J_{T_1} \}$ covers $x_0B_U(T_1)\setminus Q_0(T_1);$ moreover, 
by~\eqref{eq:rho-2-cover} we have 
\begin{multline}
\label{eq:Q-0-T}\supp(\mu_{x_0}^{\PS})\cap \left(\cup_{y\in J_{T_1} }yB_U(\rho_0/3)\right)\\
\subset \left(x_0B_U(T_1+\rho_0)\cap Q^c\right)\cup\left(x_0(B_U(T_1+\rho_0)\setminus B_U(T_1-\rho_0))\right).
\end{multline}

Let $\{f_y: y\in J_{T_1}\}$ be a partition of unity subordinate to $\{yB_U(\rho_0/4):y\in J_{T_1}\}.$ 
Applying Claim $A$ to $\{f_y: y\in J_{T_1} \}$ and using ~\eqref{e;non-div-use} and~\eqref{eq:avoid-var-use}, we have
\begin{align*}&
 {e^{(n-1-\delta)s_0}}
\int_{x_0B_U(T_1)\setminus Q_0(T_1)} \psi(x_0u_\tbf a_{s_0})d\tbf \\  &\le
{e^{(n-1-\delta)s_0}}
\sum_{y\in J_{T_1}}  \int_{yB_U(\rho_0/4)} \psi(y u_\tbf a_{s_0}) f_y(yu_\tbf)  d\tbf  
\\
&\ll \sum_{y\in J_{T_1}} \mu_y^{\PS} (f_{y,\rho_0/10,+})\\ 
&\ll \mu_{x_0}^{\PS} (x_0 B_U(T_1+\rho_0) \cap Q^c) + \mu_{x_0}^{\PS} ((B_U(T_1+\rho_0)\setminus B_U(T_1-\rho_0)) &&{\text{by~\eqref{eq:Q-0-T}}} \\
&  \ll \e \cdot \mu_{x_0}^{\PS} (x_0 B_U(T_1)).
\end{align*}

Observing that
\begin{equation}\label{sumab} \frac{1}{\mu_x^{\PS}(B_U(T))}\int_{B_U(T)}\psi(xu_\tbf)d \tbf =
\frac{e^{(n-1-\delta)s_0} }{\mu_{x_0}^{\PS}(B_U(T_1))}\int_{B_U(T_1)}\psi(x_0u_\tbf a_{ s_0})d \tbf 
\end{equation}
we have shown that for all $T>T_0 e^{s_0}$,
$$\frac{1}{\mu_x^{\PS}(B_U(T))}\int_{B_U(T)}\psi(xu_\tbf)d \tbf =m^{\BR}(\psi) +O(\e)$$
which
 finishes the proof.
\qed

Both theorems \ref{eqbms} and \ref{eqbr} are proved in \cite{SM} for the case $n=2$. Theorem \ref{eqbms}
is also proved in \cite{SM} for the unit tangent bundle of a convex cocompact hyperbolic $n$-manifold. However as clear from the above proofs,
the proof in the convex cocompact case is considerably simpler since the support of $m^{\BMS}$ is compact.

 \begin{remark}
Although Theorems \ref{eqbms} and \ref{eqbr} are stated for the norm balls $B_U(T)$
with respect to the maximum norm on $U\simeq \br^{n}$, the only  property of the max norm we have used
is that $\{\tbf \in \br^n: \|\tbf\|_{\max}=1\}$ is contained in a finitely many union of algebraic sub-varieties 
in the proof of Lemma \ref{lem;coord-thick}.
In fact, our proofs work for any norm $\|\cdot \|$ on $\br^n$ as long as the $1$-sphere
$\{\tbf \in \br^n: \|\tbf\|=1\}$ is
contained in a finitely many union of algebraic sub-varieties.
\end{remark} 

\begin{remark}\label{uniremark}
Theorem \ref{t;horo-equid-BR} cannot be made uniform on compact subsets; for instance, if $x^-$ is very close to a parabolic limit point,
the convergence is expected to be slower. However, Egorov's theorem implies that for a given compact subset $Q\subset X$ and any $\e>0$,
there exists a compact subset $Q'$ with $m^{\BR}(Q-Q')\le \e$ on which the convergence in Theorem \ref{t;horo-equid-BR}
is uniform. We will use this observation later.
\end{remark}

\medskip

\section{Window lemma for horospherical averages}
In this section we first prove that for $x^-\in \Lambda_{\rm r}(\G)$,
the PS-measure of $xB_U(T)$ is 
not concentrated near the center $x$, in the sense that for any $\eta>0$, there exists $r>0$ such that
for all large $T\gg 1$,
\be\label{ncon} \mu_x^{\PS}(B_U(T)\setminus B_U(rT))\ge (1-\eta) \mu_x^{\PS}(B_U(T)).\ee
This is of course immediate in the case when $\G$ is a lattice, in which case 
$\mu_x^{\PS}(B_U(T))=\mu_x^{\Leb}(B_U(T))=c\cdot T^{n-1}$ for some fixed $c>0$.
 The inequality \eqref{ncon} also follows rather easily  when $\G$ is convex cocompact 
 since $\mu_x^{\PS} (xB_U(T)) \asymp T^\delta$ for all $x\in \text{supp}(m^{\BMS})$.
  For a general geometrically finite group, our argument is based on
   Sullivan's shadow lemma. We remark that
   \eqref{ncon} is not a straightforward consequence of Shadow lemma and finding $r$ in \eqref{ncon} is rather tricky
   as we need to consider several different possibilities for the locations of $xa_{-\log T}$ and $xa_{-\log(rT)}$
   in the convex core of $\G$ (see the proof of Lemma \ref{p;horo-wind}).
 
\subsection{Shadow lemma}\label{shadowsection}

For $\xi\in \Lambda_{\rm bp}(\G)$,
choose $g_{\xi}\in G$ such that $g_\xi^-=\xi$, and set, for each $R>0$,
$$\mathcal H(\xi, R)=\cup_{s>R}  \; g_\xi Ua_{-s} K .$$
 
The rank of the horoball $\mathcal H(\xi, R)$ of depth $R$ is defined to 
be the rank of a finitely generated abelian subgroup $\G_{\xi}:=\op{Stab}_\G(\xi)$ and is known to be strictly 
smaller than $2\delta$.

 The thick-thin decomposition of the convex core of the geometrically finite manifold $\G\ba \bH^n$ (see \cite{Bo}) implies
 that there exists $R\ge 1$ and a finite set $\{\xi_1, \cdots, \xi_{\ell}\}\subset \Lambda_{\rm bp}(\G)$ of representatives of $\G$-orbits
 in $\Lambda_{\rm bp}(\G)$
 such that  $\Gamma \mathcal H (\xi_i, R)$'s are disjoint and 
 \begin{equation}\label{convexcore}
\text{supp}(m^{\BMS}) \subset \mathcal C_0\sqcup \left(\sqcup_{i=1}^{\ell} \G\ba \G \mathcal H_i \right)\end{equation}
where $\mathcal C_0\subset X$ is compact  and
$\mathcal H_i=\mathcal H(\xi_i, R)$.





The following is a variation of Sullivan's shadow lemma, obtained by Schapira-Maucourant \cite{SM}:
\begin{thm}\label{thm;shadow-lemma} \label{sha}
Let $\Om \subset X$ be a compact subset.
There exists $\lambda=\lambda(\Om)>1$ such that for all $x\in \Om\cap \op{supp}(m^{\BMS})$ with $x^{-} \in \Lambda_{\rm r}(\G)$ and $T>1$,
\[
{\lambda^{-1}}T^\delta e^{(k(x,T)-\delta)\dist(\mathcal C_0,xa_{-\log T})} \leq\mu_x^{\PS}(B_U(T))\leq 
\lambda T^\delta e^{(k(x,T)-\delta)\dist(\mathcal C_0 ,xa_{-\log T})},
\]
where $k(x,T)$ is the rank of $\mathcal H_i$ if $xa_{-\log T}\in \G \mathcal H _i$ for some $i\geq 1,$
and equals $0$ if $xa_{-\log T}\in \mathcal C_0$.
\end{thm}

\subsection{Non-concentration property of PS measures}

\begin{lem}[Window lemma for PS-measure]\label{p;horo-wind} Let $\Om \subset X$ be a compact subset. For any $0<\eta<1$,
there exists $0<r=r(\eta, \Om)<1$  and $T_0>1$ such that
 for all $x\in \Om$ with $x^{-} \in \Lambda_{\rm r}(\G)$ and $T>T_0$,
we have \begin{equation} \label{nc} \mups(B_U(rT))\leq \eta \cdot \mups(B_U(T)).
\end{equation}
\end{lem}

\begin{proof} Let $\mathcal C_0$ be given as \eqref{convexcore}.
We first claim that it suffices to prove the following: 
for any $0<\eta<1$,
there exists $0<r<1$ such that  for all $y\in \mathcal C_0\cap \op{supp}(m^{\BMS})$ with $y^{-} \in \Lambda_{\rm r}(\G)$,
 and $T>1$,
we have \begin{equation} \label{nc} \mu_y^{\PS}(B_U(rT))\leq \eta \cdot \mu_y^{\PS}(B_U(T)).\end{equation}
 
 Without loss of generality, we may assume $\Om$ contains $\mathcal C_0$.
 By Lemma \ref{omr}, there exists $R_0>1$ such that
for all $x\in \Om$, $$xB_U(R_0)\cap (\mathcal C_0\cap \op{supp}(m^{\BMS}))\ne \emptyset.$$ Let $T_0>R_0 r^{-1}$, so that
we have $rT+R_0<(2r)T$ for all $T>T_0$.

If $y=xu_{\tbf}\in xB_U(R_0)\cap \mathcal C_0\cap \op{supp}(m^{\BMS})$, then $x, y\in \Om$. Hence 
by Lemma \ref{uc}  applied to $\Om$  and \eqref{nc},  there exists $c_0>1$ such that
\begin{multline*} \mu_x^{\PS} (B_U(rT))\le  \mu_y^{\PS}(B_U(rT+R_0))\le  \mu_y^{\PS} (B_U(2rT)) \\
\le c_0
 \mu_y^{\PS} (B_U(rT)) \le \eta c_0 
\mu_y^{\PS} (B_U(T)) 
\le \eta c_0 \mu_x^{\PS} (B_U(T+R_0))\\ \le \eta c_0 \mu_x^{\PS} (B_U(2T)) \le  \eta c_0^2 \mu_x^{\PS} (B_U(T)) .\end{multline*}
This proves the claim. Therefore, we need to verify \eqref{nc} only for $x\in \mathcal C_0\cap \op{supp}(m^{\BMS})$ 
with $x^{-} \in \Lambda_{\rm r}(\G)$. 
In particular, $xa_{-\log T}\in \text{supp}(m^{\BMS})$.

Let $p_0:=\max_i \op{rank} (\mathcal H_i)$
and $\lambda =\lambda (\Om)$ be as given in Theorem \ref{sha}.  As remarked before, $2\delta>p_0$.

Set
$$r(\eta):=(\eta \lambda^{-2})^{1+1/(2\delta -p_0)}.$$
 Since $\eta \lambda^{-2}<1$, we have $r(\eta)<\min \{ (\eta \lambda^{-2})^{1/(2\delta -p_0)},  \eta \lambda^{-2}\}$.
 
In view of Theorem \ref{sha}, it suffices to show that $r:=r(\eta)$ satisfies the following:
 $$ \eta \lambda^{-1} T^\delta e^{(k(x,T)-\delta)\dist(\Xc,xa_{-\log T})}
 \ge r^\delta  \lambda T^\delta e^{(k(x, rT) -\delta)\dist(\Xc,xa_{-\log rT})} . $$
or equivalently
\begin{equation}\label{time}\eta \lambda^{-2} e^{(k(x,T)-\delta)\dist(\Xc,xa_{-\log T})} e^{(-k(x, rT) +\delta)\dist(\Xc,xa_{-\log rT})}
 \ge r^\delta   . \end{equation}

From the triangle inequality, we have \begin{align}\label{tri}
\dist(\Xc,xa_{-\log T})-|\log r|&\leq \dist(\Xc,xa_{-\log rT})
\leq \dist(\Xc,xa_{-\log T})+|\log r|.
\end{align}

We prove this by considering two cases:

{{\bf Case 1}: $k(x,T)\ge k(x,rT)$.}  

Then
\begin{align*}& {(k(x,T)-\delta)\dist(\Xc,xa_{-\log T}) -(k(x, rT) -\delta)\dist(\Xc,xa_{-\log rT})} \\ & \ge
(k(x, rT) -\delta) (\dist(\Xc,xa_{-\log T}) -\dist(\Xc,xa_{-\log rT}) )  \\ &\ge
-|k(x, rT) -\delta| \cdot |\log r| .
\end{align*} 
Hence the lefthand side of \eqref{time} is bigger than or equal to $\eta \lambda^{-2} r^{|k(x, rT) -\delta| }$.
Considering two cases $k(x, rT) \le \delta$ and $k(x, rT) > \delta$ separately, it is easy to check that our $r=r(\eta)$
satisfies  $\eta \lambda^{-2} r^{|k(x, rT) -\delta| }\ge r^\delta$, proving \eqref{time}.

{{\bf Case 2}: $k(x,T)< k(x,rT)$.} 

We first consider the case when $k(x, T)=0$,
so that $\dist(\Xc,xa_{-\log T})=0$ and $0< \dist(\Xc,xa_{-\log rT})\le |\log r| $ by \eqref{tri}.
Then the left-hand side of \eqref{time}
becomes
$$ \eta \lambda^{-2}  e^{(-k(x, rT) +\delta) \dist(\mathcal C_0,xa_{-\log rT})}
\ge \eta \lambda^{-2} r^{|k(x, rT) -\delta|}\ge r^{\delta}
$$ as before, proving the inequality \eqref{time}.

We now assume that $k(x, T)\ge 1$. Then $k(x, rT)\ge 2$, and hence $\delta>1$.
In this case, $xa_{-\log T}$ and $xa_{-\log rT}$ are in two distinct horoballs, and hence 
there exists  $r\leq \rho\leq 1$ such that
$xa_{-\log \rho T}\in \mathcal C_0$. We take a maximum such $\rho$.  Then
$$\dist(\Xc,xa_{-\log T})= d (xa_{-\log (\rho T)} , xa_{-\log T}) \le | \log \rho|;$$
 $$\dist(\Xc,xa_{-\log rT})= d (xa_{-\log (\rho T)} , xa_{-\log rT}) \le  \log (\rho r^{-1}) .$$

It follows
\[
 e^{(k(x,T)-\delta)\dist(\Xc,xa_{-\log T})}\geq  e^{(1-\delta)\dist(\Xc,xa_{-\log T})}\ge
\rho^{\delta-1}
\]
and
$$ e^{(k(x,rT)-\delta)\dist(\Xc,xa_{-\log rT})}\le  \max\{1,(\rho/r)^{k(x,rT)-\delta}\}.$$
Therefore
\eqref{time} is reduced to the inequality \be \label{time2}
\eta \lambda^{-2} \rho^{\delta -1} \ge r^{\delta} \max\{1,(\rho/r)^{k(x,rT)-\delta}\}.\ee

If $ \max\{1,(\rho/r)^{k(x,rT)-\delta}\}=1$, since $\rho>r$,
this follows from $\eta \lambda^{-2} r^{\delta -1} \ge r^{\delta}$, which holds, by the definition of $r=r(\eta)$.

It remains to prove that when $(\rho/r)^{k(x,rT)-\delta}\ge 1$, \be\label{e;mups-doubl-case2}
\eta \lambda^{-2} \rho ^{2\delta -k (x,r T) -1} \ge  r^{2\delta-k(x,rT)}.\ee 
 
 By our definition, we have $r(\eta)\le  ({\eta \lambda^{-2}})^{1+1/(2\delta-p_0)}.$
Therefore we have 
\[
\rho^{2\delta-1-k(x,rT)}/r^{2\delta-k(x,rT)}\geq\max\{\rho^{-1}, (\rho/r)^{2\delta-k(x,rT)}\}.
\]
The conclusion now follows 
by taking two cases: $\rho\leq r ({\lambda^2}{\eta^{-1}})^{1/(2\delta-k(x,rT))}$ and alternatively
$r ({\lambda^2}{\eta^{-1}})^{1/(2\delta-k(x,rT))}\leq\rho\leq 1.$ This completes the proof.
\end{proof}




\subsection{Equidistribution for windows}
We now draw the following corollaries of Lemma~\ref{p;horo-wind}, and Theorems \ref{eqbms} and \ref{t;horo-equid-BR}.

For $\psi\in C_c(X)$ and $T>0$,
we define the notation 
$$\mathcal P_T\psi(x)=\int_{B_U(T)}\psi(xu_\tbf)d\mu_x^{\PS}(\tbf);$$

$$\mathcal L_T\psi(x)=\int_{B_U(T)}\psi(xu_\tbf)d\tbf.$$

\begin{thm}[Window lemma for horospherical average] \label{wint} 
Fix a compact subset $\Om \subset X$.
 For any $\eta>0$,
there exists $0<r=r(\eta, \Om)<1$ such that the following holds:
\begin{enumerate}
\item  for  $x\in \Om$ with $x^{-} \in \Lambda_{\rm r} (\G)$ and for any non-negative $\psi\in C_c(X)$ with $m^{\BR}(\psi)>0$,
  there exists $T_0=T_0(x,\psi)$ such that 
\[
 \mathcal L_{rT}\psi(x)\leq \eta \cdot  \mathcal L_{T}\psi(x) \quad \text{ for all $T>T_0$}.
\]
\item
for  $x\in \Om$ with $x^{-} \in \Lambda_{\rm r}(\G)$ and for any non-negative $\psi\in C_c(X)$ with $m^{\BMS}(\psi)>0$,
  there exists $T_0=T_0(x,\psi)$ such that 
\[
 \mathcal P_{rT}\psi(x)\leq \eta \cdot  \mathcal P_{T}\psi(x) \quad \text{ for all $T>T_0$}.
\]
\end{enumerate}

\end{thm}

\begin{proof}
Let $r=r(\eta/4 , \Om)$ be as in Lemma \ref{p;horo-wind}. Let $x\in \Om$ with $x^{-} \in \Lambda_{\rm r}(\G)$.
By Theorem~\ref{t;horo-equid-BR}, there exists $T_0=T(x,\psi)$ so that for all $T>T_0$,
\begin{align*}{\mathcal L}_{rT}\psi(x)  &\le 2 m^{\BR}(\psi) \mu_x^{\PS}(B_U(rT))\\
{\mathcal L}_{T}\psi(x)  &\ge \tfrac{1}{2} m^{\BR}(\psi) \mu_x^{\PS}(B_U(T)) .\end{align*}
Hence
$$ {\mathcal L}_{rT}\psi(x)\le 4 \frac{\mu_x^{\PS}(B_U(rT))}{\mu_x^{\PS}(B_U(T))}{\mathcal L}_T\psi(x)\le
\eta \cdot {\mathcal L}_T\psi(x)$$
by the choice of $r$. This proves (1). (2) is proved similarly using Theorem \ref{eqbms} in place of \ref{t;horo-equid-BR}.
\end{proof}

\begin{thm}[Equidistribution for window averages]\label{eqwi}
For any compact subset $\Om \subset X$,
the following hold:
  for any $x\in \Om$ with $x^-\in \Lambda_{\rm r}(\G)$  and any $\varphi\in C_c(X)$,
  we have\begin{enumerate}
 \item \be\label{e;hopf-mu2}
\lim_{T\to\infty}\frac{\int_{B_U(T)-B_U(rT)} \varphi(xu_{\tbf}  )d\tbf }{\mu_x^{\PS} (B_U(T)-B_U(rT)) }=
{m^{\BR}(\varphi)};\ee
 \item \be\label{e;hopf-mu11}
\lim_{T\to\infty}\frac{\int_{B_U(T)-B_U(rT)} \varphi(xu_{\tbf}  )d\mu_x^{\PS}(\tbf) }{\mu_x^{\PS} (B_U(T)-B_U(rT)) }=
{m^{\BMS}(\varphi)} \ee
where $r=r(1/2, \Om)$ be as in Lemma \ref{p;horo-wind} for $\eta=1/2$.
\end{enumerate}

 \end{thm}
\begin{proof} 
By Theorem \ref{t;horo-equid-BR},
we have 
$${\mathcal L}_T\varphi (x)=m^{\BR}(\varphi)\cdot  \mu_x^{\PS}(B_U(T)) +a_T\quad \text{with } a_T=o(\mu_x^{\PS}(B_U(T))) ;$$
$${\mathcal L}_{rT}\varphi (x)=m^{\BR}(\varphi)\cdot  \mu_x^{\PS}(B_U(rT)) +b_T\quad \text{with } b_T=o(\mu_x^{\PS}(B_U(rT))) .$$
Since $\mu_x^{\PS}(B_U(T) - B_U(rT)) \ge \tfrac 12 \mu_x^{\PS}(B_U(T))  $, it follows that
$|a_T| + |b_T|=o( \mu_x^{\PS}(B_U(T) - B_U(rT)))$. Hence
(1) follows. Similarly (2) can be seen using Theorem \ref{eqbms}.
\end{proof}

\begin{remark}\label{secondremark}
Note that if Theorem \ref{t;horo-equid-BR}
holds for $\psi$ uniformly for all points in a given compact subset $\Om$, then
 Theorem \ref{eqwi}(2) also holds for $\psi$ uniformly for all  $x\in \Om$. This observation will be used later.
\end{remark}

\subsection{Remark on measure classification}
Burger \cite{Bu} classified all locally finite $U$-invariant measures on $\G\ba \PSL_2(\br)$ when $\G$ is convex cocompact with $\delta>1/2$.
Roblin \cite{Roblin2003} extended Burger's work in much
greater generality, and classified all $UM$-invariant ergodic measures on $\G\ba G$ when
$\G$  is a geometrically finite subgroup of a simple Lie group $G$ of rank one. Extending this work, Winter \cite{Wi}
obtained a classification of all $U$-invariant ergodic measures on $\G\ba G$ when $\G$ is also
assumed to be Zariski dense. 
In the case of $G=\SO(n,1)^\circ$ and $\G$ 
geometrically finite, we can also deduce this classification result 
from Theorem \ref{mixing}, using the Hopf ratio theorem. 

First, recall the Hopf ratio theorem proved by Hochman formulated in a setting we are concerned with:
\begin{thm} \cite{Hoc} Let $H$ be a connected Lie group and $\G$ a discrete subgroup.
Let $\br^k=N\subset H$ be a connected abelian subgroup.  Let $\mu$ be a locally finite $N$-invariant 
ergodic measure on $\G\ba G$.
Let $\psi_1,\psi_2\in C_c(\G\ba H)$. Suppose $\psi_2\ge 0$. Then for $\mu$-almost all $x$ such that
$\int_{B_N(\infty)} \psi_2(x u) d\tbf =\infty$,
$$\lim_{T\to \infty} \frac{\int_{B_N(T)} \psi_1(x u) d\tbf }{\int_{B_N(T )} \psi_2(x u) d\tbf }= \frac{\mu(\psi_1)}{\mu(\psi_2)}.$$
\end{thm}

\begin{thm}\label{thm;br}
The only ergodic $U$-invariant measure on $X$ 
which is not supported on a closed orbit of $MU$ is the $\BR$ measure. \end{thm}

\proof
Let $\mu$ be such a measure. 
 Let $\psi\in C_c(X)$ be a non-negative function so that $\mu(\psi)>0$ and $m^{\BR}(\psi)>0.$
Then since the support of $\mu$ is not contained in any closed $MU$-orbit,
there exists $x\in X$ with $x^-\in \Lambda_{\rm r}(\G)$ and
the Hopf ratio ergodic theorem holds: for all $\varphi\in C_c(X)$, we have
\be\label{e;hopf-mu}
\lim_{T\to\infty}\frac{{\mathcal L}_T\varphi(x)}{{\mathcal L}_T\psi(x)}=\frac{\mu(\varphi)}{\mu(\psi)}.
\ee 

Therefore
$$\frac{\mu(\varphi)}{\mu(\psi)}= \frac{m^{\BR}(\varphi)}{m^{\BR}(\psi)}.$$
It follows that $\mu$ and $m^{\BR}$ are proportional to each other.
\qed

\medskip

 We mention that when $G$ is a general simple group of rank one, we expect 
an analogue of Theorem \ref{eqbr}  holds. However
in these cases,
the horospherical subgroup is not abelian any more 
and the Hopf ratio theorem is not available for a general non-abelian nilpotent group action (cf. \cite{Hochman2}).
However  a weaker type of the Hopf ratio theorem  is still available
(see \cite[Theorem 1.4]{Hochman2})
and together with this, it is plausible that 
an analogue of Theorem \ref{eqbr} would
yield an alternative proof for the above mentioned measure classification theorem.

\section{Rigidity of $AU$-equivariant maps}\label{ss;factor}
For the rest of the paper, we let  $\field=\bbr$ or $\bbc$ and
$G=\PSL_2(\field)$.
Let \[
U:=\left\{u_\tbf =\begin{pmatrix} 1& 0 \\ \tbf & 1\end{pmatrix}: \tbf \in \field\right\},
\; {\check U}:=\left\{\umt=\begin{pmatrix}1 & {{\bf r}}\\ 0 & 1\end{pmatrix}: {{\bf r}}\in\field\right\},
\] 
and
\[
A=\left\{ a_s=\begin{pmatrix} e^{s/2}& 0 \\ 0 & e^{-s/2} \end{pmatrix}: s\in \br \right\}.
\]

Let $\G_1$ and $ \G_2$ be geometrically finite,
and  Zariski dense subgroups of $G$ and set for each $i=1,2$
$$X_i:=\G_i\ba G .$$
We denote 
 $m^{\BMS}_{\G_i}$  the BMS-measure on $X_i$
associated to $\G_i$ for each $i=1,2$.  We assume that
$|m^{\BMS}_{\G_1}|=|m^{\BMS}_{\G_2}|=1$. When there is no room for confusion, we will omit the subscript $\G_i$
from the notation of these measures.

Suppose $$\msec_1,\ldots,\msec_\ell: X_1\to X_2$$
are Borel measurable maps 
and consider a set-valued map:
\be\label{e;mfib-bms}
 \mfib(x)=\{\msec_1(x),\ldots,\msec_\ell(x)\}.
\ee
 We assume that $\mfib$ is $U$-equivariant in the sense that
there exists a $U$-invariant Borel subset $X'\subset X_1$ with
 $m^{\BMS}(X')=1$ such that for all $x\in X'$ and {\tt every} ${ u_\tbf}\in U$,
we have \be\label{Ueq} \mfib(x{u_\tbf})=\mfib(x){u_\tbf}.\ee

The main aim of this section is to prove Theorem \ref{t;factor-rigidity-bms} that if $\mfib$ is $AU$-equivariant on a $\BMS$-conull set,
it is also ${\cont}$-equivariant. 
\begin{thm}\label{t;factor-rigidity-bms}
 Suppose that for all $x\in X'$ and {\tt every} ${a u_\tbf}\in AU$,
we have $$\mfib(x{au_\tbf})=\mfib(x){au_\tbf}.$$ 
 Then there exists a $\BMS$-conull subset $X''\subset X'$ such that for all $x\in X''$ and  
 for {\tt every} ${ \umt}\in {\cont}$ with $x \umt\in X''$, we have 
\[
\mfib(x\umt)=\mfib(x)\umt. 
\]
\end{thm}


This is proved in~\cite{FS-Factor} for the case
$\G$ is convex cocompact and $\ell=1$;
 the proof is based on Ratner's proof of the rigidity of $U$-factors \cite{Rat-Horocycle} in the lattice case.
Here we use similar strategy and generalize this
to the case of a geometrically finite group allowing also $\ell\geq1.$ The presence of cusps
requires extra care in this extension.

Let us recall that following terminology from~\cite{KM-Recurrence}.
Let $C,\alpha>0$ and we denote by $| \cdot |$ the absolute value of
$\field$.
A function $f:\field^n\to\field$ is said to be {\it $(C,\alpha)$-good
on a ball $B$} if the following holds: for any ball $V\subset B$
and any $\e>0$ we have  
\be\label{eq;good-func}
\ell \{x\in V: |f(x) |< \e\} \leq C\left(\frac{\e}{\sup_{V} |f |}\right)^{\alpha} \ell ({V})
\ee
where $\ell $ denotes the Lebesgue measure on $\field^n$.
It follows from Lagrange's interpolation and induction
that if $f$ is a polynomial in $n$ variables and of degree bounded by $d,$
then $f$ is $(C,\alpha)$-good on $\field^n$ where $C$ and $\alpha$ depend
only on $n$ and $d.$

The $(C,\alpha)$-good property for fractal measures was studied in~\cite{KLW-Fractal}. 
We need the following lemma (a version of this is \cite[Lemma 5.1]{FS-Factor} for $\G$ convex cocompact);
our proof is soft and uses compactness arguments.
This can be thought of as a weak form of the $(C,\alpha)$-good property of polynomials. Recall a point $x\in X_1$ is
called a $\BMS$-point (resp.~$\BR$-point) if it lies in the support of $m^{\BMS}$ (resp. $m^{\BR}$).

For $d\ge 0$ and $\ell\ge 1$, let $\pcal_{d,\ell}$ be the set of functions $\Theta:U\to\br $ of the form 
\[
\Theta({\tbf}):=\min\{|\Theta_1({\tbf})|^2,\ldots, |\Theta_{\ell}({\tbf})|^2\}
\] 
where the function
$\Theta_i:U\to\mathbb F$ is a polynomial of degree at most $d$ for each $1\leq i\leq\ell$.

\begin{lem}\label{l;poly-ps-good}\label{cone}
Let $d,\ell>0$ be fixed. For any compact subset $\kcal\subset X_1$,
there exists some $C_1=C_1(\kcal, d,\ell)>0$ depending on $d,\ell$ and $\kcal$ with the following properties:
 Let $x\in\kcal$ be a $\BMS$ point with $x^-\in \Lambda_{\rm r}(\G)$ 
and let $s\ge 0$ be so that $xa_{-s}\in\kcal.$  
Then for any $\Theta\in\pcal_{d,\ell}$, we have 
\[
\frac{1}{\mups(B_U(e^s))}\int_{B_U(e^s)} \Theta({\tbf})d\mups(\tbf)\geq
 C_1\cdot \mbox{$\sup_{\tbf\in B_U(e^s)}  \Theta (\tbf) $}.
\]
\end{lem}

\begin{proof}
Write $\kcal^{\BMS}=\kcal\cap\supp(m^{\BMS})$. 
Note that the above statement  is invariant
under scaling the map $\Theta.$
Further, for any $x\in \kcal^{\BMS},$ any $s\in\bbr$ and all $\Theta\in\pcal_{d,\ell}$, 
we have
\begin{align*}
 \frac{1}{\mu_{x}^{\PS}(B_U(e^{s}))}\int_{B_U(e^{s})} \Theta({\tbf}) d\mu_{x}^{\PS}(\tbf)&=
 \frac{1}{\mu_{x}^{\PS}(B_U(e^{s}))}\int_{B_U(1)} \Theta(e^{s}\tbf) d\mu_{x}^{\PS}(e^{s}\tbf)\\
&=\frac{1}{\mu_{xa_{-s}}^{\PS}(B_U(1))}\int_{B_U(1)}
 \tilde{\Theta}(\tbf) d\mu_{xa_{-s}}^{\PS}(\tbf)
\end{align*}
where  $\tilde{\Theta}(\tbf):=\Theta(e^s \tbf)$.
Suppose now the statement (1) fails. Then we have 
\begin{itemize}
\item a sequence $x_i\in\kcal^{\BMS}$, a sequence $s_i\to \infty $ such that $y_i:=x_ia_{-s_i}\in\kcal,$
\item a sequence $\tilde \Theta_i\in\pcal_{d, \ell}$ with $\sup_{B_U(1)}\tilde \Theta_i(\tbf)=1$
\end{itemize}
so that $\frac{1}{\mu_{y_i}^{\PS}(B_U(1))}\int_{B_U(1)}\tilde  \Theta_i({\tbf})d\mu_{y_i}^{\PS}(\tbf)\to0$ as $i\to \infty$.

Passing to a subsequence we may assume that
$y_i\to y\in\kcal^{\BMS}$ and $\tilde \Theta_i\to \tilde \Theta\in\pcal_{d,\ell}$ 
with $\sup_{\tbf\in B_U(1)} \tilde \Theta(\tbf) =1.$
Since the map $x\mapsto\mups$ is continuous on $\kcal^{\BMS}$
and 
\[
\mbox{$0< \inf_{x\in \kcal^{\BMS}} \mups(B_U(1)) \le \sup_{x\in \kcal^{\BMS}} \mups(B_U(1))< \infty $,}
\]
it follows that
\[
 \int_{B_U(1)}\tilde \Theta(\tbf) d\mu_y^{\PS}(\tbf)=0.
\]
This implies that $\mu_y^{\PS}(B_U(1)\cap\{\tbf: \tilde \Theta(\tbf)\neq0\})=0$
which contradicts the fact that $y\in\supp(m^{\BMS})$ in view of
Zariski density of $\G$, proving the claim.
\end{proof}


We also recall the following mean ergodic theorem.

\begin{thm} \cite[Thm. 17]{Rudol} \label{thm;Rudol}
For any Borel set $\mathcal K$ of $X_1$ and any $\eta>0$, the set 
$$
\left\{x\in X_1:\liminf_T \tfrac{1}{\mu_x^{\PS}(B_U(T))}\int_{B_U(T)}
\chi_{\mathcal K} (xu_\tbf)d\mu_x^{\PS}(\tbf)\geq (1-\eta)m^{\BMS}( \mathcal K )\right\}
$$
has full $\BMS$ measure.
\end{thm}

In the following, let $X'$ be the set which satisfies \eqref{Ueq} for $\mfib$.
Fix $\eta>0$. By Lusin's theorem, there exists a compact subset 
\be\label{uniformcc} \kcal_\eta \subset X'\cap \supp(m^{\BMS}) \ee  with $m^{\BMS}(\kcal_\eta )> 1-\eta$ so that
$\msec_i$ is uniformly continuous on $\kcal_\eta$ for each $1\leq i\leq\ell $.
Since $\{x\in X_1 : x^-\in \Lambda_{\rm r}(\G_1)\}$ has a full $\BMS$-measure,
it follows from  Theorem~\ref{thm;Rudol} that
there exist a compact subset 
\be \label{odef} \Om_\eta\subset  \{x\in X' \cap \supp(m^{\BMS}): x^-\in \Lambda_{\rm r}(\G_1)\}\ee with $m^{\BMS}(\Om_\eta)>1-\eta$
and $T_\eta>1$ such that
 for any $x\in \Om_\eta$ and $T\ge T_\eta$,
\begin{align}
\label{ruu}\frac{1}{\mu_x^{\PS}(B_U(T))}\int_{B_U(T)}\chi_{\kcal_\eta}(xu_\tbf)
d\mu_x^{\PS}(\tbf)&\geq (1- \eta)m^{\BMS}(\kcal_\eta)\\
\notag &\ge 1-2\eta.
\end{align}

The following is a key ingredient of the proof of Theorem \ref{t;factor-rigidity-bms}.

\begin{prop}\label{keyP} Under the hypothesis of Theorem \ref{t;factor-rigidity-bms}, there exists a compact subset
$\Om \subset\{x\in X' :x^-\in \Lambda_{\rm r}(\G)\}$ with 
 $m^{\BMS}(\Om)>0.9$,
  $\e'>0$ and $s_0>0$ such that
for any $|\bf r|<\e'$, for any  $s>s_0$ and $x\in\Om $ such that
 $x\check u_{\bf r}, x\check u_{\bf r} a_s, xa_s \in \Om$,
\be\label{cccdmain}
\mfib(x\umt)\check u_{-{{\bf r}}}\subset \mfib(x)\cdot\{g\in G: d(e,g)\leq c\cdot e^{-s}\}.\ee
 where $c>1$ is an absolute constant.
\end{prop}

\begin{proof} Fix a small $\eta>0$. Let $K_\eta$ and $\Om_\eta$ be as in \eqref{uniformcc} and \eqref{odef}.
Fix a small $\e>0$. Let $0<\e'<\e/2$ be such that
\be \label{Kc} \dist(\msec_i(x),\msec_i(x'))\leq\e \ee for all $1\leq i\leq\ell$ and all $x,x'\in\kcal_\eta$ 
with $\dist(x,x')\leq\e'.$ 
Fix  $x\in \Om_\eta$,  $|\bf r|<\e'$ and $s>0$ such that
$xa_s , x\check u_{\bf r}, x\check u_{\bf r} a_s\in \Om_\eta$.

We will first explain the idea of proofs assuming $\mfib$ is an actual map, i.e., $\ell=1$.
Writing $$\mfib(x\umt )\check u_{-\bf r} =\mfib (x) h_s\quad\text{$h_s\in G$,}$$
we would like to show that
\be \label{hs} d(e, h_s)=O(e^{-s}). \ee

Set $g_s:= a_{-s}h_s a_s$ so that
$$  \mfib(x\umt )\check u_{-\bf r} a_s = \mfib (x)a_s g_s=  \mfib (x a_s) g_s .$$
As  $\Upsilon$ is $A$-equivariant on $\Om_\eta$, we have
$\mfib(x\umt )\check u_{-\bf r} a_s=\mfib(x\umt  a_s)\check u_{-e^{-s} \bf r} $
and hence 
\be \label{a1} \mfib(x\umt  a_s)\check u_{-e^{-s} \bf r}= \mfib (xa_s) g_s .\ee

We will study the divergence of the points $\mfib(x\umt  a_s)u_{-e^{-s} \bf r}$ and $\mfib (x)a_s$
 along $u_\tbf$ flow
  and
show that for all $s$ large and for all $\tbf \in B_U(e^s)$, 
\[
\mbox{$\mfib (x\umt a_s)\check u_{-e^{-s}{{\bf r}}} u_\tbf$  and $\mfib (xa_s) u_\tbf$}
\] 
stay within bounded distance from each other, that is, $u_{-\tbf} g_s u_{\tbf}$
is uniformly bounded for all $\tbf \in B_U(e^s)$.
 This will imply that the element $g_s$ is {\it close} to the centralizer of $U$, which is $U$ itself.
 We will then be able to conclude that the element $h_s=a_{s}g_s a_{-s}$ is of size $O(e^{-s})$.

For all $|\tbf|\leq e^{s}$,  set
\[
\beta_{{\bf r}}(\tbf)=\tfrac{{\bf t}}{1+e^{-s}{{\bf r}} \tbf}.
\]
Then we have
\be\label{e;U-flow-close}
\check u_{e^{-s}{{\bf r}}}{u_\tbf}
=\begin{pmatrix}1+e^{-s}{{\bf r}} \tbf & e^{-s}{{\bf r}}\\ 
 \tbf & 1\end{pmatrix}=u_{\beta_{\bf r} (\tbf)}g_{{\bf r}}
\ee
where 
 $g_{\bf r} \in P=AM{\cont}$ with $d(e , g_{{\bf r}}) \leq \e'.$ It is worth mentioning 
 that the function $\beta_{{\bf r}}(\tbf)$ above is our time change map, and is responsible 
 for the two aforementioned properties of $g_{\bf r}.$

\noindent{\bf Step 1:}
We claim that for all $|\tbf|\leq e^s$ such that
$xa_su_{\tbf},x\umt a_s{u_{\beta_{-{{\bf r}}}(\tbf )}}\in\kcal_\eta$, and for any $i$,
\be\label{e;orbits-stay-close}
\dist(\msec_i(x \umt)\check u_{-\bf r} a_s u_{\bf t},\mfib(xa_s)u_{\tbf})\leq\e .
\ee


Using the $U$-equivariance, we have for any $\tbf\in\field$ and $i$,
\begin{align*}
& \msec_i(x \umt)\check u_{-\bf r} a_s u_{\bf t} \\
 &\in \mfib(x\umt a_s)\check u_{-e^{-s}{{\bf r}}}{u_\tbf}\quad \text{using $x\check u_{\bf r} a_s= x a_s\check u_{e^{-s}\bf r}$}
\\
 &=\mfib(x\umt a_s)u_{\beta_{-{{\bf r}}}(\tbf )}g_{-{{\bf r}}}
 &&\text{ by~\eqref{e;U-flow-close}}\\
  &=\mfib(x\umt a_su_{\beta_{-{{\bf r}}}(\tbf )}) g_{-{{\bf r}}}\\
 &=\mfib(xa_s u_{\tbf} g_{-{{\bf r}}}^{-1})g_{-{{\bf r}}}
 &&\text{ by~\eqref{e;U-flow-close}};
\end{align*}
where we used the identities:
$$\umt a_su_{\beta_{-{{\bf r}}}(\tbf )}=a_s \check u_{-e^{-s} {\bf r}}
  u_{\beta_{-{{\bf r}}}(\tbf )}
=a_s u_{\tbf} g_{-{{\bf r}}}^{-1}$$ by~\eqref{e;U-flow-close}.
The choice of $\e'$ made in \eqref{Kc} 
implies that for all $|\tbf|\leq e^s$ such that
$xa_su_{\tbf},xa_s u_{\tbf} g_{-{{\bf r}}}^{-1}= x\umt a_s{u_{\beta_{-{{\bf r}}}(\tbf )}}\in\kcal_\eta$, we have
\[
\dist(\msec_i(x \umt)\check u_{-\bf r} a_s u_{\bf t},\mfib(xa_s{u_\tbf}))\leq2\e .
\]
 This implies 
\eqref{e;orbits-stay-close} 
 by the $U$-equivariance of $\Upsilon$.

\noindent{\bf Step 2:} 
Letting $$\kcal_x(s,{{\bf r}}):=\{\tbf\in B_U(e^s): xa_su_{\tbf},x\umt a_su_{\beta_{-{{\bf r}}}(\tbf )}\in\kcal_\eta\}$$
and $s_0=\log T_\eta$,
we claim that  for all $s>s_0$,
\be\label{ks} \mu_{xa_s}^{\PS}
(\kcal_x(s,{{\bf r}}))\geq (1- c_0\eta)\mu_{xa_s}^{\PS}(B_U(e^s))
\ee where $c_0>0$ is an absolute constant.

Since $xa_s,x\umt a_s\in\Om_\eta$,
 we get from \eqref{ruu} that if $s>s_0$,
\be\label{c33} \mu^{\PS}_{xa_s}\{\tbf\in B_U(e^s): xa_su_{\tbf} \in\kcal_\eta\} \geq (1-2\eta)\mu_{xa_s}^{\PS}(B_U(e^s))
\ee
and 
\be\label{c3} 
{\mu^{\PS}_{x\check u_{\bf r} a_s}
\{\tbf\in B_U(e^s): x\umt a_su_{\tbf} \in\kcal\} \geq (1-2 \eta)\mu_{x\check u_{\bf r} a_s}^{\PS}(B_U(e^s)).}
\ee
Note that $\op{Jac}(\beta_{\bf r})(\tbf) =1+O(\e')$ for all $\tbf\in B_U(e^s)$, and  hence \eqref{c3}
implies 
\begin{align*} 
&\mu^{\PS}_{x\check u_{\bf r} a_s}\{\tbf\in B_U(e^s): x\umt a_su_{\beta_{-\bf r} (\tbf)} \notin\kcal\}
\\  &\leq \mu^{\PS}_{x\check u_{\bf r} a_s} \{\tbf \in B_U(e^s + O(\e ') ): x\umt a_su_\tbf  \notin\kcal\} \\ &\le
 2\eta \cdot \mu_{x\check u_{\bf r} a_s}^{\PS}(B_U(e^s +O(\e' )) \\ &\le
 (2c_1\eta)  \cdot \mu_{x\check u_{\bf r} a_s}^{\PS}(B_U(e^s))  
 \end{align*}
where $c_1$ is given by Lemma \ref{uc} for $\supp (m^{\BMS})$.
This is equivalent to saying that
\be\label{c4} 
{\mu^{\PS}_{x a_s\check u_{e^{-s}\bf r}}
\{\tbf\in B_U(e^s): x\umt a_su_{\beta_{-\bf r} (\tbf)} \notin
\kcal\} \leq  (2c_1 \eta) \cdot \mu_{xa_s\check u_{e^{-s}\bf r}}^{\PS}(B_U(e^s))} .
\ee

Note that $( x a_s \check u_{e^{-s}\bf r} u_{\beta_{-\bf r} (\tbf)})^+= (xa_s u_\tbf)^+$ 
since $g_{-\bf r}\in P$.
It follows from the definition of the $\PS$-measure
$d\mu_y^{\PS}(\tbf)=e^{\delta\beta_{(yu_\tbf) ^+}(o, yu_\tbf(o))} d\nu_o((yu_\tbf)^+)$ and 
the fact $|e^{-s}{\bf r}|=O(e^{-s} \e')$ that on $B_U(e^s)$,
\[
d\mu_{x\check u_{\bf r} a_s}^{\PS}({\beta_{-\bf r} (\tbf)})= (1+ O(e^{-s}\e')) 
d\mu_{xa_s}^{\PS}(\tbf).
\]

Therefore \eqref{c4} implies that
\be \label{c5} \mu^{\PS}_{x a_s}
\{\tbf\in B_U(e^s): x\umt a_su_{\beta_{-\bf r} (\tbf)} \notin
\kcal_\eta \}  \leq (c_2 \eta )\cdot \mu_{xa_s}^{\PS}(B_U(e^s)) \ee
for some absolute constant $c_2>0$. Hence this together with \eqref{c33} implies
the claim \eqref{ks}.

\noindent{\bf Step 3:} If $\eta>0$ and $\e'>0$ are
 sufficiently small, then for each $i$, there exists $k(i)$ such that
 for all $s>s_0$,
 \be\label{e;theta-small}
\sup_{\tbf\in B_U(e^s)}\dist(\msec_i(x \umt)\check u_{-\bf r} a_s u_{\bf t} ,\msec_{k(i)} (xa_s)u_{\tbf})
\leq 1 .
\ee 

Put $$\Theta(\tbf)=
\min\{\dist(\msec_i(x \umt)\check u_{-\bf r} a_s u_{\bf t},\mfib(xa_s)u_{\tbf})^2,1\}.$$ Then
for any  $s>s_0$, \begin{align*}&
\frac{1}{\mu_{xa_s}^{\PS}(B_U(e^s))}\int_{B_U(e^s)}\Theta(\tbf)d\mu^{\PS}_{xa_s}(\tbf)\\&\leq
c_0\eta+\frac{1}{\mu_{xa_s}^{\PS}(B_U(e^s))}\int_{\kcal_x(s,{{\bf r}})}\Theta(\tbf)d\mu_{xa_s}^{\PS}(\tbf)&&
\text{ by~\eqref{ks}}\\
&\leq c_0\eta+\e &&\text{ by~\eqref{e;orbits-stay-close}}.
\end{align*}

Recall now that if $y,z\in X$ are two point then
for all $\tbf\in \field$ so that $\dist(yu_\tbf,zu_\tbf)$ is sufficiently small,
 the map $\tbf\mapsto \dist(yu_\tbf,zu_\tbf)^2$ is 
governed by a polynomial of bounded degree, see~\cite{Rat-Horocycle}.
 By \eqref{e;orbits-stay-close}, for any sufficiently small $\e>0$,
  $\Theta_{x,s}\in\mathcal P_{d,\ell}$ where
$d$ depends only on the dimension. 
If we set $y:=xa_s\in \Omega_\eta$, then  $ya_{-s}=x\in \Omega_\eta$. Hence
 applying Lemma~\ref{l;poly-ps-good} for $y=xa_s$, we obtain that
\[
\sup_{|\tbf|\leq e^s} \Theta(\tbf)\leq 2 (c_0\eta+\e) /C_1:=\zeta 
\] 
where $C_1=C_1(\Omega_\eta, d,\ell)$.
 If $\eta$ and $\e$ are sufficiently small so that $2 (c_0\eta+\e) /C_1 <1$,
then  $$\dist(\msec_i(x \umt)\check u_{-\bf r} a_s u_{\bf t},\mfib(xa_s)u_{\tbf}) <\zeta^{1/2}. $$

It follows that for each $i$, there exist $1\leq k(i)\leq\ell,$ and a subset $J(s)\subset B_U(e^s)$
with $\ell (J(s)) \geq\frac{1}{\ell}  \ell (B_U(e^s)) $ so that
if we set $$\Theta_{i}(\tbf):=\dist(\msec_i(x \umt)\check u_{-\bf r} a_s u_{\bf t},,\msec_{k(i)}(xa_s)u_{\tbf})^2,$$
then
$$\sup_{\tbf\in J(s)}\Theta_{i}(\tbf)  \leq \zeta
. $$
Therefore the above and the fact that polynomials of a bounded degree are $(C,\alpha)$-good
with respect to the Haar measure, see~\eqref{eq;good-func}, imply that
\[ 
\frac{1}{\ell} \ell (B_U(e^s))  \le
\ell (J(s) )\le  C 4^{-\alpha} (\sup_{|\tbf|\leq e^s} \Theta_{i})^{-\alpha} \ell ( B_U(e^s) ). 
\]

Hence $$\sup_{|\tbf|\leq e^s} \Theta_{i}\leq \zeta \cdot C^{1/\alpha}  \ell=  2C^{1/\alpha}  \ell
 (c_0\eta+\e) /C_1  .$$
Therefore
 \eqref{e;theta-small} holds for $\eta$ and $\e$ sufficiently small.

\noindent{\bf Step 4: } 
For every $1\leq i\leq \ell$, 
define $g_{s,i}\in G$ by the following
\[
\msec_{i}(x\umt a_s)=\msec_{k(i)}(xa_s)g_{s,i} .
\] 

We claim $$d( e, a_s g_{s,i} a_{-s})=O(e^{-s}).$$

The equation ~\eqref{e;theta-small} in particular
 implies that $g_{s,i}$ is contained in an $O(1 )$ neighborhood of the identity.

We further investigate the element $g_{s,i}.$ 
Write $g_{s,i}=\begin{pmatrix} x_s & y_s \\ z_s & w_s\end{pmatrix}$, so that
\[
u_{-\tbf}g_{s,i}{u_\tbf}=
\begin{pmatrix}x_s+y_s\tbf & y_s \\ z_s+(w_s-x_s)\tbf-y_s\tbf^2 & w_s-y_s\tbf\end{pmatrix}.
\]
Therefore,~\eqref{e;theta-small} and the fact that $\det g_{s,i}=1$ imply 
\[
|z_s|= O(1),\; |1-x_s|= O(e^{-s}),\; |1-w_s| = O(e^{-s}),\; |y_s| = O(e^{-2s}).
\]
This implies
\be\label{e;gs-constrain}
d(e, a_sg_{s,i}a_{-s}) = O(e^{-s}).
\ee This proves \eqref{cccdmain} for $s_0$ (given in Step 2)
and $\Omega=\Omega_\eta$, $\e'>0$ 
 for sufficiently small $\eta, \e'>0$ (given in Step 3).
\end{proof}




\noindent{\bf Proof of Theorem \ref{t;factor-rigidity-bms}}

Let $\Om $ and $\e'>0$ be as in Proposition \ref{keyP}.
By the ergodicity of  the $A$-flow for the BMS measure, which follows from
Theorem \ref{fm}, and the Birkhoff ergodic theorem, there exists a conull subset $X''$ of $X'$ such that
for all $x\in X'' A$, 
\be \label{birk} \lim_{S\to \infty} \frac{1}{S}\int_0^S\chi_{\Om}(xa_s)ds=m^{\BMS}(\Om) >0.9\ee

Let $x\in X''$ and $\umt\in {\cont}$ such that $x\umt\in X''$.
We will show $$\Upsilon (x \umt)=\Upsilon (x)\umt .$$

By \eqref{birk},
we can choose arbitrarily large $s_0>1$ such that $x a_{s_0}, x\umt a_{s_0} \in \Om $, and
$e^{-s_0}{\bf r}$ is of size at most $\e'$.
Setting $x_0:=xa_{s_0}$ and ${\bf r}_0:= e^{-s_0} \bf r$, it suffices to show that
$\Upsilon (x_0 \check u_{{\bf r}_0})=\Upsilon (x_0)\check u_{{\bf r}_0}$, thanks to the
$A$-equivariance. Therefore,
we may assume without loss of generality that
 $$x\in X'' A\cap \Om , \;\; x\umt\in X'' A\cap \Om \;\;\text{and} \;\; |{\bf r}|\le \e' .$$

By \eqref{birk}, we have a sequence $s_m\to + \infty$ so that 
$x\umt a_{s_m}, xa_{s_m} \in \Om$ for all $m$.
By Proposition \ref{keyP}, we have
\be\label{cccd}
\mfib(x\umt)\check u_{-{{\bf r}}}\subset \mfib(x)\cdot\{g\in G: d(e,g)\leq c e^{-s_m}\}.
\ee
Hence
$\mfib(x\umt)\check u_{-{{\bf r}}}= \mfib(x)$, proving the claim.

\section{Joining classification}\label{s;Joining-class}
We let $G=\PSL_2(\field)$ for $\field=\br, \c$.
Let $\G_1$ and $ \G_2$ be geometrically finite,
and  Zariski dense subgroups of $G$.
Set $X_i:=\G_i\ba G $ for $i=1,2$ and
$$Z:=X_1\times X_2.$$

We keep the notations $U$, $A, {\cont}$, etc. from section \ref{ss;factor}. Let 
\be 
M=\begin{cases} 
\{e\} &\text{ for $G=\PSL_2(\br)$}\\ 
\left\{\begin{pmatrix} e^{i\theta}& 0 \\ 0 & e^{-i\theta} \end{pmatrix}:\theta\in \br \right\} 
&\text{ for $G=\PSL_2(\bbc)$.}
\end{cases}
\ee
Since $M$ is considered as a subgroup of $\PSL_2(\c)$, two elements which differ by $-1$ are identified.
We set $P=MA\cont$.

We will
use the notation $\Delta$ for the diagonal embedding map of $G$ into $G\times G$; so
$\Delta(g)=(g, g)$ for $g\in G$.
For $\tbf\in \field$, $|\tbf|$ denotes the absolute value of $\tbf$ and
 for $T>0$, we set
 $$B_U(T)=\{u_{\tbf}\in U: |\tbf|\le T\}.$$
\subsection{Construction of a polynomial-like map}\label{ss;quasi-reg}
We fix a rational cross-section $\lcal$ for $\dU$ in $G\times G$ as follows:
\[
\lcal=(\{e\}\times G)\cdot \Delta(P)=P\times G.\] 
Then $\lcal\cap\Delta(U)=\{e\}$ and the product map from 
$\lcal\times\Delta(U)$ to $G$ defines a diffeomorphism onto 
a Zariski open dense subset of $G\times G.$
We will use $\lcal$ as the transversal direction to $\Delta(U)$
in $G\times G.$ 

We observe that $N_{G\times G} (\Delta(U))=\Delta(AM )\cdot (U\times U)$ and
\[
N_{G\times G}(\dU)\cap\lcal=\Delta( AM)\cdot (\{e\}\times U)
\]

Suppose that we are given a sequence $h_k=(h_k^1, h_k^2) \in G\times G$ such that
$h_k\not\in N_{G\times G}(\bigger)$ with $h_k\to e$ as $k\to \infty$. 
Associated to $\{h_k\}$,
 we will construct a quasi-regular map 
\[
\varphi:\bigger\to\Delta( AM)\cdot (\{e\}\times U)
\]
following~\cite[Section 5]{MT}. 
Via the identification $\field =\Delta(U)$ given by $\tbf \mapsto \Delta(u_\tbf)$, we will define
the map $\varphi$ on $\field$, which will save us some notation. Accordingly,
we will write $\tbf\in B_U(1)$ to mean
that $u_\tbf\in B_U(1)$, etc.

For $g= \begin{pmatrix}a & b\\ c & d\end{pmatrix}\in G$ and for $\tbf\ne -ab^{-1}$,
 define $$\alpha_g(\tbf)=\frac{c+d\tbf}{a+b\tbf}. $$ 
 We denote the pole of $\alpha_g$ by $R(g)$ and put $R(g)=\infty$
if $\alpha_g$ is defined everywhere. That is, $R(g)=-ab^{-1}$ if $b\ne 0$ and $\infty$ otherwise.

Set $$\alpha_k(\tbf):=\alpha_{h_k^1}(\tbf)\text{ and } R_k=R(h_k^1) \text{ for each $k$.}$$
Note that $R_k\to\infty$ as $k\to\infty.$
A direct computation shows that for all $\tbf\in \field -\{R_k\}$, we can write 
$h_k\Delta(u_\tbf) $ as an element of $\Delta(U)\mathcal L$ where the $\Delta(U)$ component is
given by $\Delta(u_{\alpha_k(\tbf)})$. We will denote by $\varphi_k(\tbf)$ for its transversal component so that
\[
h_k\Delta(u_\tbf) =\Delta(u_{\alpha_k(\tbf)})\varphi_{k}(\tbf) \in \Delta (U) \lcal . \]

We renormalize these maps $\varphi_k$'s using a representation 
corresponding to $\bigger.$ Recall that by a theorem of Chevalley 
there is a finite dimensional representation $(\rho,W)$ of $G\times G$,
where $G\times G$ acts from the right on $W$
and a unit vector $\qpz\in W$ so that
\[
 \bigger=\{h\in G\times G: \qpz\rho(h)=\qpz\}.
\]
Then $$N_{G\times G}(\bigger)=\{h: \qpz\rho(h)\rho(\biggere)=\qpz\rho(h)
\text{ for all $u_\tbf \in U$}\}.$$
We choose a norm $\|\cdot\|$ on $\Upsilon$ so that 
$B(\qpz,2)\cap\overline{\qpz\rho(G\times G)}\subset \qpz\rho(G\times G).$ 

Now for each $k$, define $\tilde{\phi}_k:\mathbb F \to W$ by 
$$\tilde{\phi}_k(\tbf) =\qpz\rho(h_k\biggere);$$
 $\tilde{\phi}_k$ is a polynomial map of degree bounded in terms of $\rho$ and
$\tilde{\phi}_k(0)=\qpz$.

\medskip

\noindent{\bf Explicit construction of $\tilde \phi_k$:} Consider the following representation: let $G\times G$ act on
$W=\field^2 \oplus
\field^2\oplus M_2(\field)$ by
$$(g_1,g_2). (v_1, v_2, Q) =(v_1 g_1, v_2 g_2, g_1^{-1} Q g_2).$$ 
Then the stabilizer of $\qpz:=(e_1, e_1, I_2)$ is precisely $\Delta(U)$. If we 
write $h_k^i=\begin{pmatrix} a_k^i & b_k^i \\ c_k^i& d_k^i\end{pmatrix}$,
 then, up to an additive constant vector, say, $\qpz_k\in W$
we have 
$$\tilde \phi_k (\tbf) =\qpz_k+\left( b_k^1\tbf, 0, b_k^2\tbf, 0, \begin{pmatrix} -\mathcal A_k \tbf  & 0\\
 \mathcal A_k \tbf^2+
\mathcal B_k\tbf  & \mathcal A_k \tbf  \end{pmatrix} \right)
$$ where $\mathcal A_k=b_k^1d_k^2-b_k^2d_k^1$ and $\mathcal B_k =a_k^1d_k^2+b_k^1c_k^2 -b_k^2c_k^1-a_k^2d_k^1$.
Hence  $\tilde{\phi}_k$ is a polynomial of degree at most $2$.

Let $T_k>0 $ be the infimum of $T>0$ such that $$\sup_{\tbf \in B_U(T)} \|\tilde{\phi}_k(\tbf) -\qpz \|=1.$$
Since $h_k\not\in N_{G\times G}(\dU)$, $\tilde{\phi}_k$ is a non-constant polynomial and hence
 we get $T_k\neq\infty,$
moreover, in view of our assumption $h_k\to e$ we have $T_k\to\infty$
as $k\to\infty.$ 

By normalizing $\tilde \phi_k$ by
 $$\phi_k(\tbf):=\tilde{\phi}_k(T_k \tbf),$$  
we obtain a sequence of equicontinuous polynomials $\phi_k$. Hence,
after passing to a subsequence, $\phi_k$ converges to $\phi$ where 
\begin{itemize}
 \item $\tbf \mapsto \phi (\tbf) $ is a non-constant polynomial of degree at most $2$.
 \item $\sup_{\tbf\in B_U(1)}\|\phi(\tbf ) -\qpz\|=1$ and $\phi(0)=\qpz,$ 
 \item $\{\phi(\tbf):\tbf\in B_U(1)\}\subset\qpz\rho(G\times G)$
 \item the convergence is uniform on compact subsets of $\field.$ 
\end{itemize}
Put $$\varphi=(\rho_{\lcal})^{-1}\circ\phi$$ where $\rho_\lcal$ is the restriction to $\lcal$ of the orbit map
$g\mapsto \qpz \rho(g) $.
Then $\varphi:\field \to\lcal$ is a rational map defined on a Zariski 
open dense subset $\mathcal{O}\subset\field$ containing zero
and $\varphi(0)=e.$ We have
\[
 \mbox{$\varphi(\tbf)=\lim_{k}\varphi_k({T_k\tbf})$}
\]
and the convergence is uniform on compact subsets of $\mathcal{O}.$

Note also that for any ${\tbf_0}\in \field$, we have 
\begin{align*}&
 \phi(\tbf )\rho(\Delta(u_{\tbf_0}))\\ &=\lim_k \tilde{\phi}_k({T_k\tbf})\rho(\Delta(u_{\tbf_0}))\\
 &=\lim_{k}\qpz\rho(h_k\Delta(u_{T_k\tbf}))\rho(\Delta(u_{\tbf_0}))\\ &=
 \lim_{k}\qpz\rho(h_k\Delta(u_{T_k(\tbf+\tbf_0/T_k)}))=\phi(\tbf).
\end{align*}
Therefore $\varphi(\tbf)\in N_{G\times G}(\bigger)\cap\lcal.$

\medskip

The following observation will be important in our application:
\begin{lem}\label{r;isolated-zero} There is $\eta>0$ such that $\varphi({\tbf})=e$ and $\tbf\in B_U(\eta)$ implies $\tbf=0$, that is,
$0$ is an isolated point in $\varphi^{-1}(e)$.
\end{lem}

\begin{proof}
Since $\phi$ is a non-constant
polynomial of degree at most $2$, the set $\{\tbf\in \field:\phi(\tbf)=\qpz\}$
consists of at most two points. \end{proof}

\renewcommand{\XX}{Z}

\subsection{$\Delta(U)$-recurrence for the pull back function $\Psi$}\label{s;joining}\label{sec;doubling-PS} 
As before, we normalize $|m_{\G_i}^{\BMS}|=1$ for $i=1,2$.
For the sake of simplicity, we will often omit the subscript $\G_i$ in the notation of $m^{\BMS}_{\G_i}$ and
$m^{\BR}_{\G_i}$.

For the rest of the section, let $\mu$ be an ergodic $U$-joining 
of $Z$ with respect to $m^{\BR}_{\G_1}\times m^{\BR}_{\G_2}.$ 
In particular, if $\pi_i:X_1\times X_2\to X_i$ denotes the projection
onto $i$-th coordinate,  then $$(\pi_i)_*\mu=m^{\BR}_{\G_i}.$$
We also suppose that $\mu$ is an infinite measure,
that is, at least one of $m_{\G_i}^{\BR}$ is infinite. 
Without loss of generality we will assume $m_{\G_2}^{\BR}$
is an infinite measure.

\begin{lem}\label{full}
The following set has a full $\mu$-measure in $Z$:
$$\{(x_1, x_2)\in Z: x_i^-\in \Lambda_{\rm r}(\G_i) \text{ for each $i=1,2$}\} .$$
\end{lem}
\begin{proof} Since $\Lambda_{\rm r}(\G_i)$ has the full Patterson-Sullivan measure in $\Lambda(\G_i)$ by Sullivan \cite{Sullivan1979},
 $Y_i:=\{x_i\in X_i: x_i^-\in \Lambda_{\rm r}(\G_i)\}$ has a full $m^{\BR}$-measure. 
Since $(\pi_i)_*\mu=m^{\BR}$, we have $\pi_1^{-1}(Y_1)\cap \pi_2^{-1}(Y_2)$ has a full $\mu$-measure. 
\end{proof}

Recall that a Borel measure $\mu$ is $h$ quasi-invariant if  $h.\mu$ is a positive multiple of $\mu$
where $h.\mu(B ):=\mu(Bh)$ for any Borel subset $B$.

\begin{lem}\label{lem:one-factor-inv}
If $\mu$ is quasi-invariant under $(e,g)$ for some $g\in G,$ then it is invariant under $(e,g).$ 
\end{lem}

\begin{proof}
Suppose $(e,g)\mu=c\mu.$ 
Let $\Omega\subset X_1$ be a compact subset with $m_{\G_1}^{\BR}(\Omega)>0.$ 
Then we have 
\[
c\mu(\Omega\times X_2)=(e,g)\mu(\Omega\times X_2)=\mu((\Omega\times X_2)(e, g))=\mu(\Omega\times X_2).
\] 
Since $0<\mu(\Omega\times X_2)=m_{\G_1}^{\BR}(\Omega)<\infty,$ we get $c=1.$
\end{proof}

\begin{dfn}\rm
We fix the following:
\begin{itemize}
\item a non-negative function $\psi\in C_c(X_1)$ with
$m^{\BR}(\psi)>0$, and $$\Psi:=\psi\circ \pi_1 \in C(Z);$$
\item a compact subset $\Omega_1\subset \{x\in X_1: x^-\in \Lambda_{\rm r}(\G_1)\}$ with  $m^{\BR}(\Om_1)>0$ such that
 both Theorems \ref{t;horo-equid-BR}  and  \ref{eqwi}(2) hold for $\psi$ uniformly for all $x_1\in \Om_1$
 (such $\Omega_1$ exists in view of the remarks \ref{uniremark} and \ref{secondremark});

\item a constant $$0<r:=\tfrac{1}{4}\;{ r(0.5, \Om_1)}<1$$ where  $ r(0.5, \Om_1)$ is as given in Theorem \ref{wint};
\item  a compact subset 
\[
\mathcal Q\subset\Omega_1 \times \Omega_2\subset \XX
\] 
such that
 $\mu(\mathcal Q) >0$ and for all $x\in \Q$ and for all $f\in C_c(Z)$,
\be\label{hhh} \lim_{T\to \infty}  \frac{\int_{B_U(T)}f(x\biggere)d\tbf}
{\int_{B_U(T)}\Psi(x\biggere)d\tbf}= \frac{\mu(f)}{\mu(\Psi)}.\ee

\end{itemize}
\end{dfn}


Note that since $(\pi_1)_*\mu=m^{\BR}$, we have
 $$0<\mu(\Psi)<\infty; \text{ in particular, } \Psi\in L^1(\mu).$$

Since $\Psi$ is defined as the pull-back of a function on $X_1$,
we can transfer the recurrence properties of $U$-orbits  in $X_1$ to statements about $\Delta(U)$-recurrence
properties with respect to $\Psi$. We  record these properties of $\Psi$ in the following two lemmas.

\begin{lem} \label{ri} For any $x\in \XX$ with
 $\pi_1(x)^-\in \Lambda_{\rm r}(\G_1)$, we have  $$\int_{B_U(T)}\Psi(x\biggere) d\tbf\to \infty \quad \text{ as $T\to \infty$}.$$ \end{lem}
\begin{proof} For such an $x$, we have, by Theorem \ref{eqbr},
$$\int_{B_U(T)}\Psi(x\biggere) d\tbf
=\int_{B_U(T)}\psi(\pi_1(x) u_\tbf) d\tbf
 \sim \mu_{\pi_1(x)}^{\PS} (B_U(T))\cdot  m^{\BR}(\psi) .$$
 Since $m^{\BR}(\psi)>0$, we have $\int_{B_U(T)}\Psi(x\biggere) d\tbf\to \infty$.
\end{proof}

\begin{lem} \label{lwint}
 There exists $T_0=T_0(\psi, \Om)>1$ such that
for any $x\in \XX $ with $\pi_1(x)\in \Om$ and  any $T>T_0$,
$$\int_{B_U(r T)}\Psi(x\biggere)d\tbf\le \tfrac 12 \int_{B_U(T)}\Psi(x\biggere)d\tbf .$$

\end{lem} \begin{proof} For $x\in \XX$ with $\pi_1(x)\in \Om$,
we have \begin{align}
\notag
&\int_{B_U(T)}\Psi(x\biggere)d\tbf-\int_{B_U(rT)}\Psi(x\biggere)d\tbf \\
\notag&= \int_{B_U(T)} \psi(\pi_1(x)u_\tbf)  d\tbf-
\int_{B_U(rT)}\psi(\pi_1(x)u_\tbf) d\tbf\\
\notag&\ge \tfrac{1}{2}\int_{B_U(T)}\psi(\pi_1(x)u_\tbf )d\tbf
\label{e;window-lbound} \\
&= \tfrac{1}{2}\int_{B_U(T)}\Psi(x\biggere)d\tbf \end{align}
where the inequality follows by Theorem~\ref{wint}.
\end{proof}

 

\subsection{Joining measure}\label{join}

In this section we will use ergodic theorems 
and polynomial like behavior of unipotent orbits, 
the construction of a polynomial-like map in \S~\ref{ss;quasi-reg} in order to produce extra quasi-invariance for the measure.
We begin with the following.

\begin{lem}\label{l;normalizer-inv}
Let
$Y\subset \XX $ be a Borel subset such that for all $y\in Y$,
\begin{enumerate}
\item 
$\lim_{T\to \infty} \int_{B_U(T)}\Psi(y\biggere)d\tbf =\infty$;
\item for all $f\in C_c(\XX)$,
$$
\lim_{T\to \infty}  \frac{\int_{B_U(T)}f(y\biggere)d\tbf}
{\int_{B_U(T)}\Psi(y\biggere)d\tbf}= \frac{\mu(f)}{\mu(\Psi)}.$$
\end{enumerate}
If $h\in N_{G\times G}(\bigger)$ satisfies
$Y \cap Y h\neq\varnothing$, then $\mu$ is $h$ quasi-invaraint.
\end{lem}   
 
\proof  Let $h\in N_{G\times G}(\bigger)=\Delta(AM) \cdot (U\times U)$ and $y\in Y$ such that $yh\in Y$.
Since, under conjugation, $\Delta(AM)$ acts on $\bigger$ by homothethy composed with
 rotations, 
$$h^{-1}\biggere h=\Delta (u_{\beta_h(\tbf)})$$
where $\beta_h: U \to U$ is a homothety composed with a rotation.

 If the $\Delta(A)$-component of
$h$ is $(a_s, a_s)$, then  $\beta_h(B_U(T)) =B_U (e^s T)$ and the Jacobian of $\beta_h$ is equal to
$e^{(n-1)s}$ where $(n-1)$ is the dimension of $U$ as a real vector space.

For any $f\in C_c(\XX)$ and for any all large $T\gg 1$,
\begin{align*}&
 \left|\frac{\mu(h. f)}{\mu(h.\Psi)}-\frac{\mu(f)}{\mu(\Psi)}\right| \\
 &=
 \left|\frac{\mu(h. f)}{\mu(h.\Psi)} -\frac{\int_{B_U(T)}f(y \biggere h)d\tbf}
 {\int_{B_U(T)}\Psi(y\biggere h)d\tbf}+ \frac{\int_{\beta_h(B_U(T))}f(y h\biggere)d\tbf}
 {\int_{\beta_h(B_U(T))}\Psi(yh\biggere)d\tbf}-
 \frac{\mu(f)}{\mu(\Psi)}\right|\\
 &\leq \left|\frac{\mu(h.f)}{\mu(h.\Psi)} -\frac{\int_{B_U(T)}f(y\biggere h)d\tbf}
 {\int_{B_U(T)}\Psi(y\biggere h)d\tbf}\right|+
 \left|\frac{\int_{B_U(e^s T)}f(yh\biggere )d\tbf}
 {\int_{B_U(e^sT)}\Psi(yh\biggere )d\tbf}-\frac{\mu(f)}{\mu(\Psi)}\right| .
\end{align*}
Since both $y$ and $yh$ belong to $Y$, it follows that the last two terms tend to zero as $T\to \infty$. Hence
$${\mu(h. f)}= \frac{\mu(h. \Psi)}{\mu(\Psi)} \cdot  {\mu(f)},$$ finishing the proof.
\qed

 For $0<r_0<1$, and $T>0$,  set 
 \be \label{ir} I_{r_0}(T):=B_U(T)\setminus B_U(r_0T).\ee

For  a Borel function $f$ on $\XX,$  and $x\in Z$,
set
\be \label{dt}
 \mathcal D_Tf(x)=\int_{I_r(T)}f(x\biggere)d\tbf.
\ee

\begin{cor}\label{c;window-rT}
 For any $f\in C_c(\XX)$, and any $x\in \Q$, we have 
 \begin{equation}\label{uniff}
\lim_{T\to \infty}\frac{ \mathcal D_Tf(x) }
{ \mathcal D_T \Psi (x)}
= \frac{\mu( f)}{\mu(\Psi)} .
\end{equation}
\end{cor}

\begin{proof} 
This follows from Lemma \ref{lwint} and \eqref{hhh}.
\end{proof}

\begin{remark}\label{rk:separable}
Let $\mathcal F\subset L^1(\XX,\mu)$ be any countable subset. Then 
there is a full measure subset $\mathcal Q'\subset \mathcal Q$ so that~\eqref{uniff}
holds for all $f\in\mathcal F$ and $x\in\mathcal Q'.$ 
Indeed by the Hopf ratio theorem and the fact that $\mathcal F$ is countable 
there is full measure subset $\qcal'\subset\mathcal Q$ 
so that analogue of~\eqref{hhh} holds for all $f\in\mathcal F.$ Then using Lemma~\ref{lwint}
we have $\frac{\D_T f(x)}{\D_T \Psi(x)}\to\frac{\mu(f)}{\mu(\Psi)}$ for all $f\in\mathcal F$
and all $x\in\mathcal Q'.$
\end{remark}

Fix a small $\e>0$, and 
choose $\eta>0$ small enough so that 
$\mu(\qcal\{g:|g|\leq\eta\})\leq (1+\e)\mu(\qcal).$ 
We  put 
$$\qcal_+=\qcal\{g:|g|\leq\eta/4\},\text{ and } \qcal_{{++}}=\qcal\{g:|g|\leq\eta\}.$$
Set $$\mathcal{F}:=\{\chi_\qcal,\chi_{\qcal_+},\chi_{\qcal_{++}}\}.$$

Using this and Egorov's theorem and Remark~\ref{rk:separable},
we can find a compact subset
$$\qcal_\e\subset\qcal$$ with $\mu(\qcal_\e)>(1-\e)\mu(\qcal)$, such that
for any $f\in\mathcal{F}$ and any $\theta>0$,
there exists some $T_0=T_0(f,\theta)$ so that if $T\geq T_0,$ then
\be\label{e;unif-conv}
 \left|\frac{\D_T f(x)}{\D_T \Psi(x)}-\frac{\mu(f)}{\mu(\Psi)}\right|\leq\theta\quad\text{for all $x\in\qcal_\e.$}
\ee
Such a subset will be referred to as a {\it set of uniform convergence} for the family 
$\mathcal{F}$ (cf. \cite[Lemma 7.5]{MO}). 
 

\begin{lem}\label{l;window_1}
Fix $0<\sigma<1/2$. There exist  $T_0=T_0(\psi, \Om_1)>0$
and $c_0=c_0(\psi,\Om_1) >1$ such that for all  $T>T_0$ and for all $x\in\mathcal{Q}_{\e}$ we have
\[
c_0^{-1}\int_{I_{r^-_\sigma}(T^+_\sigma )}\Psi(x\biggere)d\tbf\leq \int_{B_U(T)}\Psi(x\biggere)d\tbf \leq
c_0 \int_{I_{r^+_\sigma}(T^-_\sigma )}\Psi(x\biggere)d\tbf 
\]
 where
$r^{\pm}_{\sigma}=(1\pm \sigma ) r$ and $T^{\pm}_{\sigma}=(1\pm  \sigma) T$.

\end{lem}

\begin{proof} 
Since Theorem \ref{eqwi}(2)  holds for $\psi$ uniformly for all $\pi_1(x)\in \Om_1$, there exists $T_0>1$ such that
 for any $x\in\mathcal{Q}_\e$ and for any $T>T_0$, \begin{align*} \label{eq:change-var-2}
\int_{I_{r^+_\sigma}(T^-_\sigma )}\Psi(x\Delta(u_\tbf))d\tbf &=
\int_{I_{r^+_\sigma}(T^-_\sigma )}\psi(\pi_1(x) u_\tbf )d\tbf \\ 
&\ge \tfrac{1}{2} m^{\BR}(\psi)\left(\mu_{\pi_1(x)}^{\PS}(B_U({T_\sigma^-})-B_U({r_\sigma^+T_\sigma^-}))\right).
\end{align*}

On the other hand, by
Lemma~\ref{p;horo-wind} and  Lemma \ref{uc}, there exist $T_1>T_0$ and $c'_0>0$, depending only on $\Om_1$ such that
if $\pi_1(x)\in \Om_1$,
$$\mu_{\pi_1(x)}^{\PS}(B_U({T_\sigma^-})-B_U({r_\sigma^+T_\sigma^-}))
\ge \tfrac 12 \mu_{\pi_1(x)}^{\PS}(B_U({T_\sigma^-}))\ge c'_0
\mu_{\pi_1(x)}^{\PS}(B_U(T)).$$
 
  Since
 Theorem~\ref{t;horo-equid-BR}
 holds for $\psi$, uniformly for all $\pi_1(x)\in \Om_1$, we have, for all sufficiently large $T\gg 1$,
 $$
m^{\BR}(\psi) \mu_{\pi_1(x)}^{\PS}(B_U(T))  \ge  \tfrac 12  \int_{B_U(T) }\psi(\pi_1(x) u_{\tbf })d\tbf .
$$   
Therefore,  
\begin{align*} 
&\int_{I_{r^+_\sigma}(T^-_\sigma )}\Psi(x\Delta(u_\tbf))d\tbf\geq \tfrac{c'_0}{2}m^{\BR}(\psi) \mu_{\pi_1(x)}^{\PS}(B_U(T))
\\ &  \ge  \tfrac{c'_0}{4}  \int_{B_U(T) }\psi(\pi_1(x) u_{\tbf })d\tbf 
=  \tfrac{c'_0}{4} \int_{B_U(T) }\Psi(x\Delta(u_\tbf))d\tbf .
\end{align*} The other direction can be proved similarly.
\end{proof}

The following lemma will be used to compare ergodic averages along
two nearby orbits.

\begin{lem}\label{l;window2}
Let $\{R_k\}$ be a sequence tending to infinity as $k\to\infty$ and fix a small number
$0<\sigma<1/2$. For each $k$, let
 $\alpha_k:\field\to\field$ be a rational map with no poles on $B_{R_k}(\field)$. Suppose that for
 all $\tbf\in  B_{R_k}(\field)$,
$$1-\sigma\leq |{\rm Jac}(\alpha_k)(\tbf) |\leq 1+\sigma .$$
Then there exist $c_1>1$ and $T_1=T_1(\Om,\psi,\mathcal F)>1$ such that
 for all $T_1<T<R_k/4$, $x\in \Q_\e$ and $f\in \mathcal F$, we have
 $$
c_1^{-1} \cdot \D_T f (x)  \le
\int_{I_{r}(T)}f(x\Delta(u_{\alpha_k(\tbf)}))d\tbf\leq 
c_1 \cdot  \D_T f (x). $$
\end{lem}

\begin{proof} 
First note that
\be\label{eq:change-var-1}
\int_{I_{r}(T) }f(x\Delta(u_{\alpha_k(\tbf)}))d\tbf=
\int_{\alpha_k(I_{r}(T)}f(x\Delta(u_{\tbf}))|{\rm Jac}(\alpha_k)(\tbf)| d\tbf.
\ee  

Setting $r^\pm_{\sigma}=(1\pm\sigma ) r$ and $T^\pm_{\sigma}=(1\pm\sigma) T$,
 note that 
$$I_{r^+_\sigma}(T^-_\sigma)
\subset\alpha_k(I_{r}(T)) \subset I_{r^-_\sigma}(T^+_\sigma ).$$
Now for all $T>2T_0(\psi, \Om),$ where $T_0$ is as in Lemma~\ref{l;window_1},
we have 

\begin{align*}
\int_{I_{r}(T)}f(x\Delta(u_{\alpha_k(\tbf))})d\tbf&\geq
(1-\sigma) \int_{I_{r^+_\sigma}(T^-_\sigma )}f(x\biggere)d\tbf\\
&= (1-\sigma)  \frac{\mu(f )}{\mu(\Psi)} \int_{I_{r^+_\sigma}(T^-_\sigma )}\Psi(x\biggere)d\tbf+\alpha_1(T)\\
{}^{\text{Lemma~\ref{l;window_1}}\leadsto}&\geq(1-\sigma)c_0  \frac{\mu(f )}{\mu(\Psi)}  \int_{B_U(T)}\Psi(x\biggere)d\tbf  + \alpha_1(T) \\
&={(1-\sigma)c_0^{-1}} \int_{B_U(T)}f(x\biggere)d\tbf+\alpha_2(T)\\
&\geq{(1-\sigma)c_0^{-1}} \D_T f (x) +\alpha_2(T) 
\end{align*}
where $\alpha_i(T)$'s satisfy
$\lim_{T\to \infty} \frac{\alpha_i(T)}{\int_{B_U(T)}\Psi(x\biggere)d\tbf }=0$

By Lemma \ref{lwint}
and \eqref{e;unif-conv}, it follows that for all $x\in \Q_\e$,
$$\lim_{T\to \infty} \frac{\alpha_i(T)}{\D_Tf(x) }=0$$
where the convergence is uniform on $\Q_\e$.

Therefore for all $x\in \Q_\e$ and $T$ large,
$${(1-\sigma)c_0^{-1}} \D_T f (x) +\alpha_2(T)  \ge c_1^{-1} \D_T f (x)$$
and hence
$$\int_{I_{r}(T)}f(x\Delta(u_{\alpha_k(\tbf))})d\tbf \ge c_1^{-1} \D_T f (x)$$
for some $c_1>1$ and for all $T$ bigger than some fixed $T_1>1$.
The other inequality can be proved similarly.
\end{proof}


For a subset $S\subset G\times G$, 
we denote by $\langle S\rangle$ the minimal connected
subgroup of $G\times G$ containing $S$.

\begin{thm}\label{l;extra-inv}
Let $h_k\in G\times G- N_{G\times G}(\bigger)$ be a sequence tending to $e$ as $k\to \infty$.
If $\qcal_\e h_k\cap\qcal_\e\neq\varnothing$ for all $k$, then
 $\mu$ is quasi-invariant under a nontrivial 
connected subgroup of $\Delta(AM) (\{e\}\times U)$. Moreover, if $h_k\in \{e\}\times G$ for all $k$, then
 $\mu$ is invariant under a nontrivial 
connected subgroup of $\{e\}\times U$. 
 \end{thm}

\proof 
We use the notation used in the construction of the map $\varphi$
  in section \ref{ss;quasi-reg} with respect to $\{h_k\}$.
By our assumption we have that there are points $y_k\in\qcal_\e$
so that $x_k=y_kh_k\in\qcal_\e.$
Recalling the maps $\varphi_k$ and $\alpha_k$ 
from above, we have
\[
 x_k\biggere=y_kh_k\biggere=y_k\Delta(u_{\alpha_k(\tbf)})\varphi_k(\tbf).
\]
Now let 
$$ \tau_k':= \sup_{t\in B_U(\tau)} d(e, {\varphi}_k(\tbf)) =\eta/4,
\;\; \tau_k=\min\{\tau_k',R_k\},$$
and
$$ R_k=\sup\{0<R<\infty :0.9\leq{\rm Jac}(\alpha_k)|_{B_R(\field)}\leq1.1\}.$$

Note that $$ \sup_{t\in B_U(\tau_k')} d(e, {\varphi}_k^{-1}(\tbf)) =\eta/4 .$$

Note that $\Theta_k =\tau_k/T_k$ is bounded away from $0;$
in particular, $\tau_k\to\infty.$ 
Passing to a subsequence we may assume that $\Theta_k$ converges to some $\Theta .$ 


By the definition of $\mathcal Q_\e$, we have,
for all large enough $T$, \[
 \left| \frac{\int_{I_r(T)} f_1 (z_k\biggere )d\tbf}
 {\int_{I_r(T)} f_2 (z_k\biggere )d\tbf}-
 \frac{\mu(f_1)}{\mu(f_2 )} \right| \leq \theta
 \] for $f_1, f_2\in \mathcal F=\{\chi_\qcal,\chi_{\qcal_+},\chi_{\qcal_{++}}\}$
 and $z_k=x_k,y_k.$

With this notation, the above implies: for all large enough $k$
and all $T_0\leq T\leq \tau_k$ we have 
\be\label{e;simul-qcal-general}
 \{\tbf\in I_r(T): x_k\biggere,y_k\Delta(u_{\alpha_k(\tbf)})\in\qcal\}\neq\varnothing.
\ee
To see this, let $k$ be large and let $T_0\leq T\leq\tau_k$. Then
\begin{align*}
 \{\tbf\in I_r(T):x_k\Delta(u_\tbf)\in\qcal\}&\subset 
 \{\tbf\in I_r(T):y_k\Delta(u_{\alpha_k(\tbf)})\in\qcal_+\}\\
 &\subset \{\tbf\in I_r(T):x_k\Delta(u_\tbf)\in\qcal_{++}\}.
\end{align*}
On the other hand we have
\be\label{e;qcal-qcal++-general}
\begin{array}{c}
\ell \{\tbf\in I_r(T):x_k\Delta(u_\tbf)\in\qcal\} \geq 
(1-\e)\ell \{\tbf\in I_r(T) :x_k\Delta(u_\tbf)\in\qcal_{++}\}
\end{array}
\ee where $|\cdot |$ denotes the Lebesgue measure on $\field$.
From these two and Lemma~\ref{l;window2} we get
\begin{align}
\label{e;xk-x-general} & \ell \{\tbf\in I_r(T):y_k\Delta(u_{\alpha_k(\tbf)})\in\qcal\} \\&
\geq c_1 \ell \{\tbf\in I_r(T):y_k\Delta(u_\tbf)\in\qcal\} \\
&\notag\geq c_1(1-\e) \ell \{\tbf\in I_r(T):y_k\Delta(u_\tbf)\in\qcal_+\} \\
&\notag \geq c_1(1-\e) \ell \{\tbf\in I_r(T):x_k\Delta(u_{\alpha^{-1}_k(\tbf)}) \in\qcal\} \\
&\notag \geq c_1^2(1-\e) \ell \{\tbf\in I_r(T):x_k\Delta(u_{\tbf}) \in\qcal\} \\
&\notag\geq c_1^2(1-\e)^2 \ell \{\tbf\in I_r(T):x_k\Delta(u_\tbf)\in\qcal_{++}\}.
\end{align}
Now~\eqref{e;simul-qcal-general} follows by applying~\eqref{e;qcal-qcal++-general} 
and~\eqref{e;xk-x-general}, in view of the choice of $\e$
and the fact that by Corollary~\eqref{c;window-rT} we have 
$$\ell \{\tbf\in I_r(T):x_k\Delta(u_\tbf)\in\qcal_{++}\} >0 \text{ for all large enough $T$}.$$

For each $k$, let $m_k\geq0$ be the maximum integer 
so that $r^{m_k}\tau_k\geq T_0.$
Then for any $\ell\geq 0$ and all large enough $k$ we have $\ell\leq m_k.$
Let $\ell\geq 0$ and apply~\eqref{e;simul-qcal-general} with 
$T_{k,\ell}=r^\ell\tau_k.$
Then for each $k$ we find $t\in I_{r}(T_{k,\ell} )$ so that
$z_{k,\ell}=y_k\Delta(u_{\alpha_k(\tbf)})$ satisfies
$z_{k,\ell}\in\qcal$ and $z_{k,\ell}\varphi_k(\tbf)\in\qcal.$
Passing to a subsequence we get: there exist some $z_\ell\in\qcal$ 
and some $s\in B_U(\Theta)\setminus B_U(r\Theta)$ so that 
$z_\ell\varphi(s)\in\qcal.$
Therefore by Lemma~\ref{l;normalizer-inv} we have $\mu$
is $\varphi( s)$ quasi-invariant. 
Now if we choose $\ell$ large enough, then $\varphi(s)\neq e$ 
in view of Lemma~\ref{r;isolated-zero}
however, it can be made arbitrary close to the identity by choosing large $\ell$'s.  
This implies the first claim, since the image of $\varphi$ is contained in the subgroup $\Delta(AM) (\{e\}\times U)$.

Now, if
$h_k\in \{e\}\times G$, the image of $\varphi$ is contained in $
N_{G\times G}(\Delta (U))\cap (\{e\}\times G)=\{e\}\times U$. Indeed, $\varphi_k(t)=(e,u_{-t}g_ku_t)$
and $\alpha_k(t)=t.$ Therefore we get $\mu$ is quasi-invariant under the action of a nontrivial connected subgroup of $\{e\}\times U$; hence the claim follows from
 Lemma~\ref{lem:one-factor-inv}.
\qed

\subsection{Infinite joining measure cannot be invariant by $\{e\}\times V$ for $V<U$}\label{ss;not-prod}

 We recall some basic facts about dynamical systems. Consider an action of one-parameter subgroup $W=\{w_t\}$ on 
a separable, $\sigma$-compact and localiy compact topological space $X$ with an invariant Radon measure $\mu_0$.
 A Borel subset $E\subset X$ is called {\it wandering} if $\int_{\br} \chi_E(x w_t)<\infty$
for almost all $x\in E$. The Hopf decomposition theorem says that $X$ is a disjoint union of
invariant subsets $\mathcal D(W)$
 and $\mathcal C(W)$ where $\mathcal D(W)$ is a countable union of wandering subsets which is maximal
in the sense that any wandering subset is contained in $\mathcal D(W)$ up to null sets (see \cite{Kre}).
The sets $\mathcal D(W)$ and $\mathcal C(W)$ are called the dissipative part,
 and  the conservative part of $X$ respectively.
 If $\mathcal D(W)$ (resp. $\mathcal C(W)$) is a null set, this action is called {\it conservative} (resp. {\it dissipative}).
If the $W$-action is ergodic, then it is either conservative or dissipative.
The following is well known for a single transformation (e.g. \cite{Aa}), but we could not find a reference for a flow; so we provide
a proof for the sake of completeness. 

\begin{lem}\label{cons} If $\mu_0$ is ergodic and {\it infinite}, then for any non-negative $f\in L^1(\mu_0)$,
$$\frac{\int_{-T}^T f(x w_t) dt}{2T}\to 0$$ 
for almost all $x\in X$.
\end{lem}

\begin{proof} 
 Note that there is nothing to prove if $\mu_0$ is not conservative, hence, we assume $\mu_0$
is conservative in the rest of the argument.
Since $X$ is $\sigma$-compact it suffices to prove that for any compact subset
$\kcal\subset X$ and almost all $x\in \kcal$ the above holds.
Let $\kcal\subset X$ be a compact subset.
We will show that for any $\e>0$,
the set $\{x\in \kcal: \frac{\int_{-T}^T f(x w_t) dt}{2T}\to 0\}$
has co-measure less than $\e>0$. Write $X=\cup_{N=1}^{\infty} \Omega_N$ as an increasing union of compact subsets with 
$\mu_0(\Omega_1)>0$.
By the Hopf ratio theorem and Egorov's theorem 
we can find a subset $\kcal_\e$ of $\kcal$ with co-measure at most $\e$ 
such that the following convergence
is uniform for all $x\in \kcal_\e$ and for all $N$:
$$\frac{\int_{-T}^T f(xw_t) dt }{\int_{-T}^T \chi_{\Omega_N}(xw_t) dt }
\to \frac{\mu_0(f)}{\mu_0(\Omega_N)}.$$
Hence for any $\eta>0$, there exists $T_\eta$ such that
for all $x\in \kcal_\e$, $T>T_\eta$,
$$\limsup_N \left| \frac{\int_{-T}^T f(xw_t) dt }{\int_{-T}^T \chi_{\Omega_N}(xw_t) dt }
-\frac{\mu_0(f)}{\mu_0(\Omega_N)}\right| \le \eta.$$

Since  $\mu_0(\Omega_N)\to |\mu_0|=\infty$ and $\int_{-T}^T \chi_{\Omega_N}(xw_t) dt =2T$ for all large $N$,
it follows that $$  \frac{\int_{-T}^T f(xw_t) dt }{2T } \le \eta,$$
for all $T>T_\eta$ and hence  $\frac{\int_{-T}^T f(xw_t) dt }{2T }\to 0$ for all $x\in \kcal_\e$.
\end{proof}

\begin{remark}
We recall that if $\G$ is not a lattice, then the BR measure is an infinite measure;
this was proved in~\cite{OS} 
using Ratner's measure classification theorem. 

We take this opportunity 
to present an alternative argument.
To see this, we note that if the $\BR$-measure were a finite measure, it would have to be $A$-invariant, since 
$|a_s m^{\BR}|=e^{(2-\delta)s}|m^{\BR}|$ for all $s$ and hence $\delta=2$. For $\G$ geometrically finite, this implies $\G$ is a lattice. 
In the general case, one can utilize facts from entropy to prove a similar result
as we now explain. 
Indeed by the Mautner phenomenon, any $AU$-invariant finite ergodic measure 
on $\G\ba G$ is $A$-ergodic so we may reduce to the ergodic case\footnote{Indeed similar reductions are possible using Hopf argument in more general settings.}.
Now we have an $A$-ergodic measure which is $U$ invariant; 
in particular it has maximum entropy. This implies the entropy contribution
from ${\cont}$ has to be maximum as well which implies the measure is also ${\cont}$
invariant, see~\cite[Theorem 9.7]{MT} for a more general statement. This implies $\G\ba G$ has a finite $G$ invariant measure, finishing the proof.
\end{remark}

We need the following lemma which says almost all ergodic components of $m^{\BR}$ is infinite
for any one-parameter subgroup of $U$; our proof of this lemma
uses Ratner's classification theorem for {\it finite} invariant measures for unipotent flows. 
\begin{lem}\label{l;br-infinite} Let $\G$ be a Zariski dense, discrete subgroup of $G$. Suppose
$\G$ is not a lattice. Let $V$ be a one-parameter subgroup of $U$, and
let $m^{\BR}=\int_Y \eta_y d\sigma(y)$ be the ergodic decomposition with respect to $V$.
Then for $\sigma$-a.e. $y$, we have $\eta_y$
is an infinite measure.
\end{lem}

\begin{proof} We will use the fact that the set $\Lambda_{\rm bp}(\G)$ of parabolic limit points is a null set
for the Patterson-Sullivan measure since $\Lambda_{\rm bp}(\G)$ is a countable set and that  a proper Zariski closed subset  of $G$
is a null set for the $\tilde m^{\BR}$-measure, since $\G$ is Zariski dense.
 Assume the contrary, that is: the set
\[
Y_0=\{y\in Y: \eta_y\;\text{ is a finite measure }\}
\]
has positive measure.

It follows from  Ratner's measure classification theorem~\cite{Rat-Ann}: 
that for all $y\in Y_0$, we have one of the following holds
\begin{enumerate}
\item $\supp \; \eta_y= x V$ for some compact orbit $xV$;
\item $\supp \; \eta_y =xU$ for some compact orbit $xU$;
\item there exists $H$ which is locally isomorphic to $\PSL_2(\bbr)$ 
so that  for some $g\in G$, $V\subset g^{-1} H g$, $\eta_y$ is 
a $g^{-1} Hg $ invariant (finite) measure on a closed orbit $\G Hg$;
\item $\eta_y$ is $\PSL_2(\bbc)$ invariant. 
\end{enumerate}
In both (1) and (2) above we get $x^-\in \Lambda_{\rm bp}(\Gamma)$ and these form a measure zero subset of
$m^{\BR}.$ 
The conclusion in (4) cannot hold on a positive measure set
as it would imply $\G$ is a lattice, contrary to our assumption.
Therefore for $\sigma$-a.e.\
$y\in Y_0$ the conclusion (3) above holds. We first note 
that the collection of $H$ so that (3) holds is countable, see~\cite[Theorem 1.1]{Rat-Ann}
or~\cite[Proposition 2.1]{DM-Linearization}.
Therefore if (3) holds there exists some $H$ with $\G H$
a closed orbit (with finite volume) so that
\be\label{e;singular-set}
\tilde m^{\BR}\{g\in G: gV\subset Hg\}>0 .
\ee 
Since $\{g\in G: gV\subset Hg\}$ is a proper Zariski closed subset, this yields a contradiction.
\end{proof}

\begin{lem}\label{l;not-product}\label{nf} 
The joining 
 $\mu$ is not invariant under $\{e\}\times V$ for any non-trivial connected subgroup $V$ of $U$.
\end{lem}

\proof 

It suffices to prove the claim when $V$ is one dimensional subgroup of $U$.
Set $V=\{v_t:t\in \br \}$ and $\tilde V=\{\tilde v_t=(e,v_t) : t\in \br\}$.


By the choice of $\Psi$, $\Psi\in L^1(\mu)$, and 
for any $x=(x^1, x^2)\in \XX$ and any $T\ge 1$, we have
\be\label{ergdd}
\frac{\int_{-T}^T \Psi(x\tilde v_t) dt}{2T} =\psi(x_1).
\ee


Also note that every element of the sigma algebra 
\[
\Xi=\{B\times X_2: \text{$B\subset X_1$ any Borel set}\}
\]
is $\tilde V$ invariant. 
In particular, $\tilde V$-ergodic components of $\mu$ are supported on
atoms of $\Xi$ which are of the form $\{x^1\}\times X_2$ for $x^1\in X_1$.
Let 
\[
\mu=\int_{\XX} \mu_zd\tau(z)
\] 
be an an ergodic decomposition of $\mu$ for the action of $\tilde V,$ see~\cite{Aa}.
Then the above discussion implies that for $\tau$-a.e.\,$z.$,
${\rm supp}(\mu_z)\subset {\{x_{z}^1\}\times X_2}$
for some $x_{z}^1\in X_1$. In particular, taking the projection onto $X_2$
we get an ergodic decomposition of $m^{\BR}_{\G_2}$ for the action of $V$:
$
m^{\BR}_{\G_2}=(\pi_2)_*\mu= \int_{Z}\tilde \mu_z d\tau(z)
$
where $\tilde \mu_z=(\pi_2)_*\mu_z$.
By Lemma~\ref{l;br-infinite}, a.e.\ $\tilde\mu_z$ is an infinite measure.





This gives a contradiction if we apply \eqref{ergdd} for a point $(x^1,x^2)$ where $x^2$ lies in the conull set
given by Lemma \ref{cons} applied to some $\eta_y$ as above.
\qed
 

\subsection{  }\label{ss;finite-fiber}
 We draw  two corollaries of Theorem \ref{l;extra-inv} in this subsection.
Let $\mathcal P(X_2)$ denote the space of probability Borel measures.
By the standard disintegration theorem (cf. \cite[1.0.8]{Aa}), for each $i=1,2$,
 there exists an $m^{\BR}$ co-null set
$X_1'\subset X_1$ and a measurable function $X_1' \to \mathcal P(X_2)$  given by $x^1\mapsto \mu_{x^1}^{\pi_i}$ 
such that for any Borel subsets $Y\subset \XX$ and $C\subset X_1$,
$$\mu(Y\cap \pi_1^{-1}(C))= \int_C \mu_{x^1}^{\pi_1}(Y) \;  dm^{\BR} (x^1) .$$
The measure $\mu_{x^1}^{\pi_i}$ is called  the fiber measure over $\pi_1^{-1}(x^1)$.

\begin{thm}\label{c;finite-fiber}
 There exist a positive integer $\fibmeas>0$ 
and an $m^{\BR}$ conull subset $X'\subset X_1$
so that $\pi_1^{-1}(x^1)$ is a finite set with cardinality $\fibmeas$
for all $x^1\in X'.$
Furthermore, $$\sfiber_{x^1}(x^2)=1/\fibmeas$$ for any
$x^1\in X'$ and $(x^1,x^2)\in\pi_1^{-1}(x^1).$ 
\end{thm}

\proof 
We first prove that for a.e. $x^1\in X_1$, the fiber measure $\sfiber_{x^1}$ is fully atomic.
Assuming the contrary, we will 
show that $\mu$ is invariant under $\{e\}\times V$ for some non-trivial connected subgroup of $U$,
which will be a contradiction by Lemma \ref{nf}.

Put $B=\{x^1\in X_1: \sfiber_{x^1}\text{ is not fully atomic}\},$ 
and suppose that $m^{\BR}( B)>0$. For any $x^1\in B$
we write $$\sfiber_{x^1}=(\sfiber_{x^1})^a+(\sfiber_{x^1})^c$$
where $(\sfiber_{x^1})^a$ and $(\sfiber_{x^1})^c$ are respectively the purely atomic part and
the continuous part of the fiber measure \cite{Joh}.
Let 
\[
B'=\{(x^1,x^2): x^1\in B,\;\; x^2\in\supp((\sfiber_{x^1})^c)\}.
\] 
We take $\Q\subset B'$ and $\Q_\e\subset \Q$ be as in section \ref{join} for each small $\e>0$;
In particular, ~\eqref{e;unif-conv} holds for $\Q_\e.$ 

Let now $x=(x^1,x^2)\in \Q_\e$ be so that 
there exists a sequence $\{x_k=(x^1,x^2_k)\}\subset \Q_\e$
so that $x_k\to x.$ Such $x$ exists since $\Q\subset B'$.
We write $$x_k=(x^1,x^2_k)=(x^1,x^2)(e,g_k)$$
where $g_k\neq e $ and $g_k\to e.$
There are two possibilities to consider:
Recall that $$N_{G\times G}(\Delta (U))\cap (\{e\}\times G)=\{e\}\times U.$$

\medskip

{\bf Case 1}. For all large enough $k$, we have $g_k\in U,$
and hence $(e,g_k)\in N_{G\times G}(\dU).$
Since $(x^1, x^2), (x^1, x^2 g_k)\in \Q_\e$,
 Lemma~\ref{l;normalizer-inv} implies that  $\mu$ is quasi-invaraint 
under $\langle(e,g_k)\rangle$, Since  $g_k\to e$, and $U$ is 
 a unipotent group, we get $\mu$ is invariant by a non-trivial connected subgroup of $\{e\}\times U$, which is a contradiction by Lemma \ref{nf}. 

\medskip

{\bf Case 2.} By passing to a subsequence, we have
$g_k\notin U$, that is, $h_k:=(e,g_k)\not\in  N_{G\times G}(\dU).$ 
By Theorem~\ref{l;extra-inv}, we get $\mu$ is
 invariant, by a non-trivial connected subgroup  of $\{e\}\times U$.
This is again a contradiction by Lemma \ref{nf}. 
 This shows that almost all fiber measures are atomic.

Set
\[
\Sigma=\{(x^1,x^2)\in \XX: \sfiber_{x^1}(x^1)=\max_{y\in \pi_1^{-1}(x^1)} \sfiber_{x^1}(y)\}.
\]
Then $\Sigma$ is a $\dU$-invariant set. Since almost all fiber measures are atomic,
we have
$\mu(\Sigma)>0$. Therefore, in view of the $\dU$-ergodicity of $\mu,$
we have $\Sigma$ is conull. We thus conclude that for 
$\mu$-almost every point, the fiber measures are uniform distribution
on $\ell$-points.
\qed

\begin{cor}\label{c;A-quasi-inv}\label{alunch}
The joining measure $\mu$ is quasi-invaraint under a 
 non-trivial connected subgroup $A'$ of $\Delta (AM) (\{e\}\times U)$ which is not contained
 in $\{e\}\times U$.

\end{cor} 

\proof
Let the notation be as in Theorem~\ref{l;extra-inv}.
In particular, $\qcal$ is a compact subset with $\mu(\qcal)>0$
and $\qcal_\e\subset\qcal$ with $\mu(\qcal_\e)\geq(1-\e)\mu(\qcal).$ 

With this notation let $x_k,y_k\in\qcal_\e$; 
suppose $x_k=y_kh_k,$ with $h_k\to e$ as $k\to\infty.$ 
Since $(\pi_{i})_*\mu=m^{\BR},$ we can choose $x_k$ so that
for at least one of $i=1,2$, $\pi_{i}(h_k)\notin N_G(U)$.
This, in particular, implies that $h_k\notin N_{G\times G}(\dU).$

Now apply Theorem~\ref{l;extra-inv} with $\{x_k\}$ and $\{y_k\}.$
We get a map 
\[
\varphi:\dU\to N_{G\times G}(\dU)\cap\lcal=\Delta( AM)\cdot(\{e\}\times U)
\]
so that $\mu$ is quasi-invariant under a non-trivial connected subgroup, $L$ say,
of $\langle{\rm Im}(\varphi)\rangle.$

Note that by Lemma~\ref{lem:one-factor-inv} and Lemma~\ref{l;not-product}
we have $\{e\} \times V$ is not contained in $L$
for any nontrivial subgroup $V$ of $U.$ 
Therefore the conclusion follows.


\medskip

By replacing $\mu$ by $(e,u).\mu$, we may assume that $\mu$ is
$\Delta(A'U)$-invariant for a non-trivial connected Lie subgroup 
$A'$ of $AM$ in the rest of the section.

\subsection{Reduction to the rigidity of measurable factors}\label{ss;end-proof}

By Theorem~\ref{c;finite-fiber},
 we have: $\mu$-a.e  fibers of $\pi_1$  have cardinality
$\ell $ for some fixed $\ell\in \mathbb N$.

We put $$\mfib(x^1)=\pi_1^{-1}(x^1).$$
Then there exist a $U$-invariant  $\BR$ co-null subset  $X'\subset X_1$ and $\ell$ measurable maps
 $$\msec_1,\ldots,\msec_\ell: X'\to X_2$$ so that 
 $ \mfib(x^1)=\{\msec_1(x^1),\ldots,\msec_\ell(x^1)\}$  for all $x^1\in X'$ (see \cite{Rat-Joining}, \cite{Ro}).
 Note that if $\mu$ is $\Delta(L)$-quasi invariant for some subgroup $L\subset G,$
 then $\mfib$ is $L$-equivariant. 
Therefore
 the set-valued map $\mfib$ is $A'U$-equivariant; for all $x^1\in X'$ and every $a'u_\tbf\in A'U$,
we have \be\label{uinvv} \mfib(x^1 a'u_\tbf)=\mfib(x^1) a'u_\tbf.\ee

\begin{lem} \label{sleep} Let $X'$ be a $U$-invariant Borel subset of $X_1$ with
 $m^{\BMS}(X')=1$ satisfying \eqref{Ueq}. Then there exist $x\in X'$, and a subset $L_x\subset AM$ 
  generate $AM$ 
 such that for any $g\in L_x$,
 \be \label{almostf}\mfib(xg) \subset \mfib(x)gU.\ee
  \end{lem}
\begin{proof}
Let $\eta>0$ be small, and let $K_\eta\subset X'$, $\Omega_\eta\subset X'$  and $T_\eta>0$
be as in \eqref{uniformcc} and \eqref{odef}. For $\e>0$,
let $\e'>0$ be as in \eqref{Kc}.
We also assume that for
all $g\in AM$ with $d(e, g)<\e'$,
 the Jacobian of $g$-action on $U$ is  bounded from above and below by $1\pm\tfrac{ \e}2$ respectively.

 By the ergodicity of $A$-action (see \eqref{birk}), there exists a compact subset $\Omega_\eta'\subset \Omega_\eta$
 such that $m^{\BMS}(\Omega_\eta')>1-4\eta$ and
  for any $x\in \Omega_\eta'$, $xa_{-s_i}\in \Omega_\eta$ for an infinite sequence
 $s_i\to +\infty$.

Since $m^{\BMS}(\Omega'_\eta)>1-4\eta$, if $\eta>0$ is small enough,
we have $x\in \Omega'_\eta$ such that 
the measure of
$L_x:=\{am\in AM: xg\in \Omega_\eta', d(e, g)<\e' \}$ is at least half
of the measure of the set $\{g\in AM: d(e, g)<\e'\}$ where the measure is taken with respect to the 
Haar measure of $AM$.

 Let $g_0\in L_{x}$
be a Lebesgue density point of $L_{x}.$ Replacing $x$ with $xg_0,$
 we may assume that $e$
is a density point of $L_{x}.$ 

In the following, we fix
 $g \in L_x$. 
 
As $g$ normalizes $U$,
we have $$\mfib(xg)g^{-1} u_{\tbf}= \mfib(xu_{\tbf}g) g^{-1}.$$
 Hence
 if $xu_\tbf, xu_\tbf g \in K_\eta$, then  for each $i$,
 $$
d( \msec_i(xg)g^{-1} u_{\tbf} ,\Upsilon (x^1 u_\tbf ))\leq  2\e
$$
by the continuity of $\mfib$ in $K_\eta$.

Since $x,x g\in\Om_\eta'$,
 we get from \eqref{ruu} that if $T>T_\eta$,
 
\be\label{c331} \mu^{\PS}_{x}\{\tbf\in B_U(T): xu_{\tbf} \in\kcal_\eta\} \geq (1-2\eta)\mu_{x}^{\PS}(B_U(T))
\ee
and  
\be\label{c332} \mu^{\PS}_{xg}\{\tbf\in B_U(T): xgu_{\tbf} \in\kcal_\eta\} \geq (1-2\eta)\mu_{xg}^{\PS}(B_U(T))
\ee

Since $d\mu_{x}^{\PS}$ and $d\mu_{xg}^{\PS}$ are absolutely continuous with each other as $g\in AM$
and the Radon-Nikodym derivative is given by  $|\text{Jac}(g)|^{\delta}$ satisfying
$ $$1-\tfrac{\e}{2}\le |\text{Jac}(g)| \le 1+\tfrac{\e}{2}$, we have 
\begin{multline*}\mu_x^{\PS}\{\tbf\in B_U(T): xu_{\tbf}g \notin\kcal_\eta\}\le \mu_{xg}^{\PS}
\{\tbf\in B_U((1+\e) T): xg u_{\tbf} \notin\kcal_\eta\} 
\\ \le ( c_0 \eta  )\mu_x^{\PS}(B_U(T))\end{multline*}
for some absolute constant $c_0>1$.
Consequently, we have \be
\mu^{\PS}_{x}\{\tbf\in B_U(T): xu_{\tbf} , xu_{\tbf} g \in\kcal_\eta\} \geq (1-c_1\eta)\mu_{x}^{\PS}(B_U(T)).
\ee
for some $c_1>1$.

Fixing $i$, put
$$
\Theta(\tbf)=\min\{\dist(\msec_i(xu_\tbf),\Upsilon(xg)g^{-1}u_\tbf)^2,1\} .
$$

Then
for any  $T>T_\eta$, 
$$\frac{1}{\mu_{x}^{\PS}(B_U(T))}\int_{B_U(T)}\Theta(\tbf)d\mu^{\PS}_{x}(\tbf)
\le 2\e +4\eta .$$
Let $s_i\to \infty$ be a sequence tending to $+\infty$ such that
$xa_{-s_i}\in \Omega_\eta$. Then
 Lemma~\ref{l;poly-ps-good} implies that
\[
\sup_{|\tbf|\leq e^{s_i}} \Theta(\tbf)\leq 1/2
\] 
if $\eta$ and $\e$ are sufficiently small.

It follows that 
\be \label{haha} \sup_{\tbf\in U}\dist(\msec_i(xu_\tbf),\Upsilon(xg)g^{-1}u_\tbf)\le 1.\ee

We claim that \be\label{posi}  \mfib(xg)g^{-1}\subset \mfib(x)U.\ee
Suppose not; then there exists $i$ such that for all $j$,
$\msec_i(xg)g^{-1} =\msec_j (x)g_{m,j}$ for $g_{m,j}\notin U$.
On the other hand, \eqref{haha} implies that
$u_{-\tbf} g_{m,j}u_{\tbf}$ is uniformly bounded for all $\tbf$. It implies that
$g_{m,j}$ belongs to the centralizer of $U$, which is $U$ itself. This yields a contradiction, proving the claim.
Note that we have shown \eqref{posi} for  any $g\in L_x$.
Since $e$ is a density point of $L_x$,
it generates $AM$ by the lemma below, and finishes the proof. 
 
\end{proof}

\begin{lem} 
Let $H$ be a connected Lie group. If $W\subset H$ is a Borel subset
such that $e\in W$ is a Lebesgue density point, then $W$ generates $H$.
\end{lem}

\begin{proof} As $e\in W$ is a Lebesgue density point, $\mu_H(W)>0$ where $\mu_H$ is a Haar measure.
The convolution $$f(g):=\chi_{W}\star \chi_{W^{-1}}(g)=\int_{H} \chi_W( hg)\chi_{W^{-1}}(h^{-1}) d\mu_H(h) $$ is a continuous function
and $f(e)=\mu_H(W)>0$. Therefore there exists a neighborhood $\mathcal O$ of $e$ in $H$ on which $f$ never vanishes.
This means that $\mathcal O\subset W^{-1}W$.
Since any neighborhood of $e$ in $H$ generates $H$, the claim follows. \end{proof}

Let us recall that the $\BMS$ measure and the $\BR$-measure on $X_1$ have product structures, and
have the ``same transversal measures''. To be more precise,
let $\psi\in C_c(X)$ and further assume that $\supp(\psi)\subset yP_\e U_\e$ with $P=MA{\cont}$. Then
\be\label{e;bms-br-recall}
\begin{array}{l} m^{\BR}(\psi)=\int\int \psi(ypu)d\nu(yp)d\mu_{yp}^{\Leb}(u),\;
 \text{and}\\
 m^{\BMS}(\psi)=\int\int \psi(ypu)d\nu(yp)d\mu_{yp}^{\PS}(u)
 \end{array}
\ee
where $d\nu$ is the transversal measure on $P$. 

It follows that
\begin{lem}\label{xp}
If $Q\subset X_1$ is $U$-invariant and $m^{\BR}(Q)=1$, then
$Q$ has full $\BMS$-measure.
\end{lem}

\begin{prop}\label{l;M-inv}
The set-valued map $\mfib$ is $AMU$-equivariant; there exists an $AMU$-invariant $\BR$ co-null subset
$X''\subset X_1$ such that for all $x^1\in X''$ and  every  $am u_{\tbf} \in AMU$, 
we have \be \label{aminv} \mfib(x^1amu_{\tbf})=\mfib(x^1)amu_{\tbf}.\ee
\end{prop}

\begin{proof} 
It suffices to show that $\mu$ is $\Delta(AM)$-quasi-invariant. 
Let $Y\subset \XX $ be a $\Delta(U)$-invariant $\mu$-conull subset such that for all $y\in Y$
we have
\begin{enumerate}
\item 
$\lim_{T\to \infty} \int_{B_U(T)}\Psi(y\biggere)d\tbf =\infty$;
\item for all $f\in C_c(\XX)$,
$$
\lim_{T\to \infty}  \frac{\int_{B_U(T)}f(y\biggere)d\tbf}
{\int_{B_U(T)}\Psi(y\biggere)d\tbf}= \frac{\mu(f)}{\mu(\Psi)}.$$
\end{enumerate}


Note that by the definition of $\mfib$,
the support of $\mu$ is $ \{(x^1, x^2): x^1\in X_1, x^2 \in \Upsilon(x^1)\}\}$, and hence
$\{(x^1, x^2): x^1\in X', x^2 \in \Upsilon(x^1)\}$ has full $\mu$-measure.
Therefore,
replacing $X'$ with a conull set, we may assume that $X'$ is a $U$-invariant subset with full $\BR$-measure
satisfying \eqref{Ueq},
$X'\subset \{x^-\in\Lambda_{\rm r}(\G_1)\},$ and that
\be \label{xxy} \{(x^1, x^2): x^1\in X', x^2 \in \Upsilon(x^1)\}\subset Y.\ee


By Lemma \ref{xp}, we have
$m^{\BMS}(X')=1$. 
Let $x^1\in X'$ and $L_{x^1} \subset AM$ be as in Lemma \ref{sleep}, so that
 for any $g\in L_{x^1}$, we have
 \be \msec_1(x^1g)= \msec_j(x^1)u_g'g\ee
  for some $j$ and $u_g'\in U$. 
 
  As $u_g'g\in UAM$, we can write $u_g'g=u_ggu_g^{-1}$ for some $u_g\in U$.
   Since
  $X'$ is $U$-invariant and $\mfib$ is $U$-equivariant on $X'$, we have
  $$  (x^1g, \msec_j (x^1) u_ggu_g^{-1})= (x^1 g, \msec_1(x^1 g ))\in Y.$$
  Therefore, $$Y\cap Y (g, u_ggu_g^{-1})\ne \emptyset .$$
  Hence 
by Lemma \ref{l;normalizer-inv},
$\mu$
is quasi-invariant under the closed subgroup $R$ generated by
$(g,u_ggu_g^{-1})$. Since $L_{x^1}$ generates $AM$, this will finish the proof
if we show 
\[
\mbox{$u_g=e$} \quad\text{for each $g\in L_{x^1}$.  }
\]

Let $A'\subset \Delta(AM)$ be a connected subgroup  as  in the remark following Corollary \ref{alunch}.
Suppose $u_g\ne e$. 
Consider the commutator of elements of $A'$ and $(g,u_ggu_g^{-1})$.
Note that the first component belongs to the commutator subgroup $[AM,AM]=\{e\}$. 
For any non-trivial $b\in A'$,
set 
\begin{align*} 
(e, v):&=
(b,b)(g,u_ggu_g^{-1})(b^{-1},b^{-1})(g^{-1},u_gg^{-1}u_g^{-1})\\
&=(e,bu_ggu_g^{-1}b^{-1} u_gg^{-1}u_g^{-1}).
\end{align*}
We make the following observations:
\begin{itemize}
\item $U$ is the commutator subgroup of $AMU$ and hence $v\in U.$
\item $v\neq e$ since $u_g\ne e$. 
\item $v\ne e$ can be made
arbitrarily close to $e$ by choosing these parameters 
close but not equal to $e$.
\end{itemize}

Recall now that $A'$ and $R$ leave $\mu$ quasi-invariant. Therefore any $(e,v)$ as above leaves $\mu$
quasi-invariant and thus invariant by Lemma~\ref{lem:one-factor-inv}. 
Hence we get that if $u_g\neq e,$ then there is a sequence
$v_i\to e$ in $ U\setminus\{e\}$, so that $(e,v_i)$ leaves the measure invariant
for all $i\geq 1.$
Therefore, there is a non-trivial connected subgroup $V< U$
so that $\mu$ is invariant under $\{e\}\times V.$ 
This contradicts Lemma~\ref{l;not-product} and finishes the proof.
\end{proof}

\begin{prop}\label{p;factor-rigidity}
Let $\mfib: X_1\to X_2$ be as above. In particular, it satisfies
that $\mfib( x g)=\mfib(x) g$ for all $x\in X'$ and all $g\in MAU$.
 Then there exists $q_0\in G$ such that $\G_1\cap q_0^{-1} \G_2 q_0$ has finite index in $\G_1$ 
so that if  we put
$\G_2 q_0\G_1 =\cup_{1\leq j\leq\ell } \Gamma_2 q_0\gamma_{j}$ for $\gamma_{j}\in \G_1$, then
\[
\mfib(\G_1 g)=\{\G_2 q_0\gamma_{j} g: 1\le j\le \ell \}
\] 
on a $\BR$ conull subset of $X_1.$ Moreover $\mu$ is a $\Delta(U)$-invariant measure
supported on $\{(x^1, \Upsilon (x^1)): x^1\in X_1\}$ and hence a finite cover self-joining (see Def. \eqref{sj}).
\end{prop}

\proof By Theorem \ref{t;factor-rigidity-bms} and \eqref{xp}, 
we have $\mfib( x g)=\mfib(x) g$ for all $x\in X'$ and all 
$g\in {\cont}MAU$ where $X'$ is given by loc. cit.
Fixing $x_0\in X'$, 
define $$\mfib' (x_0 g):=\mfib (x_0) g$$ for all $g\in G$.
Then
$\mfib$ and $\mfib'$ coincide with each other on $Y':=x_0({\cont}MAU)$.
Since ${\cont}MAU$ is a Zariski open subset of $G$, $m^{\BMS}(Y')=1$, and hence $\mfib=\mfib'$ 
on a $\BMS$-conull subset.
Hence we may assume without loss of generality 
\be\label{uu} \mfib (xg)=\mfib (x) g\ee for all $g\in G$ and $x\in Y'$ with
$xg\in Y'$.
Let $\G_1 g_0\in Y'$ and write $\mfib(\G_1 g_0)=\{\G_2 h_1,\ldots,\G_2 h_\ell\}.$ Then
 for every $g\in G$ such that $\G_1 g_0g\in Y'$
we have
$$\mfib(\G_1 g_0g)=\{\G_2 h_1g,\ldots,\G_2 h_\ell g\}.$$ 
Note that for all $\gamma\in\G_1$ we have
$\G_1g_0 (g_0^{-1}\gamma g_0)=\G_1 g_0\in Y'.$ 
Therefore, applying the $G$-equivariance \eqref{uu} of $\Upsilon$ to
$g_0^{-1}\gamma g_0\in g_0^{-1}\G_1 g_0$, we get the set
$\{\G_2 h_1g,\ldots,\G_2 h_\ell g\}$ is right invariant under $g_0^{-1}\G_1 g_0$. 
It follows that $\G_2\ba \G_2 h_i (g_0^{-1}\G_1 g_0)$ is finite for each $i$. 
Putting $q_i:=h_ig_0^{-1}$, we have $\G_2\ba \G_2 q_i\G_1 $ is finite. Let $q_1, \cdots, q_r$
be such that the corresponding cosets $\G_2\ba \G_2 q_i\G_1 $ are distinct and
$\cup_{1\le i\le r}\G_2\ba \G_2 q_i\G_1 =\cup_{1\le i\le \ell} \G_2\ba \G_2 q_i\G_1 $.
Thus, if for each $1\leq i\leq r$ we put
$\G_2 q_i\G_1 =\cup_{1\leq j\leq\ell_i} \Gamma_2 q_i\gamma_{ij}$ for $\gamma_{ij}\in \G_1,$ then
\be\label{aff} 
\mfib(\G_1 g)=\{\G_2 g_1\gamma_{11} g,\ldots,\G_2 g_1\gamma_{1\ell_1} g,\ldots,\G_1 g_r\gamma_{r1}g,\ldots,\G_2 g_r\gamma_{r\ell_r}g\}
\ee on a $\BMS$ conull subset of $X_1.$

In particular we get $q_1^{-1}\G_1q_1$ is commensurable 
with a subgroup of $\G_2.$
Repeating the argument with $\G_2$ we get,
up to a conjugation, $\G_1$ and $\G_2$ are commensurable with each other. 
Note that in view of~\eqref{e;bms-br-recall}, the $U$-invariant set $X''=Y'\cap X'U $
has full $\BR$ measure. Let now $g\in G$ be so that $\G_1 g\in X''$. Then we can write 
$g= g'u$ where $\G_1 g'\in X'$ and $u\in U.$ 
Now property \eqref{uinvv} of $\Upsilon$ and~\eqref{aff} imply
\begin{align*}&
\mfib(\G_1 g)=\mfib(\G_1 g') u\\&=
\{\G_2 g_1\gamma_{11} g'u,\ldots,\G_2 g_1\gamma_{1\ell_1} g'u,\ldots,\G_2 g_r\gamma_{r1}g'u,\ldots,\G_2 g_r\gamma_{r\ell_r}g'u\}\\
&=\{\G_2 g_1\gamma_{11} g,\ldots,\G_2 g_1\gamma_{1\ell_1} g,\ldots,\G_2 g_r\gamma_{r1}g,\ldots,\G_2 g_r\gamma_{r\ell_r}g\}.
\end{align*} 
Now for every $1\leq i\leq r$
we have that the set 
\be\label{fin}
\{(x^1,x^2): x^1=\G_1 g\in X'', x^2\in\{\G_2 g_i\gamma_{i1} g,\ldots,\G_2 g_i\gamma_{i\ell_1} g\}\}
\ee
is $\Delta(U)$-invariant and has positive $\mu$ measure.
Therefore, by the ergodicity of $\mu$, we get $r=1$ and $\mu$ is a $\Delta(U)$-invariant measure
supported on the set \eqref{fin}. This implies that $\mu$ is a finite self-joining as defined in the introduction.
This finishes the proof.
\qed 

\begin{remark}\label{exj2}
Note that for a general discrete non-elementary subgroup $\G$,
the Patterson-Sullivan density always exists, although it may not be a unique $\G$ conformal density of dimension $\delta_\G$, and
hence the $\BR$-measure is well-defined on $\G\ba G$. Given this,
the above proof yields a stronger version of Corollary \ref{exj} in the introduction where $\G_2$ is not assumed to be geometrically finite. Namely, we have:
 if $\G_1$ is geometrically finite and Zariski dense and $\G_2$ is a Zariski dense (not necessarily geometrically finite) subgroup
 of $G$ with infinite co-volume,
  then a $U$-joining on $\G_1\ba G\times \G_2\ba G$ with respect to the pair $(m^{\BR}_{\G_1}, m^{\BR}_{\G_2})$ exists
only when $\G_1$ is commensurable with a subgroup of $\G_2$, up to conjugation.
\end{remark}

\subsection{}\label{factors}
In this section we deduce Theorem~\ref{thm:factor} from Theorem \ref{thm;joining-class}.
 When $\G$ is a lattice, this deduction  is due to \cite{Rat-Joining} and \cite{Wit} (see also ~\cite[Sec.\ 6]{EL-Joining}).
In the following, we assume that $\G$ has infinite co-volume. Although the basic strategy is similar to
the finite volume case, we will need a certain property of geometrically finite groups of infinite co-volume in the proof.


Define the following closed subgroup of $G$:
$$H_\G:=\{h\in  G: h.{\tilde m}^{\BR}={\tilde m}^{\BR}\}$$
where $h.{\tilde m}^{\BR}(B)={\tilde m}^{\BR}( Bh)$ for any Borel subset $B$ of $G$.
Clearly, $\G$ is contained in $H_\Gamma$.
\begin{lem}\label{lem:H-Gamma}
The following hold:
\begin{itemize}
\item $\G$ is a finite index subgroup of $H_\G.$
\item $H_\G=\{g\in G: g\Lambda(\G)=\Lambda(\G)\}.$
\end{itemize}
\end{lem}

\begin{proof}
Set $F:=\{g\in G: g\Lambda(\G)=\Lambda(\G)\}.$
Since the support of $\tilde m^{\BR}$ is given by $\{g\in G: g^-= g(w_o^-)\in \Lambda (\Gamma)\}$,
the inclusion $H_\G\subset F$ follows.

Note that $F$ is a discrete subgroups of $G$.
Indeed the identity component $F^\circ$ 
is normalized by $\G.$ Since $G$ is simple 
and $\G$ is Zariski dense in $G$, it follows that $F^\circ$ is either $G$ or trivial.
The former however contradicts the fact that $\G$ is not a lattice.

Hence the isometric action of $F$ on the convex hull $\mathcal C(\G)$ of $\Lambda(\G)$  is properly discontinuous.
As $\G$ is geometrically finite, the
 orbifold $\G\ba \mathcal C(\G)$ has finite volume. 
 Now, $F\ba \mathcal C(\G)$ is an orbifold 
which is covered by $\G\ba \mathcal C(\G)$ and hence
$F\ba \mathcal C(\G)$ has finite volume as well, so  $\G$ is of finite index in $F$. This implies the first claim.

As $[F:\G]<\infty$ and $\G$ is geometrically finite, their Patterson-Sullivan densities are equal up to a constant multiple
and hence the corresponding $\BR$-measures are proportional to each other.
This implies the second claim. \end{proof}

We denote by ${\rm Comm}_G(\G)$ the commensurator subgroup of $\G$, that is,
 $g\in {\rm Comm}_G(\G)$ if and only if $g\Gamma g^{-1}$ and $\Gamma$ are commensurable with each other.

\begin{cor}\label{cor:comm}
We have:
${\rm Comm}_G(\G)\subset H_\G.$
\end{cor}

\begin{proof} This follows from Lemma~\ref{lem:H-Gamma} since
\[
\Lambda(\G)=\Lambda(\G\cap g\G g^{-1})=\Lambda(g\G g^{-1}).
\]
for any $g\in {\rm Comm}_G(\G)$.
\end{proof}

Let $p$ be the factor map as in Theorem \ref{thm:factor}, and $\tilde p$ denote the lift of $p$ to $G.$   
Then $\tilde p$ is a left $\G$-invariant and right $U$-equivariant map from $G$ to $Y.$
In view of Lemma~\ref{lem:H-Gamma} and 
replacing $\Gamma$ by a bigger subgroup of $H_\G$, if necessary, we may assume 
\be\label{eq:Gamma-pi}
\G=\{h\in H_\G: \tilde p(hg)=\tilde p(g)\;\text{for ${\tilde m}^{\BR}$-a.e.\ $g\in G$}\}.
\ee

Define  
\[
Q:=\{(h,u)\in H_\G\times U: \tilde p(hgu)=\tilde p(g)\; \text{for ${\tilde m}^{\BR}$-a.e.\ $g\in G$}\}.
\]
Since ${\tilde m}^{\BR}$ is left $H_\G$-invariant and right $U$-invariant we get that 
$Q$ is a closed subgroup of $H_\G\times U.$  The subgroup $Q$ acts on $G$ by 
$$q(g)= h gu$$
for $q=(h,u)\in Q$ and $g\in G$.

\begin{lem}\label{lem:Q-norm}
The subgroup $Q$ is contained in $ N_G(\G)\times U.$ 
\end{lem}

\begin{proof}
If $(h,u)\in Q$, then for any $\gamma\in\G$ and ${\tilde m}^{\BR}$-a.e.\ $g\in G$, we have
\[
\tilde p(\gamma hgu)=\tilde p(hgu)=\tilde p(g)=\tilde p(\gamma g)=\tilde p(h\gamma g u).
\]
Since ${\tilde m}^{\BR}$ is right $U$-invariant, the above implies 
\[
\tilde p(\gamma hg)=\tilde p(h\gamma g)\quad\text{ for ${\tilde m}^{\BR}$-a.e.\ $g\in G$}.
\]
By the definition of $H_\G$, we deduce
\[
\tilde p(g)=\tilde p(h\gamma h^{-1} g)\quad\text{ for ${\tilde m}^{\BR}$-a.e.\ $g\in G$},
\]
that is, $h\gamma h^{-1}\in \G$ by \eqref{eq:Gamma-pi}. So $h\in N_G(\G).$
\end{proof}

\medskip

{\bf Proof of Theorem~\ref{thm:factor}.}
Set $X=\G\ba G$.
Given the factor map ${ p}: (X,m^{\BR})\to (Y,\nu)$ we can disintegrate $m^{\BR}$ into
$m^{\BR}=\int_Y\tau_y\;d\nu(y)$ where $\tau_y(p^{-1}(y))=1$ for $\nu$-a.e.~$y$. 
Recall, e.g.\ from~\cite{Furman},
that we can construct the independent self-joining relative to $p$ by
\be\label{eq:factor-joining}
\bar\mu:=\int_Y\tau_y\times_p\tau_y\;d\nu(y).
\ee
Then $\bar\mu$ is a $U$-joining on $Z:=X\times X$ supported on $X\times_p X:=\{(x_1,x_2):p(x_1)=p(x_2)\}.$
Disintegrate $\bar\mu$ into $\Delta(U)$-ergodic components:
\be \label{bbu} \bar\mu=\int_Z\mu_z\;d\sigma(z).\ee By Theorem~\ref{thm;joining-class},
 we have $\mu_z$ is a finite cover self joining for $\sigma$-a.e.~$z$.
That is, $\mu_z$ is the image of $m^{\BR}$ on a closed orbit $[(e,g_z)]\Delta(G)(e,u_z)$
where $g_z\in{\rm Comm}_G(\Gamma)$ and $u_z\in U.$ 
Note that $(g_z,u_z)\in H_\G\times U$  for $\sigma$-a.e.\ $z$ by Corollary~\ref{cor:comm}.

We now compare the two descriptions ~\eqref{eq:factor-joining}  and \eqref{bbu} of $\bar\mu$
and get the following consequences. Since \[
\tilde p(g_zgu_z)=\tilde p(g)\quad \text{ for $\sigma$-a.e.\ $z$ and ${\tilde m}^{\BR}$-a.e.\ $g\in G$},
\]
we have $(g_z,u_z)\in Q.$

Recalling from Lemma~\ref{lem:Q-norm} that $\G\times\{e\}$ is normal in $Q,$
consider the following subgroup  $$L:=\G\times\{e\}\ba Q.$$
 Then the above discussion implies that in \eqref{bbu}, we can consider $\sigma$ as a probability measure
 on $L$. Hence we may write $$ \bar\mu=\int_L\mu_\ell\; d\sigma(\ell)$$
where, for $\sigma$-a.e.\ $\ell$, $\mu_\ell$ is the image of the $\BR$-measure  on 
\[
(\G\times\G)\ba\{(g,\ell(g)): g\in G\},
\]
up to a constant multiple.

 Following the proof of~\cite[Thm.\ 3.9(iii)]{Furman}, we now claim that the convolution $\sigma \star \sigma$ is equal to $\sigma$.
Comparing with~\eqref{eq:factor-joining}, we get that for $\nu$-a.e.\ $y$, the measure $\tau_y$ is supported on a single $L$ orbit.
Furthermore, for $\tau_y$-a.e.\ $x\in X$ we have
\be\label{eq:L-orbit}
\int_Lf(\ell(x))d\sigma(\ell)=\int_Xf(x')d\tau_y(x')\quad\text{for all $f\in C_c(X)$}.
\ee
Now using~\eqref{eq:L-orbit}, for all $f\in C_c(X)$ and $\tau_y$-a.e.\ $x$, we have
\begin{multline*}
\sigma \star \sigma(f)=\int_L\int_Lf(\ell_2\ell_1(x))d\sigma(\ell_1)d\sigma(\ell_2)
=\int_L\int_Xf(\ell_2(x'))d\tau_y(x')d\sigma(\ell_2)\\
=\int_X\int_Lf(\ell_2(x'))d\sigma(\ell_2)d\tau_y(x')=\int_X\int_Xf(x'')d\tau_y(x'')d\tau_y \\ 
=\int_Xf(x'')d\tau_y(x'')=\int_Lf(\ell(x))d\sigma(\ell)=\sigma(f).\end{multline*}
This proves the claim.

It follows that $\sigma$ is the probability Haar measure 
of a compact subgroup $R$ of $L$ by ~\cite[Lemma 3.11, Rk. 3.12]{Furman}. 
 Since $[H_\G:\G]<\infty$ by Lemma~\ref{lem:H-Gamma},
it follows that  $R$ is a finite subgroup of $\Gamma\ba N_G(\Gamma)\times \{e\}\subset \Gamma\ba Q$.

All together we get $\tau_y$ is the push forward of the Haar measure $\sigma$
to an orbit of $R$ in $X$ for $\nu$-a.e.\ $y\in Y.$
That gives $(Y,\nu)$ is isomorphic to $R\ba (\G\ba G,m^{\BR})$.
This implies ~\ref{thm:factor}.
\qed

\end{document}